%% file: master.tex
\documentclass[oneside,a4paper,english, 11pt]{amsart}

\input{artmymacros.sty}
\input{HLmacros.tex}

\title{Serre weight conjectures for $\GSp_4$}

\author{Daniel Le}
\address{Department of Mathematics,
Purdue University,
150 N. University Street, 
West Lafayette, IN 47907-2067}
\email{ledt@purdue.edu}

\author{Bao V.~Le Hung}
\address{Department of Mathematics,
Northwestern University, 
2033 Sheridan Road\\
Evanston, IL 60208, USA}
\email{lhvietbao@googlemail.com}

\author{Heejong Lee}
\address{Korea Institute for Advanced Study, 85 Hoegi-ro, Dongdaemun-gu, Seoul 02455, Republic of Korea}
\email{heejonglee@kias.re.kr}

\begin{document}

\begin{abstract}
We prove the weight part of Serre's conjecture for Galois representations valued in $\GSp_4$ that are tamely ramified with explicit genericity at places above $p$ as conjectured by Herzig--Tilouine and Gee--Herzig--Savitt. 
 This improves on \cite{lee_thesis} where an inexplicit genericity hypothesis was required. 
 As an application, we prove a modularity lifting theorem for $\GSp_4$ under similar assumptions. 

\end{abstract} 

\maketitle

\setcounter{tocdepth}{2}
\tableofcontents

\section{Introduction}

\input{notation}

\input{representation_theory}

\input{local_model}

\input{deformation_rings}


\input{main_result}

\bibliography{Biblio}
\bibliographystyle{amsalpha}

\end{document}

%% file: HLmacros.tex
\newcommand{\tilW}{\tld{W}}
\newcommand{\uW}{\underline{W}}
\newcommand{\utilW}{\underline{\tld{W}}}
\newcommand{\uOm}{\underline{\Omega}}
\newcommand{\uT}{\underline{T}}
\newcommand{\tilw}{\tld{w}} 
\newcommand{\tilz}{\tld{z}} 
\newcommand{\tilU}{\tld{U}}
\newcommand{\bfa}{\mathbf{a}}
\DeclareMathOperator{\AP}{\mathrm{AP}}

\newcommand{\loccit}{\emph{loc.~cit.}}

\newcommand{\mo}{{-1}} 
\newcommand{\reg}{\mathrm{reg}}
\renewcommand{\pmat}[1]{{\begin{pmatrix}#1\end{pmatrix}}}
\newcommand{\psmat}[1]{{\begin{psmallmatrix}#1\end{psmallmatrix}}}

\newcommand{\Fpbar}{\overline{\mathbf{F}}_p}
\newcommand{\nv}{\mathrm{nv}}
\newcommand{\nbl}{\nabla}

\newcommand{\sig}{\sigma}
\newcommand{\nm}{\mathrm{nm}}

\newcommand{\osig}{\overline{\sigma}}
\newcommand{\ix}[1]{^{(#1)}}

\newcommand{\G}{\mathbb{G}}
\newcommand{\pcp}{\wedge_p}
\newcommand{\uC}{\underline{C}}
\newcommand{\RG}[1]{\langle #1 \rangle}
\newcommand{\Irr}{\mathrm{Irr}}
\newcommand{\CNL}{\widehat{\cC}}
\newcommand{\simc}{\mathrm{sim}}
\newcommand{\alg}{\mathrm{alg}}
\newcommand{\psip}{{\psi_p}}

\newcommand{\LuG}{\prescript{L}{}{\underline{G}}}

\newcommand{\ctimes}{\hat{\otimes}}

\newcommand{\ov}{\overline}
\newcommand{\ud}{\underline}
\renewcommand{\hat}{\widehat}

\newcommand{\Art}{\mathrm{Art}}
\newcommand{\out}{\mathrm{out}}

\newcommand{\CB}[1]{\left\{#1\right\}}

\newcommand{\DP}[1]{(\!(#1)\!)}
\newcommand{\DB}[1]{\llbracket#1\rrbracket}
\newcommand{\lam}{\lambda}
\newcommand{\recGT}{\mathrm{rec}_{\mathrm{GT}}}
\newcommand{\PG}{(\varphi,\Gamma)}

\newcommand{\tadstar}{*_{(s,\mu)}^{\varphi}}
\newcommand{\tadslash}{/_{(s,\mu)}^{\varphi}}
\newcommand{\std}{\mathrm{std}}
\newcommand{\coh}{\mathrm{coh}}

\newcommand{\tor}{\mathrm{tor}}
\newcommand{\sub}{\mathrm{sub}}
\newcommand{\can}{\mathrm{can}}

\newcommand{\cind}{\mathrm{c\text{-}Ind}}
\newcommand{\cyc}{\mathrm{cyc}}
\newcommand{\Iw}{\mathrm{Iw}}
\newcommand{\Hk}{\mathrm{Hk}}
\newcommand{\cont}{\mathrm{cont}}

%% file: notation.tex
\subsection{The main result}
The article \cite{HT} formulated a generalization of the weight part of Serre's conjecture for the \'etale cohomology of the Siegel modular threefold for tamely ramified Galois representations. 
\cite{GHS} later generalized this to unramified Hilbert--Siegel modular varieties (though the global context was not made explicit in this setting). 
The main goal of this article is to prove these conjectures for tamely ramified Galois representations with explicit genericity (under mild hypotheses typically appearing in the Taylor--Wiles method). 

Let $p$ be a prime number and $F$ be a totally real field in which $p$ is unramified. We denote by $G_F$ the absolute Galois group of $F$.
Let $U = U_p U^{\infty,p} \le \GSp_4(\cO_F\otimes_{\Z}\Zp)\times \GSp_4(\A_F^{\infty,p})$ be a compact open subgroup. We assume that $U^{p,\infty}$ is neat. Let $\Sh_U$ be the Hilbert--Siegel modular variety associated with $\Res_{F/\Q}\GSp_{4/F}$ at the level $U$ defined over $F$. For a $U_p$-representation $V$, we denote by $\cF_V$ the locally constant \'etale sheaf on $\Sh_U$ associated with $V$. We consider the \'etale cohomology group $H^i_{\et}(\Sh_{U,\overline{F}},\cF_V)$. Then there is a Hecke algebra $\T$ acting on $H^i_{\et}(\Sh_{U,\overline{F}},\cF_V)$. Moreover, they are isomorphic to $(\mathfrak{g},K)$-cohomology groups of a certain space of cuspidal automorphic forms on $\GSp_4(\A_F)$ as Hecke modules by Matsushima's formula. We choose an isomorphism $\iota: \Qpbar \simeq \C$. By various works including \cite{weissauer, sorensen-gsp4, gee-taibi}, there is a Galois representation $r_{\pi,p,\iota}: G_F \ra \GSp_4(\overline{\Z}_p)$ attached to a regular algebraic cuspidal automorphic representation $\pi$ of $\GSp_4(\A_F)$. We say that a continuous semisimple representation $\rbar: G_F \ra \GSp_4(\Fpbar)$ is \textit{automorphic} if $\rbar$ is isomorphic to the semisimplification of $r_{\pi,p,\iota}\otimes_{\ov{\Z}_p}\Fpbar$ for some $\pi$. 
Let $\fm_{\rbar}\subset \T$ be the maximal ideal determined by $\rbar$ by matching Satake parameters with the characteristic polynomial of the Frobenius element at all but finitely many places of $F$. Under the standard Taylor--Wiles conditions, $\rbar$ is automorphic if and only if $H^i_{\et}(\Sh_{U,\overline{F}},\Fpbar)_{\fm_{\rbar}} \neq 0$ for some level $U$ and $i\in \Z_{\ge 0}$. 


Now let $U_p \cong \GSp_4(\cO_F \otimes_{\Z} \Z_p)$. 
A Serre weight is defined to be a simple smooth $\Fpbar[U_p]$-module. 
Every Serre weight is obtained via inflation from an irreducible representation of $\GSp_4(\cO_F \otimes_{\Z} \F_p)$ over $\Fpbar$. 
If $\rbar$ is automorphic, let $W(\rbar)$ be the set of Serre weights $\sigma$ such that $H^i_{\et}(\Sh_{U,\overline{F}},\cF_{\sig^\vee})_{\mathfrak{m}_{\rbar}} \neq 0$ for some $i$. 
If $\rbar$ is tame and generic at places above $p$, Gee--Herzig--Savitt defined an explicit set of Serre weights $W^?(\rbar)$ \cite[Definition 9.2.5]{GHS} to formulate a generalization of the weight part of Serre's conjecture. We confirm their conjecture under technical assumptions.

\begin{thm}[Theorem \ref{thm:SWC}]\label{thm:main}
    Suppose that $\rbar$ is automorphic and satisfies the Taylor--Wiles conditions. If $\rbar$ is tamely ramified and $9$-generic 
    at all places above $p$, then $W(\rbar) = W^?(\rbar)$. 
\end{thm}

\begin{rmk} We explain the difference between Theorem \ref{thm:main} and \cite[Theorem 6.2.5]{lee_thesis}. 
    \begin{enumerate}
        \item The genericity assumption on $\rbar$ at places above $p$ in \cite{lee_thesis} is much stronger than the current work. In principle, the assumption in \loccit~can be computed explicitly. However, it remains inexplicit at this moment due to the computational complexity. In contrast, the genericity assumption in this paper is explicit. In particular, for any $p \ge 41$, one can construct an example of $\rbar$ satisfying our assumptions using \cite[Lemma 4.4.5]{EL}.

         \item \cite{lee_thesis} works with algebraic automorphic forms on an inner form of $\GSp_4$ that is compact modulo center at infinity. In the current paper, we can work with $\GSp_4$ directly, thanks to the torsion vanishing result in \cite{Hamann-Lee}. Our argument can be applied to the setting of \cite{lee_thesis} as well.
    \end{enumerate}
\end{rmk}

\subsection{Previous methods} 
In recent years, there has been substantial progress on generalizations of the weight part of Serre's conjecture in terms of both conjectures and results. 
A major tool in these advancements is the Taylor--Wiles method which patches global spaces of automorphic forms into objects of conjecturally local origin. 

We first set up some notation. For simplicity, we assume in this introduction that $p$ is inert in $F$. We denote by $K=F_p$ with ring of integers $\cO_K$ and residue field $k$. 
Let $E/\Qp$ be a sufficiently large finite extension with ring of integers $\cO$, uniformizer $\varpi$, and residue field $\F$. Then $\rbar$ is valued in $\GSp_4(\F)$, and we let $\rhobar= \rbar|_{G_{F_p}}$. 
We consider a Hodge type $\eta$ lifting the half sum of all positive roots and a generic tame inertial type $\tau:I_K \ra \GSp_4(E)$. 
The inertial local Langlands correspondence attaches to $\tau$ an irreducible generic Deligne--Lusztig representation $\sig(\tau)$ of $\GSp_4(k)$ defined over $E$ (see \S\ref{subsub:ILL}). 
Let $\sig^\circ(\tau)$ be a $\GSp_4(k)$-stable $\cO$-lattice in $\sig(\tau)$ and $\osig(\tau)=\sig^\circ(\tau)/\varpi$. 
We denote by $\JH(\osig(\tau))$ the set of Jordan--H\"older factors of $\osig(\tau)$. 
We denote by $R^{\square}_{\rhobar}$ the framed deformation ring for $\rhobar$ and by $R_{\rhobar}^{\eta,\tau}$ the potentially crystalline deformation ring for $\rhobar$ of type $(\eta,\tau)$. We use $\sig$ to denote a Serre weight. There is a natural bijection between $\JH(\osig(\tau))\cap W^?(\rbar)$ and the set of irreducible components in $\Spec R_{\rhobar}^{\eta,\tau}/\varpi$, which we denote by $\sig \mapsto \cC_{\sig}(\rhobar)$.

The Taylor--Wiles method produces a functor $M_\infty$ from the category of $\cO[\GSp_4(k)]$-modules to modules over (a power series ring over) $R_{\rhobar}^{\square}$. 
The main properties of $M_\infty$ can be summarized as follows:
\begin{enumerate}
    \item \label{item:exact} $M_\infty$ is an exact functor.
    \item \label{item:support} For a tame inertial type $\tau$ and any $\sig\in \JH(\osig(\tau))$, the support $\supp_{R_{\rhobar}^\square}M_\infty(\sig^\circ(\tau))$ (resp.~$\supp_{R_{\rhobar}^\square}M_\infty(\sig)$) is a union of irreducible components in $\Spec R_{\rhobar}^{\eta,\tau}$ (resp.~in $\Spec R_{\rhobar}^{\eta,\tau}/\varpi$).
    \item \label{item:control} $M_\infty(\sig)\neq 0$ if and only if $\sig \in W(\rbar)$.
\end{enumerate}

It is expected that $\supp_{R_{\rhobar}^\square}M_\infty(\sig^\circ(\tau))$ is all of $\Spec R_{\rhobar}^{\eta,\tau}$. 
In fact, it should be equivalent to a certain modularity lifting theorem under mild hypotheses using ideas of \cite{kisin-fontaine-mazur} and results of \cite{FKP}. 
An observation due to \cite{gee-kisin}, with ideas tracing back ultimately to \cite{BM}, is that this full support statement determines the cycles of the modules $M_\infty(\osig(\tau))$ and thus determines the cycles of the modules $M_\infty(\sigma)$ using \eqref{item:exact} (assuming that the collection of $\osig(\tau)$ generates the Grothendieck group of $\F[\GSp_4(k)]$-modules). 
The weight part of Serre's conjecture can then be deduced from \eqref{item:support} and \eqref{item:control}. 

This full support statement is known in some cases. 
For $\GL_2$, potential diagonalizability techniques actually give the full support statement in far greater generality. 
Beyond $\GL_2$, potential diagonalizability has so far enjoyed less utility. 
Instead, known full support statements have relied on the irreducibility of tamely potentially Galois deformation spaces. 
\cite{LLLM,GL3Wild} show that generic tamely potentially crystalline Galois deformation rings are integral domains under a combinatorial and explicit genericity condition by showing that their special fibers are reduced. 
However, for more general groups, reduced special fibers occur rarely. 
Known cases of the Breuil--M\'ezard conjecture \cite{MLM,FLH} show that special fibers of generic tamely potentially crystalline Galois deformation rings cannot all be reduced for $\GL_n$ with $n\geq 4$ because (generic) Deligne--Lusztig representations are not residually multiplicity-free. 
In fact, the only other split group where (generic) Deligne--Lusztig representations are residually multiplicity free is $\GSp_4$. 
However, the reduced special fiber property for $\GSp_4$ fails because of a new phenomenon, known as \emph{entailment}, 
wherein the supports of $M_\infty(\sigma)$ are reducible. 
From this perspective, the case of $\GSp_4$ seems to be much closer to the general case. 

Nevertheless, generic tamely potentially crystalline Galois deformation rings are often still integral domains when $\rhobar$ is tamely ramified. 
A new local model theory shows that generic tamely potentially crystalline Galois deformation rings can be computed as the completion of an algebraic variety over $\cO$ \cite{MLM,lee_thesis}. 
This algebraic variety has a torus symmetry which allows one to show that it is unibranch at torus fixed points under an inexplicit genericity assumption. 
As a consequence, for generic tamely ramified Galois representations, all of their tamely potentially crystalline Galois deformation rings are integral domains. 
Applying the Taylor--Wiles method then yields the weight part of Serre's conjecture under these hypotheses. 
(In practice, deducing the weight part of Serre's conjecture in these situations is simplified with the additional observation that the cycles of the modules $M_\infty(\sigma)$ are necessarily effective.) 
While for large $p$, most tamely ramified Galois representations satisfy this genericity condition, verifying genericity for any specific Galois representation is computationally involved. 
In fact, it is tricky to exhibit a single explicit example of a Galois representation for $\GL_4$ or $\GSp_4$ which satisfies this genericity condition, though many exist!
While we expect that the deformation rings are integral domains under a milder genericity condition, we do not know how to prove this. 

\subsection{Methods}
In this paper, we introduce a new method for proving the weight part of Serre's conjecture that does not rely on the full knowledge of the support of patched modules. 
While the present paper concerns $\GSp_4$, many aspects of the method apply in greater generality. 

It is known that $M_\infty(\sigma) = 0$ for $\sigma \notin W^?(\rbar)$, and so it suffices to show that $M_\infty(\sigma) \neq 0$ for $\sigma \in W^?(\rbar)$. 
Rather than determine the support of $M_\infty(\sigma)$ for $\sigma \in W^?(\rbar)$, we show that $M_\infty(\sigma)$ is supported on $\cC_{\sig}(\rhobar)$. As a first step, we show that there exists $\sig$ with such a property. This can be easily shown if $\sig$ is an obvious weight (Definition \ref{def:outer}), assuming that $M_\infty(\sig)\neq 0$. In the earlier version of this article, we assumed that $\rbar$ is potentially diagonalizably automorphic to apply the change of weight result \cite{BLGGT} which implies that $M_\infty(\sig)\neq 0$ for an obvious weight $\sig$. Instead, we imploy a new technique which we call \textit{geometric weight cycling}.

Originally, a weight cycling argument shows that if $\sig$ is a modular Serre weight with a vanishing Hecke eigenvalue, then there is another modular Serre weight $\kappa$ contained in a list of Serre weights (cf.~\cite{EGH}). The difficultues of applying the technique are that it is unclear how to compute the list explicitly, and moreover, it is difficult to determine which one in the list is modular. Our idea is to use the Taylor--Wiles method to interpret the weight cycling argument in terms of supports of certain patched modules. If $M_\infty(\sig)$ is not supported on $\cC_{\sig}(\rhobar)$, then it is supported on $\cC_{\kappa}(\rhobar)$ for some $\kappa$ linked to $\kappa$. Then our geometric weight cycling shows that there is a modular Serre weight $\kappa'$ linked to $\sig$. If $M_\infty(\kappa')$ is not supported on $\cC_{\kappa'}(\sig)$, we can repeat the argument until we get the desired result (see the proof of Theorem \ref{thm:abstract-SWC}). 


Our main contribution is to propagate the statement that $M_\infty(\sigma)$ is supported on $\cC_{\sig}(\rhobar)$ to all Serre weights in $W^?(\rbar)$. 
Recall that there is a natural bijection between $\JH(\osig(\tau))\cap W^?(\rbar)$ and the set of irreducible components in $\Spec R^{\eta,\tau}_{\rhobar}/\varpi$ denoted by $\sig \mapsto \cC_{\sig}(\rhobar)$. 
Since $R_{\rhobar}^{\eta,\tau}/\varpi$ is not reduced (even generically), there could be several irreducible components in $\Spec R^{\eta,\tau}_{\rhobar}$ whose special fiber contains $\cC_{\sig}(\rhobar)$. 
Thus, even if $\cC_{\sig}(\rhobar)$ is in the support of $M_\infty(\sig^\circ(\tau))$, it is unclear which of the components in $\Spec R^{\eta,\tau}_{\rhobar}$ is in the support of $M_\infty(\sig^\circ(\tau))$. Fortunately, if $\sig$ is an \textit{outer} weight (Definition \ref{def:outer}), the local model theory shows that the component of $\Spec R^{\eta,\tau}_{\rhobar}/\varpi$ corresponding to $\sig$ occurs with multiplicity one. In particular, $\cC_{\sig}(\rhobar)$ generizes to a unique component in $\Spec R^{\eta,\tau}_{\rhobar}$ (i.e.~it is contained in a unique component in $\Spec R^{\eta,\tau}_{\rhobar}$).

Our main technical result proves that for various $\tau$, there is a pair of distinct outer weights $\sig_1,\sig_2 \in \JH(\osig(\tau))\cap W^?(\rbar)$ such that $\cC_{\sig_1}(\rhobar)$ and $\cC_{\sig_2}(\rhobar)$ generize (uniquely) to a common component $\cC$ in $\Spec R^{\eta,\tau}_{\rhobar}$ (see Theorem \ref{thm:irred-comp-def-ring} for the precise result). By \eqref{item:exact} and \eqref{item:support}, this implies that $M_\infty(\sig^\circ(\tau))$ is supported on $\cC_{\sig_1}(\rhobar)$ if and only if it is supported on $\cC_{\sig_2}(\rhobar)$. 
Moreover, for an outer weight $\sig$, the only patched module of a Serre weight in $\JH(\osig(\tau))$ that is possibly supported on $\cC_{\sig}(\rhobar)$ is $M_\infty(\sig)$. 
Thus $M_\infty(\sig_1)$ is supported on $\cC_{\sig_1}(\rhobar)$ if and only if $M_\infty(\sig_2)$ is supported on $\cC_{\sig_2}(\rhobar)$. 
Finally, we show that any two Serre weights in $W^?(\rbar)$ can be connected by a chain of components $\cC$. 
With the initial weight provided by the geometric weight cycling, this completes the proof of Theorem \ref{thm:main}.

We end the discussion by describing our method for finding the family of components $\mathcal{C}$, which crucially relies on the recent construction of the Emerton--Gee stack and an analysis of its topology using local models. Let $\cX_{\Sym}^{\eta,\tau}$ be the closed substack of the Emerton--Gee stack for $\GSp_4$ parameterizing potentially crystalline representations of type $(\eta,\tau)$. Similar to Galois deformation rings, there is a bijection between $\JH(\osig(\tau))$ and the set of irreducible components in $\cX_{\Sym}^{\eta,\tau}$, denoted by $\sig \mapsto \cC_\sig$.  
The first observation is that the set of irreducible components in $\Spec R_{\rhobar}^{\eta,\tau}$ are in correspondence with the preimage of $\rhobar$ under the normalization map $\nu: (\cX_{\Sym}^{\eta,\tau})^{\nm} \ra \cX_{\Sym}^{\eta,\tau}$. Note that for outer weight $\sig$, $\cC_{\sig}$ occurs with multiplicity one in the special fiber and thus pulls back to a unique irreducible component $\cC_{\sig}^{\nm}$ under $\nu$. For outer weights $\sig_1$ and $\sig_2$, $\cC_{\sig_1}(\rhobar)$ and $\cC_{\sig_2}(\rhobar)$ have a common unique generization if and only if $\nu(\cC_{\sig_1}^{\nm}\cap\cC_{\sig_2}^{\nm})$ contains $\rhobar$.


Following \cite{LLLMextremal}, we consider pairs of outer weights $\sig_1, \sig_2\in \JH(\osig(\tau))$ such that $\cC_{\sig_1}\cap \cC_{\sig_2}$ intersects with the normal locus of $\cX^{\eta,\tau}_{\Sym}$. In particular, there exists an open substack $\cU \subset \cC_{\sig_1}\cap \cC_{\sig_2}$ contained in the normal locus. Thus, we can find a copy of $\cU$ in $\cC_{\sig_1}^{\nm}\cap\cC_{\sig_2}^{\nm}$. For some choices of $\tau$ (relative to $\rhobar$) considered in \loccit, $\rhobar$ is in $\cU$ and thus $\nu(\cC_{\sig_1}^{\nm}\cap\cC_{\sig_2}^{\nm})$ contains $\rhobar$. However, such choices of $\tau$ are too restrictive for applications and can only prove that obvious weights are in $W(\rbar)$. Instead, we show that $\rhobar$ is in the closure of $\cU$ in $\cX^{\eta,\tau}_{\Sym}$ for a larger family of $\tau$. Our family of $\tau$ is flexible enough to reach all Serre weights in $W^?(\rbar)$. 
This can be reduced to finding torus fixed points in a certain subset of a deformed affine Springer fiber, which 
we accomplish by explicit computations along the lines of \cite{BA} (see Proposition \ref{prop:T-fixed-points}). Let us emphasize that $\rhobar$ is not known to be in the normal locus and is definitely not in $\cU$ (unless $\tau$ is in the family considered in \cite{LLLMextremal}). In particular, our argument cannot be made using Galois deformation rings instead of the Emerton--Gee stack because it uses a part of $\cX_{\Sym}^{\eta,\tau}$ \textit{away from} $\rhobar$ to analyze the geometry \textit{near} $\rhobar$.

\subsection{Application}
In the proof of Serre weight conjectures explained above, we only obtain a lower bound on $\supp_{R^{\eta,\tau}_{\rhobar}}M_\infty(\sig^\circ(\tau))$ for a family of $\tau$. This contrasts to earlier works (\cite{MLM,lee_thesis}) proving that $M_\infty(\sig^\circ(\tau))$ is fully supported on $\Spec R^{\eta,\tau}_{\rhobar}$ (under an inexplicit genericity assumption). However, we can still prove the full support result for all generic $\tau$ (and higher Hodge--Tate weights) by combining our result with the Breuil--M\'ezard conjecture for tamely potentially crystalline deformation rings in small weight \cite{FLH}, explicit Breuil--M\'ezard cycles from \cite{LHL-BM}, and weight shifting mod $p$ differential operators of \cite{Ortiz}. As a result, we deduce modularity lifting theorems for small Hodge--Tate weights and generic tame inertial types.

Since the work of Kisin, the Breuil--M\'ezard conjecture has been known to be closely related to modularity lifting theorems and the weight part of Serre's conjecture (see e.g.~\cite{EG}). 
\cite{kisin-fontaine-mazur} used the $p$-adic Langlands correspondence for $\GL_2(\Qp)$ to prove the Breuil--M\'ezard conjecture and, as a consequence, deduced many instances of the Fontaine--Mazur conjecture. 
In the reverse direction, \cite{gee-kisin} used modularity lifting theorems to prove the Breuil--M\'ezard conjecture for 2-dimensional Barsotti--Tate representations, which was in turn used to prove the weight part of Serre's conjecture for Hilbert modular forms as formulated by Buzzard--Diamond--Jarvis. 

The breakthrough work \cite{FLH} proved the Breuil--M\'ezard conjecture for unramified groups with small Hodge--Tate weights and generic tame inertial types using a novel method that avoids both a $p$-adic Langlands correspondence and Taylor--Wiles patching. 
Their method provides an algorithm to compute Breuil--M\'ezard cycles explicitly, which was implemented for low rank groups including $\GSp_4$ in \cite{LHL-BM}. In particular, their result verifies that certain Breuil--M\'ezard cycles for $\GSp_4$ exhibit entailment. 
In order to deduce modularity lifting theorems for $\GSp_4$ using the result of \cite{FLH,LHL-BM}, we need to prove a strong form of the weight part of Serre's conjecture. 
That is, we need to realize the local Breuil--M\'ezard cycles by patching global spaces of automorphic forms \`a la Taylor--Wiles. 

We assume $F=\Q$ to apply the main result of \cite{Ortiz}. To simplify the argument, let us consider here the \textit{minimal} case, which allows us to focus on the places dividing $p$. The general case can be treated using Taylor's Ihara avoidance argument. 
First, our results give lower bounds for the supports of patched modules for Serre weights. 
A computation of colength one deformation rings shows that these bounds are sharp for Serre weights in the bottom two $p$-restricted alcoves. 
We need to show that the support of $M_\infty(\sig)$ contains one extra component $\cC_{\sig'}(\rhobar)$ when $\sig$ is in the next $p$-restricted alcove and $\sig'$ is the Serre weight linked to $\sig$ in the lowest $p$-restricted alcove. 
The lower bounds can be transferred to patched coherent cohomology using a comparison of rational \'etale and coherent cohomologies. 
Combining this with a weight shifting mod $p$ differential operator on sections of automorphic vector bundles on $\Sh_U$ constructed in \cite{Ortiz}, we see that $\cC_{\sig'}(\rhobar)$ is in the support of the patched coherent cohomology corresponding to the mod $p$ (dual) Weyl module $W$ with socle $\sig$. 
Transferring this back to patched \'etale cohomology, we show that $\cC_{\sig'}(\rhobar)$ is in the support of $M_\infty(W)$. 
The upper bound for supports of patched modules for Serre weights with the highest weight in the bottom two $p$-restricted alcoves then shows that $\cC_{\sig'}(\rhobar)$ is in the support of $M_\infty(\sig)$. 




Combining our strong form of the weight part of Serre's conjecture and the instances of the Breuil--M\'ezard conjecture proved in \cite{FLH} yields the following modularity lifting result. 

\begin{thm}[Theorem \ref{thm:MLT}]
    Assume that $F=\Q$. Suppose that $\rbar$ is automorphic, satisfies Taylor--Wiles conditions, and $\rbar|_{G_{\Qp}}$ is tame and 9-generic. If $r: G_\Q \ra \GSp_4(E)$ is a lift of $\rbar$ that is unramified at all but finitely many places and potentially crystalline of type $(\lam+\eta,\tau)$ at $p$ with $(2h_{\lam}+6)$-generic tame inertial type $\tau$, then $r$ is automorphic. 
\end{thm}

\subsection{Acknowledgements}

D.L.~was supported by the National Science Foundation under agreement DMS-2302623 and a start-up grant from Purdue University. B.LH.~acknowledges support from the National Science Foundation under grants Nos.~DMS-1952678 and DMS-2302619 and the Alfred P.~Sloan Foundation. 
H.L.~was supported by the AMS--Simons Travel Grant and by the KIAS Individual Grant (HP103001) at the Korea Institute for Advanced Study. 
D.L.~and B.L.H.~thank the Max Planck Institute and the Hausdorff Center of Mathematics for excellent working conditions during preliminary stages of this project. H.L.~thanks Martin Ortiz and Sug Woo Shin for answering questions. We would like to thank the anonymous referee for careful reading and comments.

\subsection{Notation}
For a field $K$, we denote by $\ovl{K}$ a fixed separable closure of $K$ and let $G_K \defeq \Gal(\ovl{K}/K)$.
If $K$ is defined as a subfield of an algebraically closed field, then we set $\ovl{K}$ to be this field.

If $K$ is a nonarchimedean local field, we let $I_K \subset G_K$ denote the inertial subgroup and $W_K \subset G_K$ denote the Weil group.
We fix a prime $p\in\Z_{>0}$.
Let $E \subset \ovl{\Q}_p$ be a subfield which is finite-dimensional over $\Q_p$.
We write $\cO$ to denote its ring of integers, fix a uniformizer $\varpi\in \cO$ and let $\F$ denote the residue field of $E$.
We will assume throughout that $E$ is sufficiently large.

\subsubsection{Reductive groups}
\label{sec:not:RG}
Let $G$ denote a split connected reductive group (over some ring) together with a Borel $B$, a maximal split torus $T \subset B$, and $Z \subset T$ the center of $G$.  
Let $\Phi^{+} \subset \Phi$ (resp. $\Phi^{\vee, +} \subset \Phi^{\vee}$) denote the subset of positive roots (resp.~positive coroots) in the set of roots (resp.~coroots) for $(G, B, T)$. 
We use the notation $\alpha > 0$ (resp.~$\alpha < 0$) for a positive (resp.~negative) root $\alpha\in \Phi$. 
Let $\Delta$ (resp.~$\Delta^{\vee}$) be the set of simple roots (resp.~coroots).
Let $X^*(T)$ be the group of characters of $T$, and set $X^0(T)$ to be the subgroup consisting of characters $\lambda\in X^*(T)$ such that $\langle\lambda,\alpha^\vee\rangle=0$ for all $\alpha^\vee\in \Delta^{\vee}$.
Let $\Lambda_R \subset X^*(T)$ denote the root lattice for $G$.
Let  $W(G,T)$ denote the Weyl group of $(G,T)$.  
We sometimes write $W$ for $W(G,T)$ when there is no chance for confusion. Let $w_0$ denote the longest element of $W$.
Let $W_a$ (resp.~$\tld{W}$) denote the affine Weyl group and extended affine Weyl group 
\[
W_a = \Lambda_R \rtimes W, \quad \tld{W} = X^*(T) \rtimes W
\]
for $G$.
We use $t_{\nu} \in \tld{W}$ to denote the image of $\nu \in X^*(T)$.


Let $M$ be a free $\Z$-module of finite rank (e.g. $M=X^*(T)$). 
The duality pairing between $M$ and its $\Z$-linear dual $M^*$ will be denoted by $\langle \ ,\,\rangle$.
If $A$ is any ring, the pairing $\langle \ ,\,\rangle$ extends by $A$-linearity to a pairing between $M\otimes_{\Z}A$ and $M^*\otimes_{\Z}A$, and by an abuse of notation it will be denoted with the same symbol $\langle \ ,\,\rangle$. 

We write $G^\vee = G^\vee_{/\Z}$ for the split connected reductive group over $\Z$  defined by the root datum $(X_*(T),X^*(T), \Phi^\vee,\Phi)$. 
This defines a maximal split torus $T^\vee\subseteq G^\vee$ such that we have canonical identifications $X^*(T^\vee)\cong X_*(T)$ and $X_*(T^\vee)\cong X^*(T)$.

Let $V\defeq X^*(T)\otimes_{\Z}\R$.
For $(\alpha,k)\in \Phi \times \Z$, we have the root hyperplane $H_{\alpha,k}\defeq \{x \in V \mid \langle x,\alpha^\vee\rangle=k\}$ and the half-space $H^{+}_{\alpha, k} \defeq \{ x \in V \mid \langle x, \alpha^\vee\rangle > k \}$  and $H^{-}_{\alpha, k} \defeq  \{ x \in V \mid \langle x, \alpha^\vee\rangle < k \}.$ 
An alcove is a connected component of $V \setminus \big(\bigcup_{(\alpha,k)}H_{\alpha,k}\big)$. 

We say that an alcove $A$ is \emph{restricted} if $0<\langle\lambda,\alpha^\vee\rangle<1$ for all $\alpha\in \Delta$ and $\lambda\in A$.
We let $A_0$ denote the (dominant) base alcove, i.e.~the set of $\lambda\in X^*(T)\otimes_{\Z}\R$ such that $0<\langle\lambda,\alpha^\vee\rangle<1$ for all $\alpha\in \Phi^+$. 
Let $\cA$ denote the set of alcoves. 
Recall that $\tld{W}$ acts transitively on the set of alcoves, and $\tld{W}\cong\tld{W}_a\rtimes \Omega$ where $\Omega$ is the stabilizer of $A_0$.
We define
\[\tld{W}^+\defeq\{\tld{w}\in \tld{W}:\tld{w}(A_0) \textrm{ is dominant}\} \ \text{and } \tld{W}^+_1\defeq\{\tld{w}\in \tld{W}^+:\tld{w}(A_0) \textrm{ is restricted}\}.\]
We fix an element $\eta\in X^*(T)$ such that $\langle \eta,\alpha^\vee\rangle = 1$ for all positive simple roots $\alpha$ and let $\tld{w}_h$ be $w_0 t_{-\eta}\in \tld{W}^+_1$. 

A ($\eta$-shifted) \textit{$p$-alcove} is a connected component of $X^*(T)\otimes_\Z \R \backslash (\cup_{(\al,k)} (H_{\al,pk}-\eta))$. We define \textit{$p$-restricted} (resp.~\textit{dominant})  $p$-alcoves analogous to the previous paragraph.

Let $\cO_p$ be a finite \'etale $\Z_p$-algebra.
Then $\cO_p$ is isomorphic to a product $\prod_{S_p} \cO_{v}$ for some finite set $S_p$ where for each $v\in S_p$, $\cO_{v}$ is the ring of integers in a finite unramified extension of $\Q_p$. We write $F_p \defeq \cO_p[1/p]$. For each $v\in S_p$, let $F_v\defeq \cO_v[1/p]$ and $k_{v}$ be the residue field of $F_v$. We let $k_p \defeq \prod_{v\in S_p} k_v$. 

In global applications, $S_p$ will be a finite set of places dividing $p$ of a number field $F$.
When working locally, $S_p$ will have cardinality one, in which case we drop the subscripts from $k_{v}$ and denote the single extension $F_{v}$ of $\Q_p$ by $K$.

If $G$ is a split connected reductive group over $\Z_p$, with Borel $B$, maximal split torus $T$, and center $Z$, we let $G_0 \defeq \Res_{\cO_p/\Z_p} G_{/\cO_p}$ with Borel subgroup $B_0 \defeq  \Res_{\cO_p/\Z_p} B_{/\cO_p}$, maximal torus $T_0 \defeq \Res_{\cO_p/\Z_p} T_{/\cO_p}$, and $Z_0 = \Res_{\cO_p/\Z_p} Z_{/\cO_p}$. 
Assume that $\cO$ contains the image of any ring homomorphism $\cO_p \ra \ovl{\Z}_p$ and let $\cJ$ be the set of ring homomorphisms $\cO_p \ra \cO$.
Then $\un{G} \defeq (G_0)_{/\cO}$ is naturally identified with the split reductive group $G_{/\cO}^{\cJ}$. 
We similarly define $\un{B}, \un{T},$ and $\un{Z}$.  
Corresponding to $(\un{G}, \un{B}, \un{T})$, we have the set of positive roots $\un{\Phi}^+ \subset \un{\Phi}$ and the set of positive coroots $\un{\Phi}^{\vee, +}\subset \un{\Phi}^{\vee}$.
The notations $\un{\Lambda}_R$, $\un{W}$, $\un{W}_a$, $\tld{\un{W}}$, $\tld{\un{W}}^+$, $\tld{\un{W}}^+_1$, $\un{\Omega}$ should be clear as should the natural isomorphisms $X^*(\un{T}) = X^*(T)^{\cJ}$ and the like. 
The absolute Frobenius automorphism $\varphi$ on $k_p$ induces an automorphism $\pi$ of the identified groups $X^*(\un{T})$ and $X_*(\un{T}^\vee)$ by the formula $\pi(\lambda)_\sigma = \lambda_{\sigma \circ \varphi^{-1}}$ for all $\lambda\in X^*(\un{T})$ and $\sigma: \cO_p \ra \cO$.
We assume that, in this case, the element $\eta\in X^*(\un{T})$ we fixed is $\pi$-invariant.
We similarly define an automorphism $\pi$ of $\un{W}$ and $\tld{\un{W}}$.

\subsubsection{The group $\GSp_4$}\label{sec:not:GSp4}

Let  $\GSp_4$ be the split reductive group over $\Z$ defined by
\begin{align*}
    \GSp_4(R) =\{ g\in \GL_4(R) \mid{}^t g J g = \simc(g)J \text{ for some $\simc(g)\in R^\times$}\}
\end{align*}
for any commutative ring $R$, where $J= \psmat{ & & & 1 \\ & &1& \\&-1& & \\-1 & & & }.$ Then $\simc: g\mapsto \simc(g)$ defines a character of $\GSp_4$, called the \emph{similitude character}. 

When $G=\GSp_4$, we will take $B$ to be the upper triangular Borel and $T$ the diagonal torus. We also write $U$ for the unipotent radical of $B$ and $\overline{U}$ for the unipotent radical of the opposite Borel subgroup. Let $W=W(\GSp_4,T)$. We choose representatives of elements of $W$ in $\GSp_4$ by 
\begin{align*}
       s_1\defeq\pmat{& 1 & & \\ 1& & & \\ & & &1 \\ & & 1&}, \ s_2\defeq\pmat{1 & & & \\ & & 1& \\ &-1  & & \\ & & &1}.
\end{align*}
We denote by $S$ (resp.~$Q$) the Siegel (resp.~Klingen) parabolic subgroup of $\GSp_4$ generated by $B$ and $s_1$ (resp.~$s_2$).

We identify the character group $X^*(T)$ with $\Z^3$ by defining the character corresponding to $(a,b;c)\in \Z^3$ by 
\begin{align*}
    (a,b;c):\pmat{x & & & \\ & y & & 
    \\ & & zy^\mo & \\ & & & zx^\mo} \mapsto x^ay^bz^c. 
\end{align*}
Similarly, we identify the cocharacter group $X_*(T)$ with $\Z^3$ by
\begin{align*}
    (a,b;c): x \mapsto \pmat{x^a & & & \\ & x^b & & 
    \\ & & x^{c-b} & \\ & & & x^{c-a}}.
\end{align*}
The simple roots are $\al_1 = (1,-1;0)$ and $\al_2 = (0,2;-1)$. We define $\be_0 \defeq (1,1;2)$, $\be_1 \defeq (2,1;2)$, $\be_2 \defeq (1,1;1)\in X_*(T)$. Note that $\langle \al_i, \be_j \rangle = \delta_{ij}$ for $i,j\in \{1,2\}$. We choose $\eta\defeq (2,1;0)\in X^*(T)$.

By an exceptional isomorphism, we have $\GSp_4^\vee \simeq \GSp_4$. We denote by $\std:\GSp_4 \into \GL_4$ the standard representation of $\GSp_4$. Let $T_4$ be the maximal diagonal torus of $\GL_4$. Then $\std$ induces $T^\vee \into T_4^\vee$, $X_*(T^\vee) \into X_*(T_4^\vee) $ and $\cT: X^*(T) \ra X^*(T_4)$ by duality. 
Explicitly, we have
\begin{align*}
    \cT(a,b;c) = (a+b+c,a+c,b+c,c).
\end{align*}

We recall the four restricted alcoves for $\GSp_4$:
\begin{align*}
    A_0 &\defeq \CB{(a,b;c) \in X^*(T)\otimes \R \mid 0 < b<a, \  0<a+b < 1 } \\
    A_1 &\defeq \CB{(a,b;c) \in X^*(T)\otimes \R \mid b< a<1,\  1<a+b } \\
    A_2 &\defeq \CB{(a,b;c) \in X^*(T)\otimes \R \mid a-b<1<a,\  a+b<2} \\
    A_3 &\defeq \CB{(a,b;c) \in X^*(T)\otimes \R \mid b<1,\ a-b<1,\  2<a+b}.
\end{align*}
We also have four $p$-restricted $p$-alcoves $C_0,C_1,C_2$, and $C_3$.

When $G$ is a product of copies of $\GSp_4$ indexed over a set $\cJ$ we take $\eta \in X^*(T)$ to correspond to the element $(2,1;0)_{j\in\cJ} \in (\Z^3)^{\cJ}$.
In this case, given $j\in\cJ$ we write $\eta_{j}\in X^*(T)$ to denote the element which corresponds to the tuple $(2,1;0)$ at $j$ and zero elsewhere. Similarly, we write $s_{i,j} \in W$ for $i=1,2$  and $j\in \cJ$ for the tuple $s_i$ at $j$ and identity elsewhere.

\subsubsection{Galois Theory}
\label{sec:not:GT}
We now assume that $S_p$ has cardinality one.
We write $K\defeq F_{v}$ and drop the subscripts from $k_{v}$.
Let $W(k)$ be the ring of Witt vectors which is also the ring of integers $\cO_{K_0}$ of $K_0$.
We denote the arithmetic Frobenius automorphism on $W(k)$ by $\phz$, which acts as raising to $p$-th power on the residue field.  
We fix an embedding $\sigma_0$ of $K_0$ into $E$ (equivalently an embedding $k$ into $\F$) and define $\sigma_j = \sigma_0 \circ \phz^{-j}$, which gives an identification between $\cJ=\Hom(k,\F)$ and $\Z/f\Z$.

We normalize Artin's reciprocity map $\mathrm{Art}_{K}: K\s\ra W_{K}^{\mathrm{ab}}$ in such a way that uniformizers are sent to geometric Frobenius elements.


%
%
Given an element $\pi_1 \defeq (-p)^{\frac{1}{p^{f}-1}}\in \overline{K}$ we have a corresponding character $\omega_{K}:I_K \ra W(k)^{\times}$. By composing $\sig \in \Hom(k,\F)$, we define 
\[
\omega_{K,\sig}\defeq \sig \circ \omega_{K}:I_K \ra \cO^{\times}. 
\]


Let $\rho: G_K \ra \GSp_4(E)$ be a $p$-adic, de Rham Galois representation. For $\sigma: K\iarrow E$, we have the $\sigma$-labeled  Hodge--Tate cocharacter $\mu_{\sigma} \in X_*(T^\vee)$. For any algebraic representation $r$ of $\GSp_4$, then $r\circ \mu_{\sigma}$ matches the multiset of  $\sigma$-labeled  Hodge--Tate weights of $r\circ\rho$, i.e.~the set of integers $i$ such that $\dim_E\big(r\circ \rho\otimes_{\sigma,K}\C_p(-i)\big)^{G_K}\neq 0$ (with the usual notation for Tate twists). In particular, the cyclotomic character $\eps$ has Hodge--Tate weights 1 for all embeddings $\sigma:K\iarrow E$. 
The \emph{inertial type} of $\rho$ is the isomorphism class of $\mathrm{WD}(\rho)|_{I_K}$, where $\mathrm{WD}(\rho)$ is the Weil--Deligne representation attached to $\rho$ as in \cite{CDT}, Appendix B.1 (in particular, $\rho\mapsto\mathrm{WD}(\rho)$ is \emph{covariant}). An inertial type is a morphism $\tau: I_K\ra \GSp_4(E)$ with open kernel and which extends to the Weil group $W_K$ of $G_K$. We say that $\rho$ has type $(\mu,\tau)$ if $\rho$ has the Hodge--Tate cocharacters $\mu=(\mu_{\sig})_{\sig}$ and inertial type given by (the isomorphism class of) $\tau$.


\subsubsection{Miscellaneous}
\label{sec:not:mis}

For any ring $S$, we define $\mathrm{Mat}_n(S)$ to be the set of $n\times n$ matrices with entries in $S$. 
If $M\in \mathrm{Mat}_n(S)$ and $A\in \GL_n(S)$ we write
\begin{equation}
\label{def:adj}
\Ad(A)(M)\defeq A\,M\,A^{-1}.
\end{equation}

We write $\Ind$ and $\cind$ for the usual induction and compact induction, respectively. 

If $X$ is an ind-scheme defined over $\cO$, we write $X_E\defeq X\times_{\Spec\cO} \Spec E$ and $X_{\F}\defeq X\times_{\Spec \cO}\Spec \F$ to denote its generic and special fiber, respectively.

For an algebraic stack $X$ and $d\in \Z$, we let $\Irr_d(X)$ be the set of $d$-dimensional irreducible substacks of $X$.

For a topological ring $R$, if $X=\Spec R$ or $X=\Spf R$, then we write $\cO(X) \defeq R$. 

%% file: representation_theory.tex
\section{Representation theory}

\subsection{Extended affine Weyl groups}

In this section, we briefly collect some background material on Weyl groups that will be needed throughout the paper.

Recall from \S \ref{sec:not:RG} that $G$ is a split reductive group with split maximal torus $T$ and Borel $B$, $W \defeq W(G,T)$, $W_a \defeq \Lambda \rtimes W$, and $\tld{W}\defeq X^*(T) \rtimes W$. 
We denote by $\cA$ the set of alcoves of $X^*(T)\otimes \R$ and by $A_0 \in \cA$ the dominant base alcove. We let $\uparrow$ denote the upper arrow ordering on alcoves as defined in \cite[\S II.6.5]{RAGS}. For $G=\GSp_4$, the four restricted alcoves $A_0,\dots, A_3$ satisfy
\begin{align*}
    A_0 \uparrow A_1 \uparrow A_2 \uparrow A_3.
\end{align*}

Since $W_a$ acts simply transitively on the set of alcoves, $\tld{w} \mapsto \tld{w}(A_0)$ induces a bijection $W_a \risom \cA$ and thus an upper arrow ordering $\uparrow$ on $W_a$.
The dominant base alcove $A_0$ also defines a set of simple reflections in $W_a$ and thus a Coxeter length function on $W_a$ denoted $\ell(-)$ and a Bruhat order on $W_a$ denoted by $\leq$. 

If $\Omega \subset \tld{W}$ is the stabilizer of the base alcove, then $\tld{W} = W_a \rtimes \Omega$ and so $\tld{W}$ inherits a Bruhat and upper arrow order in the standard way: For $\tld{w}_1, \tld{w}_2\in W_a$ and $\delta\in \Omega$, $\tld{w}_1\delta\leq \tld{w}_2\delta$ (resp.~$\tld{w}_1\delta\uparrow \tld{w}_2\delta$) if and only if $\tld{w}_1\leq \tld{w}_2$ (resp.~$\tld{w}_1\uparrow \tld{w}_2$), and elements in different right $W_a$-cosets are incomparable. 
We extend $\ell(-)$ to $\tld{W}$ by letting $\ell(\tld{w}\delta)\defeq \ell(\tld{w})$ for any $\tld{w}\in W_a$, $\delta\in \Omega$.


Let $(\tld{W}^\vee,\leq)$ be the following partially ordered group: $\tld{W}^\vee$ is identified with $\tld{W}$ as a group, and $\ell(-)$ and $\leq$ are defined with respect to the \emph{antidominant} base alcove.  

\begin{defn} 
\label{affineadjoint} We define a bijection $\tld{w}\mapsto \tld{w}^*$ between $\tld{W}$ and $\tld{W}^\vee$ as follows: for $\tld{w} = t_{\nu}w \in \tld{W}$, with $w\in W$ and $\nu\in X^*(T) = X_*(T^{\vee})$, then $\tld{w}^*\defeq  w^{-1}t_{\nu} \in \tld{W}^\vee$. 
\end{defn}
\noindent This bijection respects notions of length and Bruhat order (see \cite[Lemma 2.1.3]{LLL}).

We recall the definition of the admissible set from \cite{KR}:
\begin{defn} \label{defn:adm} For $\lambda \in X^*(T)$, define 
\[
\Adm(\lambda) \defeq  \left\{ \tld{w} \in \tld{W} \mid \tld{w} \leq t_{w(\lambda)} \text{ for some } w \in W \right\}.
\]
\end{defn}

We define the critical strips to be strips $H^{(0,1)}_{\alpha} = \{ x \in X^*(T)\otimes \R \mid 0 < \langle x, \alpha^\vee\rangle < 1 \}$ where $\alpha \in \Phi^+$.

\begin{defn} \label{defn:regular}%
An alcove $A \in \cA$ is \emph{regular} if $A$ does not lie in any critical strip.   For any $\tld{w} \in \tld{W}$, we say $\tld{w}$ is \emph{regular} if $\tld{w}(A_0)$ is regular.  Define 
\[
\Adm^{\text{reg}}(\lambda) \defeq \{ \tld{w} \in \Adm(\lambda)\mid \tld{w} \text{ is regular} \}. 
\]
\end{defn}

For $\lam\in X_*(T^\vee)$, we denote by $\Adm^\vee(\lam)$ (resp.~$\Adm^{\vee,\reg}(\lam)$) the image of $\Adm(\lam)$ (resp.~$\Adm^{\reg}(\lam)$) under the bijection $(-)^*$ between $\tilW$ and $\tilW^\vee$.

\begin{defn}\label{defn:colength}
Let $\lam \in X^*(T)$ and $\tilw\in \Adm(\lam)$. Then we call $l(t_{\lam})-l(\tilw)$ \textit{the colength of $\tilw$} (in $\Adm(\lam)$). If $l(t_{\lam})-l(\tilw)=1$, then we say $\tilw$ is of colength one.    
\end{defn}

For later uses, we recall the various notions of genericity for elements of $X^*(T)$ from  \cite[Definition 2.1.10]{MLM}.

\begin{defn}
\label{defn:var:gen}
Let $\lambda\in X^*(T)$ be a weight and let $m\geq 0$ be an integer. 
\begin{enumerate}
\item
\label{defn:deep}
We say that $\lambda$ \emph{lies $m$-deep in its $p$-alcove} if for all $\alpha\in \Phi^{+}$, there exist integers $m_\alpha\in \Z$ such that $pm_\alpha+m<\langle \lambda+\eta,\alpha^\vee\rangle<p(m_\alpha+1)-m$. 

\item
\label{defn:mugeneric} We say that $\lambda \in X^*(T)$ is \emph{$m$-generic} if 
$m < |\langle \lambda, \alpha^{\vee} \rangle + pk|$
for all $\alpha \in \Phi$ and $k\in \Z$ (or equivalently, $\lambda-\eta$ is $m$-deep in its $p$-alcove).
\end{enumerate} 
\end{defn}

\subsubsection{Parabolic shapes}\label{sec:para-shape} 
Let $M\subset G$ be a Levi subgroup and $M^\vee \subset G^\vee$ be the dual Levi subgroup. We write $W_M \defeq W(M,T)$ and $\tilW_M\defeq X^*(T) \rtimes W_M$. Note that $W_M$ is naturally a subgroup of $W$. Let $\Delta_M$ be the set of simple roots for $(M,B\cap M, T)$. We define
\begin{align*}
    W^M \defeq \{w\in W \mid l(s_{\alpha}w) > l(w) \ \forall \alpha \in \Delta_M\}.
\end{align*}
Note that for any $w\in W$, there exists a unique decomposition $w= w_M w^M$ where $w_M\in W_M$ and $w^M \in W^M$.

Let $\lam \in X_*(T^\vee)$. We can view it as a cocharacter of $T^\vee\subset M^\vee$ and define $\Adm_M^\vee(\lam)$ as the admissible set for $M^\vee$ and $\lam$. Since the natural map $\tilW^\vee_M \into \tilW^\vee$ preserves the Bruhat order, it induces a map $\Adm_M^\vee(\lam) \into \Adm^\vee(\lam)$ and $\Adm_M(\lam) \into \Adm(\lam)$.

\begin{lemma}\label{lem:aff-ref-levi}
    Let $\tilz \in \tilW_M$ and $r \in \tilW_M$ be an affine reflection. If $l(\tilz)<l(r\tilz)$, then for any $w^M\in W^M$, $ l( (w^M)^\mo \tilz w^M) < l ( (w^M)^\mo r \tilz w^M)$.
\end{lemma}
\begin{proof}
    Let $r$ be the reflection across the hyperplane $H_{\al,m}$, i.e.~$r = t_{m\alpha}s_{\alpha}$, for some $\al \in \Phi^+$ and $m\in \Z$. Let $A_{0,M}$ be the dominant base alcove for $M$. The assumption $\ell(\tilz)<\ell(r\tilz)$ is equivalent to $A_0$ and $\tilz(A_0)$ lie on the same half-space divided by $H_{\al,m}$ \cite[\S4.4 Proposition]{humphreys-coxeter}. In other words, for any $\mu \in A_{0,M}$, $\langle \tilz(\mu), \alpha^\vee \rangle < m$ if $m>0$ and $\langle \tilz(\mu), \alpha^\vee \rangle > m$ if $m\le 0$. 
    To prove the claim, we need to show that for any $\mu \in A_0$,
     \begin{align*}
        \langle (w^M)^\mo \tilz w^M(\mu), (w^M)^\mo(\alpha)^\vee \rangle  = \langle \tilz w^M(\mu), \alpha^\vee \rangle
    \end{align*}
    is less than $m$ if $m>0$ and is greater than $m$ if $m \le 0$. 
    This follows from the observation that $w^M(\mu)$ 
    is in $A_{0,M}$.
\end{proof}

\begin{lemma}\label{lem:adm-levi}
    For any $w^M\in W^M$, the set $(w^M)^\mo \Adm_M(\lam) w^M$ is contained in $\Adm(\lam)$.  
\end{lemma}

\begin{proof}
    Let $\tilz \in \Adm_M(\lam)$. Then $\tilz \le t_{w_M(\lam)}$ for some $w_M \in W_M$. By definition, this implies that there are affine reflections $r_1,\dots, r_m\in \tilW_M$ such that $r_m\cdots r_1 \tilz = t_{w_M(\lam)}$ and
    \begin{align*}
        l(\tilz) < l(r_1 \tilz) < \dots < l(r_m\cdots r_1 \tilz).
    \end{align*}
    By Lemma \ref{lem:aff-ref-levi}, we have
    \begin{align*}
        l((w^M)^\mo \tilz w^M) < l((w^M)^\mo r_1 \tilz w^M) < \dots < l((w^M)^\mo r_m\cdots r_1 \tilz w^M).
    \end{align*}
    Since $(w^M)^\mo r_i w^M \in \tilW$ is an affine reflection, this proves that $(w^M)^\mo \tilz w^M \le t_{(w^M)^\mo w_M(\lam)}$ in $\tilW$.
\end{proof}

\begin{defn}
    For $w\in W$, we denote by $w^\diamond \in \tilW^+_1/X^0(T)$ the unique element whose image under the natural quotient $\tilW^+_1/X^0(T) \onto W$ is $w$. This defines a bijection between $W$ and $\tilW^+_1/X^0(T)$. We often denote a representative of $w^\diamond\in \tilW^+_1/X^0(T)$ in $\tilW^+_1$ by $w^\diamond$ again.  
\end{defn}

\begin{rmk}\label{rmk:two-components}
    Let $w \in W$ and write $w^\diamond= t_\nu w$. Let $s$ be a simple reflection with the corresponding root $\alpha\in \Delta$. Let $M$ be the Levi subgroup generated by $T$ and $s$. We are particularly interested in $\tilw \in \Adm(\eta)$ of the form
    \begin{align*}
        \tilw\defeq (w^\diamond)^\mo \tilw_h^\mo w_0 s w^\diamond = w^\mo t_{-\nu} t_\eta st_\nu w. 
    \end{align*}
    It is of colength one, and by varying $w$ and $s$, they form all irregular colength one elements in $\Adm(\eta)$ (see \cite[Proposition 2.1.2]{LLMPQ-CL} and the discussion below).
    
    We have the unique decomposition $w= w_Mw^M$ for some $w_M \in W_M$ and $w^M\in W^M$. It follows from the definition of $w^\diamond$ that if $w_M=e$, then $s(\nu)=\nu$, and if $w_M=s$, then $s(\nu)=\nu-\alpha$. In either case, we have $\tilw = (w^M)^\mo t_{\eta}s w^M$. Moreover, the same argument shows that
    \begin{align*}
        \tilw = ((sw)^\diamond)^\mo \tilw_h^\mo w_0 s {(sw)^\diamond}.
    \end{align*}
\end{rmk}

\begin{rmk}\label{rmk:para-shape}
    Suppose $G=\GSp_4$ and $M\neq T$. Then, $M \simeq \GL_1 \times \GL_2$ and we let $s\in W_M$ be the unique non-identity element. We have
    \begin{align*}
        \Adm_M(\eta) = \{ t_{\eta}, t_{s(\eta)}, t_{\eta}s\}.
    \end{align*}
    By Lemma \ref{lem:adm-levi}, we can view $\Adm_T(\eta) = \{t_{\eta}\}$ and $s\Adm_T(\eta)s^\mo =\{t_{s(\eta)}\}$ as subsets of $\Adm_M(\eta)$. For each $w^M \in W^M$, $(w^M)^\mo t_\eta s w^M$ as in Remark \ref{rmk:two-components} is the unique element of $(w^M)^\mo \Adm_M(\eta) w^M$ that is not contained in $(w^M)^\mo \Adm_T(\eta) w^M$ or $(sw^M)^\mo \Adm_T(\eta) sw^M$. 
\end{rmk}

\subsection{Types and weights}
\label{sec:types-weights}

In this section, we briefly recall objects in representation theory and Galois theory related to Serre weight conjectures. We refer \cite[\S2.3]{LLLMextremal} and \cite[\S2]{lee_thesis} for more details.

\subsubsection{Serre weights}
\label{sec:SW}
As in \S\ref{sec:not:RG}, we let $G$ be a split group defined over $\Zp$, $\cO_p$ be a finite \'etale $\Z_p$-algebra, $G_0 = \Res_{\cO_p/\Z_p} G_{/\cO_p}$, and  $\un{G} \defeq (G_0)_{/\cO}\cong G_{/\cO}^{\Hom(\cO_p,\cO)}$. We define $\rG\defeq G_0(\Fp)$.  A \emph{Serre weight} (of $\rG$) is an absolutely irreducible $\F$-representation of $\rG$. 

Let $\lambda\in X^*(\un{T})$ be a dominant character.
We write $W(\lambda)$ for the $G_0(\Zp)$-restriction of the (dual) Weyl module $\Ind_{\un{B}}^{\un{G}} w_0 \lambda$. We write $W(\lambda)_{R} \defeq W(\lam)\otimes_{\cO}R$ for an $\cO$-algebra $R$.  
Let $F(\lambda)$ denote the (irreducible) socle of $W(\lam)_{\F}$ viewed as a $\rG$-representation.

We define the set of $p$-restricted weights
\[
X_1(\un{T}) \defeq \left\{\lambda\in X^*(\un{T}),0\leq \langle \lambda,\alpha^\vee\rangle\leq p-1\text{ for all }\alpha\in \un{\Delta}\right\}.
\]
Then, we have the following classification of Serre weights (see \cite[Lemma 9.2.4]{GHS})
\begin{align*}
    \frac{X_1(\ud{T})}{(p-\pi)X^0(\uT)} &\xrightarrow{\sim} \CB{\text{Serre weights of $\rG$}}/\simeq \\ 
    \lam & \mapsto F(\lam).
\end{align*}
We say that $\lambda\in X_1(\un{T})$ is \emph{regular $p$-restricted} if $\langle \lambda,\alpha^\vee\rangle < p-1$ for all $\alpha\in \un{\Delta}$ and say a Serre weight $F(\lambda)$ is \emph{regular} if $\lambda$ is. 
Similarly we say that $F(\lambda)$ is $m$-deep if $\lambda$ is $m$-deep.



We have an equivalence relation on $\tld{\un{W}} \times X^*(\un{T})$  defined by $(\tld{w},\omega) \sim (t_\nu \tld{w},\omega-\nu)$ for all $\nu \in X^0(\un{T})$ (\cite[\S 2.2]{MLM}). 
For $(\tld{w}_1,\omega-\eta)\in \tld{\un{W}}^+_1\times (X^*(\un{T})\cap \un{C}_0)/\sim$, we define the Serre weight $F_{(\tld{w}_1,\omega)}\defeq F(\pi^{-1}(\tld{w}_1)\cdot (\omega-\eta))$ (this only depends on the equivalence class of $(\tld{w}_1,\omega)$).
We call the equivalence class of $(\tld{w}_1,\omega)$ a \emph{lowest alcove presentation} for the Serre weight $F_{(\tld{w}_1,\omega)}$ and note that $F_{(\tld{w}_1,\omega)}$ is $m$-deep if and only if $\omega-\eta$ is $m$-deep in alcove $\un{C}_0$.
(We often implicitly choose a representative for a lowest alcove presentation to make \emph{a priori} sense of an expression, though it is \emph{a posteriori} independent of this choice.)

We also record the following result on the decomposition of $W(\lam)_\F$ for $G=\GSp_4$.

\begin{lemma}[{\cite[\S7]{Jantzen77}}]\label{lem:weyl-decomp}
    Let $G=\GSp_4$ and $\cO_p=\Zp$. For $i\in \{0,1,2,3\}$, let $\lam_i\in C_i$ be a character. We assume that $\lam_i \uparrow \lam_{i+1}$ for $i\in \{0,1,2\}$. Then $W(\lam_0)_\F$ is irreducible, and for $i\in \{1,2,3\}$, we have the exact sequence
    \begin{align*}
        0 \ra F(\lam_i) \ra W(\lam_i)_\F \ra F(\lam_{i-1}) \ra 0.
    \end{align*}
\end{lemma}

\subsubsection{Deligne--Lusztig representations}
To a \emph{good} pair $(s,\mu)\in \un{W}\times X^*(\un{T})$ we attach a Deligne--Lusztig representation $R_s(\mu)$ of $\rG$ defined over $E$  (see \cite[\S 2.2]{LLL} and \cite[Proposition 9.2.1, 9.2.2]{GHS}, where the representation $R_s(\mu)$ is denoted by $R(s,\mu)$). 
We call $(s,\mu-\eta)$ a \emph{lowest alcove presentation} for $R_s(\mu)$ and  say that $R_s(\mu)$ is \emph{$N$-generic} if $\mu-\eta$ is $N$-deep in alcove $\un{C}_0$ for $N\geq 0$
If $\mu-\eta$ is $1$-deep in $\un{C}_0$ then $R_s(\mu)$ is an irreducible representation.
We say that a Deligne--Lusztig representation $R$ is $N$-generic if there exists an isomorphism $R\cong R_s(\mu)$ where $R_s(\mu)$ is \emph{$N$-generic}.

\subsubsection{Tame inertial types}
\label{subsubsec:TIT}
An inertial type (for $K$ over $E$) is the $G^\vee(E)$-conjugacy class of a homomorphism  $\tau:I_{K}\ra G^\vee(E)$ with open kernel and which extends to the Weil group of $G_K$.
An inertial type is \emph{tame} if one (equivalently, any) homomorphism in the conjugacy class factors through the tame quotient of $I_{K}$. 

Let $s\in \un{W}$ and $\mu \in  X^*(\un{T})\cap \un{C}_0$.
Associated with this data, we have a tame inertial type $\tau(s,\mu+\eta)$ defined as in \cite[\S2.4]{MLM} (also, see \cite[\S2.3]{lee_thesis}). 
%
If $N \geq 0$ and $\mu$ is $N$-deep in alcove $\un{C}_0$, the pair $(s,\mu)$ is said to be \emph{an $N$-generic lowest alcove presentation} for the tame inertial type $\tau(s,\mu+\eta)$. 
We say that a tame inertial type is $N$-generic if it admits an $N$-generic lowest alcove presentation. 

If $(s,\mu)$ is a lowest alcove presentation of $\tau$, let $\tld{w}(\tau) \defeq t_{\mu+\eta}s\in\tld{\un{W}}$. (In particular, when writing $\tld{w}(\tau)$, we use an implicit lowest alcove presentation for $\tau$).

Inertial $\F$-types are defined similarly with $E$ replaced by $\F$. 
Tame inertial $\F$-types have analogous notions of lowest alcove presentations and genericity.
If $\taubar$ is a tame inertial $\F$-type we write $[\taubar]$ to denote the tame inertial type over $E$ obtained from $\taubar$ using the Teichm\"uller section $\F^\times \into \cO^\times$.

If $\rhobar$ and $\tau$ are tame inertial types over $\F$ and $E$, respectively, with lowest alcove presentations, we define $\tilw(\rhobar,\tau)\defeq \tilw(\tau)^\mo \tilw(\rhobar)$.

A \textit{type} is a pair $(\lam,\tau)$ of a regular dominant character $\lam\in X^*(\uT)$ and a tame inertial type $\tau$. We say that $(\lam,\tau)$ is $N$-generic if $\tau$ is $N$-generic. 

%


\subsubsection{$L$-parameters}
\label{subsub:Lp}
Recall from \S \ref{sec:not:RG} the finite \'etale $\Qp$-algebra $F_p$.
We adapt the constructions of tame inertial types and the inertial local Langlands above to arbitrary $S_p$.
We assume that $E$ contains the image of any homomorphism $F_p\rightarrow \ovl{\Q}_p$.
Let 
\[
\un{G}^\vee\defeq\prod_{F_p\ra E}G^{\vee}_{/\cO}
\]
be the $\cO$-model of the dual group of $\Res_{F_p/\Qp}(G_{/F_p})$ and ${}^L\un{G}\defeq \un{G}^\vee\rtimes\Gal(E/\Qp)$  the Langlands dual group of $\Res_{F_p/\Qp}(G_{/F_p})$ (where $\Gal(E/\Qp)$ acts on the set $\{F_p\ra E\}$ by post-composition).
An $L$-parameter (over $E$) is a $\un{G}^\vee(E)$-conjugacy class of an $L$-homomorphism, i.e.~of a continuous homomorphism $\rho:G_{\Qp}\ra{}^L\un{G}(E)$ which is compatible with the projection to $\Gal(E/\Qp)$.
An inertial $L$-parameter is a $\un{G}^\vee(E)$-conjugacy class of an homomorphism $\tau:I_{\Qp}\ra\un{G}^\vee(E)$ with open kernel, and which admits an extension to an $L$-homomorphism $G_{\Qp}\ra{}^L\un{G}(E)$.
An inertial $L$-parameter is \emph{tame} if some (equivalently, any) representative in its equivalence class factors through the tame quotient of $I_{\Qp}$. 

Following \cite[Lemmas 9.4.1, 9.4.5]{GHS}, we have a bijection between $L$-parameters (resp.~tame inertial $L$-parameters) and collections of the form $(\rho_v)_{v\in S_p}$ (resp.~of the form $(\tau_v)_{v\in S_p}$) where for all $v\in S_p$ the element $\rho_v:G_{F_v}\ra G^\vee(E)$ is a continuous Galois representation (resp.~the element $\tau_v:I_{F_v}\ra G^\vee(E)$ is a tame inertial type for $F_v$). We have similar bijections for $L$-parameters (resp.~inertial $L$-parameters) over $\F$.

\subsubsection{Inertial local Langlands}
\label{subsub:ILL}
Let $K/\Qp$ be a finite extension with ring of integers $\cO_K$ and residue field $k$. In \cite{GT}, Gan and Takeda established the local Langlands correspondence for $\GSp_4(K)$, which we denote by $\recGT$. Fix once and for all an isomorphism $\iota: \overline{\Q}_p \simeq \C $. We denote by $\mathrm{rec}_{\mathrm{GT},\iota}$ the correspondence over $\overline{\Q}_p$ given by conjugating by $\iota$. We define a normalized local Langlands correspondence by
\begin{align*}
    \mathrm{rec}_{\mathrm{GT},p}(\pi) \defeq \mathrm{rec}_{\mathrm{GT},\iota} (\pi \otimes |\simc|^{3/2})
\end{align*}
for any irreducible smooth admissible $\overline{\Q}_p$-representation $\pi$ of $\GSp_4(K)$.

Let $\tau$ be a tame inertial $L$-parameter with a 1-generic lowest alcove presentation $(s,\mu)$. Then we define the irreducible smooth representation $\sigma(\tau)\defeq R_s(\mu+\eta)$ over $E$, which we view as a representation of $G_0(\Zp)$ via inflation. When $G=\GSp_4$, $F_p=K$, and $\tau$ is a tame inertial type for $K$, $\sig(\tau)$ satisfies the following property (\cite[Theorem 2.4.1]{lee_thesis}): for an irreducible smooth admissible representation $\pi$ of $\GSp_4(K)$ over $E$, if $\sigma(\tau)\subset \pi|_{\GSp_4(\cO_K)}$, then $\mathrm{rec}_{\mathrm{GT},p}(\pi)|_{I_K}\simeq \tau$.

If $(\lam+\eta,\tau)$ is a type, then we write $\sig(\lam,\tau)\defeq W(\lam)\otimes_{\cO} \sig(\tau)$. We denote by $\sig^\circ(\lam,\tau) \subset \sig(\lam,\tau)$ a $\GSp_4(\cO_K)$-stable $\cO$-lattice and $\osig(\lam,\tau)\defeq \sig^\circ(\lam,\tau)\otimes_{\cO} \F$ (we omit $\lam$ when $\lam=0$). Note that the set of Jordan--H\"older factors $\JH(\osig(\lam,\tau))$ is independent of the choice of $\sig^\circ(\lam,\tau)$.



Following \cite[Definition 9.2.5]{GHS}, we define the set of Serre weights associated with a tame inertial $L$-parameter.
\begin{defn}
    Let $\cR$ denote the bijection on regular Serre weights given by $F(\lam) \mapsto F(\tilw_h\cdot \lam)$. If $\rhobar: I_{\Qp} \ra {}^L\un{G}(\F)$ is a 1-generic tame inertial $L$-parameter, we define
    \begin{align*}
        W^?(\rhobar) \defeq \cR(\JH (\osig([\rhobar]))).
    \end{align*}
\end{defn}

\subsubsection{Reduction of Deligne--Lusztig representations}
We recall parameterizations of some sets of Serre weights from \cite[\S2]{MLM}. From now on, we take $G=\GSp_4$. We first define the set of admissible pairs.

\begin{defn}
    We define 
    \begin{align*}
        \AP(\eta) &\defeq \{ (\tilw_1,\tilw_2)\in (\utilW^+_1 \times \utilW^+)/ X^0(\uT) \mid \tilw_1 \uparrow \tilw_h^\mo \tilw_2 \}
        \\
        \AP'(\eta) &\defeq \{ (\tilw_1,\tilw_2)\in (\utilW^+ \times \utilW^+_1)/ X^0(\uT) \mid \tilw_1 \uparrow \tilw_2 \}
    \end{align*}
\end{defn}


\begin{lemma}[{\cite[Propositions 2.3.7 and 2.6.2]{MLM},\cite[Theorem 4.2]{DLR}}]\label{lem:JH-bij}
    Let $m\ge 3$ be an integer. Let $\rhobar$ and $\tau$ be $m$-generic tame inertial $L$-parameters over $\F$ and $E$, respectively.
    \begin{enumerate}
        \item There is a bijection
    \begin{align*}
        F_\tau: \AP(\eta) & \risom \JH(\osig(\tau)) \\
        (\tilw_1,\tilw_2) &\mapsto  F_{(\tilw_1,\omega)}
    \end{align*}
    where $\omega = \tilw(\tau)\tilw_2^\mo(0)$. Moreover, any $\sig\in \JH(\osig(\tau))$ is $(m-3)$-deep.
    \item There is a bijection
    \begin{align*}
        F_{\rhobar}: \AP'(\eta) & \risom W^?(\rhobar) \\
        (\tilw_1,\tilw_2) &\mapsto F_{(\tilw_2,\omega)}
    \end{align*}
    where $\omega = \tilw(\rhobar)\tilw_1^\mo(0)$. Moreover, any $\sig\in W^?(\rhobar)$ is $(m-3)$-deep.
    \end{enumerate}
\end{lemma}

\begin{defn}\label{def:outer}
    Let $\rhobar$ and $\tau$ be $3$-generic tame inertial $L$-parameters over $\F$ and $E$, respectively. For $w\in \uW$, we define \textit{the obvious weight of $\rhobar$ corresponding to $w$} to be $F_{\rhobar}(w) \defeq F_{\rhobar}(w^\diamond,w^\diamond) \in W^?(\rhobar)$. We define $W_\obv(\rhobar)$ to be the set $\{F_{\rhobar}(w) \mid w\in \uW\}$.  Similarly, we define \textit{the outer weight of $\tau$ corresponding to $w$} to be $F_{\tau}(w) \defeq F_\tau(w^\diamond,\tilw_h w^\diamond)$ and  $\JH_{\out}(\osig(\tau)) \defeq \{F_\tau(w) \mid w\in \uW\}$. 
\end{defn}

\begin{rmk}\label{rmk:obv-wt-W-torsor}
     The map $w\mapsto F_{\rhobar}(w)$ defines a bijection between $\uW$ and $W_\obv(\rhobar)$. Thus, we can view $W_\obv(\rhobar)$ as a $\uW$-torsor.
\end{rmk}

\begin{lemma}\label{lem:covering-uparrow}
    Suppose that $\sig_0$ and $\sig$ are Serre weights and $\sig_0$ is 6-deep.  Then, $\sig \in \JH(\ov{R})$ for all $3$-generic Deligne--Lusztig representations $R$ such that $\sig_0 \in \JH(\ov{R})$ if and only if $\sig \uparrow \sig_0$  (i.e.~$\sig=F(\lam)$ and $\sig_0=F(\lam_0)$ for $\lam,\lam_0 \in X_1(\uT)$ such that $\lam \uparrow \lam_0$).
\end{lemma}
\begin{proof}
    See the proof of \cite[Lemma 6.1.6]{lee_thesis}. Note that \loccit~assumes stronger genericity, but we can ease this assumption using \cite[Theorem 4.2]{DLR}.
\end{proof}


\subsection{Hecke algebras and weight cycling}\label{subsub:hecke-alg}  In this subsection, we study integral and mod $p$ Hecke algebras that are necessary for our weight cycling argument.

Let $\sig$ be a Serre weight. We define its Hecke algebra $\cH(\sig) \defeq \End_{\F[G(K)]}(\cind_{G(\cO_K)}^{G(K)} \sig)$. By \cite[Theorem 1.2]{florian-satake}, it is isomorphic to $\F[x_1,x_2,x_0^\pm]$ where $x_i$ is identified with the Hecke operator $\ov{T}_{-\beta_i}$ supported on $G(\cO_K)(-\beta_i(p))G(\cO_K)$ (see \loccit~for detail). 

For $i=1,2$, let $P_i\subset G$ be the parabolic subgroup generated by $B$ and the centralizer of $W$ for $\be_i(p)$. To be precise, $P_1=Q$ and $P_2=S$. We denote by $M_i$ and $N_i$ its Levi factor and unipotent radical and by $\ov{P}_i=M_i\ov{N}_i$ the opposite parabolic. Note that $\sig_{\ov{N}_i(k)} \simeq \sig^{N_{i}(k)}$ and this induces a map $\sig \into \Ind_{\ov{P}_i(k)}^{G(k)} \sig^{N_{i}(k)}$ by Frobenius reciprocity.

\begin{prop}[Weight cycling, {\cite[Proposition 2.7]{FLLMR}}]\label{prop:weight-cyc}
    For $i=1,2$, there is a commutative diagram of $\F[G(K)]$-representations
    \[
    \begin{tikzcd}
         & \cind_{G(\cO_K)}^{G(K)} \Ind_{\ov{P}_i(k)}^{G(k)} \sig^{N_{i}(k)} \arrow[rd, two heads, dashed] & \\
        \cind_{G(\cO_K)}^{G(K)} \sig \arrow[ru, hook] \arrow[rr, "\ov{T}_{-\be_i}"] & & \cind_{G(\cO_K)}^{G(K)} \sig
    \end{tikzcd}
    \]
    where the left diagonal arrow is induced by $\sig \into \Ind_{\ov{P}_i(k)}^{G(k)} \sig^{N_{i}(k)}$.
\end{prop}
We define $C_i(\sig) \defeq (\Ind_{\ov{P}_i(k)}^{G(k)} \sig^{N_{i}(k)} )/ \sig$. By the above proposition, there is a $G(K)$-equivariant surjection
\begin{align*}
        \cyc_i : \cind_{G(\cO_K)}^{G(K)}  C_i(\sig) \onto \cind_{G(\cO_K)}^{G(K)}   \sig/\ov{T}_{-\be_i}.
    \end{align*}

Now we assume that the highest weight of $\sig$ is regular, i.e.~the $T(k)$-character $\sig^{U(k)}$ has the trivial $W$-stablizer. Let $(\tilw,\omega)$ be a lowest alcove presentation of $\sig$. We write $\tilw=t_\nu w$, $\mu \defeq \pi^\mo(\tilw)\cdot(\omega-\eta) \in X^*(\uT)$, and $\tld{\mu} = \sum_{j=1}^f \mu_j p^{f-j}\in X^*(T)$. Note that $\sig^{U(k)} \simeq \tld{\mu}|_{T(k)}$. 

We consider a tame type $\sig(\tau)$ lifting $\sig$. Let $\tau=\tau(e,\mu)$. It has a lowest alcove presentation $(\pi^\mo(w)^\mo w, \omega+\pi^\mo(w)^\mo(\nu-\eta)-\eta)$. Following the explicit construction of $\tau$ and $\sig(\tau)$ from a lowest alcove presentation (cf.~\cite[\S2.4, 2.5]{MLM}), we can and do identify $\tau$ with $[\phi(w_{f-1}^\mo(\tld{\mu}))]$ as $T(\cO)$-valued character using local class field theory and $\sig(\tau)$ with $\Ind_{B(k)}^{G(k)}[w_{f-1}^\mo (\tld{\mu})][1/p]$. Even though $\tau$ is isomorphic to $[\phi(w(\tld{\mu}))]$ for any $w\in W$ as $\GSp_4(\cO)$-valued representations, the chosen presentation affects our later computation in \S\ref{sub:functions}.


The following lemma will be useful for an application of Proposition \ref{prop:weight-cyc}. 

\begin{lemma}[{\cite[Lemma 2.9]{FLLMR}}]\label{lem:JH-bound}
    We have $\JH(C_i(\sig)) \subset \JH(\osig(\tau))$.
\end{lemma}


Let $\sig^\circ(\tau)$ be the $\GSp_4(k)$-stable $\cO$-lattice given by $\Ind_{B(k)}^{\GSp_4(k)}[\tld{\mu}]$. Note that there is no Weyl group element applied to $\tld{\mu}$ here. By Frobenius reciprocity, the cosocle of $\osig(\tau)$ is $\sig$. Let $\Iw\subset \GSp_4(\cO_K)$ be the preimage of $B(k)\subset \GSp_4(k)$. We consider the rational and integral Hecke algebras
\begin{align*}
    \cH(\sig(\tau)) &\defeq \End_{E[G(K)]}(\cind_{G(\cO_K)}^{G(K)} \sig(\tau)) \simeq \End_{E[G(K)]}(\cind_{\Iw}^{G(K)} [\tld{\mu}][1/p])
    \\
    \cH(\sig^\circ(\tau)) &\defeq \End_{\cO[G(K)]}(\cind_{G(\cO_K)}^{G(K)} \sig^\circ(\tau))  \simeq \End_{\cO[G(K)]}(\cind_{\Iw}^{G(K)} [\tld{\mu}]).
\end{align*}
By \cite[Corollary 7.12]{BK98} and \cite[Proposition 2.1]{Dat}, there is a support preserving $E$-algebra isomorphism
\begin{align*}
    t_B: \cH(T,[\tld{\lam}][1/p])  \risom \cH(\sig(\tau))
\end{align*}
where $\cH(T,[\tld{\lam}][1/p])\defeq \End_{E[T(K)]}(\cind_{T(\cO_K)}^{T(K)}[\tld{\lam}][1/p])$. 
Note that we use the version of $t_B$ normalized as in \cite{BK99}. For a cocharacter $\mu$, let $T_\mu \in \cH(\sig(\tau))$ (resp.~$t_\mu\in \cH(T,[\tld{\lam}][1/p])$) be the Hecke operator supported on $\Iw \mu(p)\Iw$ (resp.~on $T(\cO_K)\mu(p)$) that takes $\mu(p)$ to the identity endomorphism of $[\tld{\lam}][1/p]$. Then for dominant $\mu$, $t_B(t_\mu) = \delta_B^{1/2}(\mu(p)) T_\mu$ where $\delta_B$ is the modulus character. It is also independent of the choice of $B$. In other words, for any $w\in W$, the isomorphism
\begin{align*}
    t_B^w:  \cH(T,[w(\tld{\lam})][1/p])  \risom \cH(\sig(\tau))
\end{align*}
satisfies $t_B(t_{\mu}) = t_B^w(t_{w(\mu)})$ for any $\mu \in X_*(T)$.

\begin{prop}\label{prop:hecke-rational}
    For $\mu\in X_*(T)$, define
    \begin{align*}
        T'_{\mu}\defeq q^{-\max_{w\in W}\langle \eta, -w(\mu) \rangle}T_\mu.
    \end{align*}
    Then there is an $\cO$-algebra isomorphism $E[X_*(T)] \ra \cH(\sig(\tau))$ sending $\mu \mapsto T_\mu'$.
\end{prop}
\begin{proof}
    Let $\rho\defeq \frac{1}{2}\sum_{\alpha\in \Phi^+}\alpha$ and $\chi \defeq \eta - \rho$. Since 
    \begin{align*}
        \max_{w\in W}\langle \eta, -w(\mu) \rangle = \max_{w\in W}\langle \rho, w(\mu) \rangle - \langle \chi, \mu \rangle,
    \end{align*}
    it suffices to show that the statement holds for $T'_\mu$ defined using $\langle \rho, w(\mu) \rangle$ in place of $\langle \eta, -w(\mu) \rangle$. For such modified $T'_\mu$, we claim that $t_B(t_\mu) = T_\mu'$. Since $\cH(T,[\tld{\lam}[1/p]])\simeq E[X_*(T)]$ matching $t_\mu$ and $\mu$, this implies the proposition. 
    
    When $\mu$ is dominant, this is obvious because $q^{-\max_{w\in W}\langle \rho, w(\mu) \rangle} = \delta_B^{1/2}(\mu(p))$. In general, write $\mu=\mu_1+\mu_2$ where $\mu_1$ is dominant and $\mu_2$ is antidominant. Then
    \begin{align*}
        t_B(t_{\mu}) &= t_B(t_{\mu_1})t_B(t_{\mu_2}) \\
        &= T_{\mu_1}T_{\mu_2} q^{- \langle \rho, \mu_1 \rangle + \langle \rho, \mu_2 \rangle}
        \\
        &=T_{\mu} q^{- \langle \rho, \mu_1 \rangle + \langle \rho, \mu_2 \rangle}q^{\sum_{\alpha\in \Phi^+} \langle \alpha,\mu_1 \rangle - \max\{\langle \alpha, \mu \rangle,0\}}
    \end{align*}
    where the third equality follows from \cite[Corollary 6.13]{BK98}. 
    Then the claim follows from
    \begin{align*}
        \sum_{\alpha\in \Phi^+} \langle \alpha,\mu_1 \rangle - \max\{\langle \alpha, \mu \rangle,0\} &= \langle 2\rho,\mu_1 \rangle - \frac{1}{2}\sum_{\alpha\in \Phi^+} (\langle \alpha,\mu \rangle + \abs{\langle \alpha,\mu \rangle})
        \\
        &= \langle \rho,\mu_1 \rangle - \langle \rho,\mu_2 \rangle - \max_{w\in W} \{\langle \rho,w(\mu) \rangle\}. \qedhere
    \end{align*}
\end{proof}

\begin{prop}\label{prop:hecke-integral}
    The $\cO$-subalgebra $\cH(\sig^\circ(\tau)) \subset \cH(\sig(\tau))$ is generated by $T_{w(\be_i)}$ for $w\in W$ and $i=1,2$ and $T_{\be_0}^\pm$.
\end{prop}
\begin{proof}
    We can show that the set $\{T_{\mu}\}_{\mu \in X_*(T)}$ is an $\cO$-basis of $\cH(\sig^\circ(\tau))$ as in \cite[Lemma 3.3.6]{lee_satake}. Thus, it suffices to show that $T_{\mu}$ is generated by $T_{w(\be_i)}$ for $w\in W$ and $i=1,2$ and $T_{\be_0}^\pm$. Let $w\in W$ be an element such that $w^\mo(\mu)$ is dominant. Then we can write
    \begin{align*}
        \mu = w( c_1 \beta_1 + c_2 \beta_2) + c_0 \be_0
    \end{align*}
    for some $c_1, c_2 \in \Z_{\ge 0}$ and $c_0\in \Z$. Then
    \begin{align*}
        T_{\mu} = T_{w(\be_1)}^{c_1}T_{w(\be_2)}^{c_2}T_{\be_0}^{c_0}
    \end{align*}
    by \cite[Corollary 6.12]{BK98} and the fact that $t_B$ is independent of the choice of $B$  (see also, \cite[Lemma 3.3.9]{lee_satake}).
\end{proof}

The following result is a generalization of \cite[Theorem 3.5.4]{lee_satake} and can be proven in a similar way.
\begin{prop}\label{prop:hecke-mod-p-red}
    For $\sig$ and $\sig^\circ(\tau)$ as above, there is a natural surjective morphism
    \begin{align*}
        \cR: \cH(\sig^\circ(\tau)) \onto \cH(\sig)
    \end{align*}
    that maps $T_{-\be_i}$ to $\ov{T}_{-\be_i}$.
\end{prop}

%% file: local_model.tex
\section{Local models}\label{sec:localmodels}

\subsection{Background}
We start by recalling potentially crystalline Emerton--Gee stacks, Breuil--Kisin modules, and local models for $\GSp_4$ from \cite{lee_thesis}. Throughout this section, we take $S_p$ to be a singleton. For the rest of the article, we assume $p>2$.

\subsubsection{The Emerton--Gee stack for $\GSp_4$}  
Recall that $K/\Qp$ is a finite unramified extension of degree $f$. Let $R$ be a $p$-adically complete $\cO$-algebra. A \textit{symplectic projective \'etale $\PG$-module with $R$-coefficients} is a triple $(M,N,\alpha)$ where $M$ (resp.~$N$) is a rank 4 (resp.~1) projective \'etale $\PG$-module with $R$-coefficients (in the sense of \cite[Definition 2.7.2]{EGstack}) and $\al: M \simeq M^\vee \otimes N$ is an isomorphism between rank 4 \'etale $\PG$-modules satisfying the alternating condition $(\al^\vee\otimes N)^\mo \circ \al = -1_{M}$. Here, $1_M$ denotes the identity map on $M$. We write $\cX_{\Sym}(R)$ for the groupoid of projective symplectic \'etale $\PG$-modules with $R$-coefficients. 
By Fontaine's equivalence, $\cX_{\Sym}(R)$ is equivalent to the groupoid of continuous representations $\rho:G_K \ra \GSp_4(R)$ when $R$ is finite over $\cO$ (\cite[Corollary 4.1.3]{lee_thesis}). It is known that $\cX_{\Sym}$ is a Noetherian formal algebraic stack over $\Spf \cO$ (Theorem 4.1.5 in \loccit).

Let $\lambda\in X^*(\uT)$ be dominant character and $\tau$ be an inertial type. There exists a closed substack $\cX_{\Sym}^{\lambda,\tau} \subset \cX_{\Sym}$ characterized as the unique closed $\cO$-flat substack such that for a finite extension $\cO'/\cO$, $\cX_{\Sym}^{\lambda,\tau}(\cO') \subset \cX_{\Sym}(\cO')$ is exactly the subgroupoid of lattices in potentially crystalline representations of Hodge--Tate weight $\lambda$ and tame inertial type $\tau$ (\cite[Proposition 4.1.6]{lee_thesis}). When $\lambda$ is regular, $\dim \cX_{\Sym,\F}^{\lambda,\tau}$ is $4f$, whereas when $\lambda$ is not regular, $\dim \cX_{\Sym,\F}^{\lambda,\tau}$ is strictly smaller than $4f$ (see \loccit). We also write $\cX_n^{\lam,\tau}$ for the potentially crystalline Emerton--Gee stack for $\GL_n$ and Hodge--Tate weight $\lam$ and tame inertial type $\tau$.

Let $\cX_{\Sym,\red}\subset \cX_{\Sym}$ be the underlying reduced substack. It is an algebraic stack over $\F$ of dimension $4f$. Moreover, there is a bijection between Serre weights of $\GSp_4(k)$ and irreducible components in $\cX_{\Sym}$ which we denote by $\sigma \mapsto \cC_{\sigma}$ (see \cite[Theorem 4.1.10]{lee_thesis} for more detail). 

\subsubsection{Breuil--Kisin modules}\label{sec:BK}
Let $\tau: I_{K}\ra \GL_n(E)$ be a 1-generic tame inertial type with a lowest alcove presentation $(s,\mu)$. We define $s_\tau \defeq s_0s_1 \cdots s_{f-1} \in W$ and $r$ to be the order of $s_\tau$. 

Let $K'$ be the subfield of $\overline{K}$ unramified of degree $r$ over $K$, $k'$ be its residue field, and $f'=fr$. We fix a $(p^{f'}-1)$-root $\pi'\in \overline{K}$ of $-p$. We write $L'\defeq K'(\pi')$, $\Delta'\defeq \Gal(L'/K')\subset \Delta\defeq \Gal(L'/K)$, and $\cJ' \defeq \Hom(K',E)$. We fix an embedding $\sig'_0: K' \into E$ extending $\sig_0$ and identify $\cJ' \simeq \Z/f'\Z$ by $\sig_{j'}\defeq \sig'_0 \circ \varphi^{-j'} \mapsto j'$. The restriction map $\cJ' \onto \cJ$ is then given by the natural surjection $\Z/f'\Z \onto \Z/f\Z$.

Let $R$ be a $p$-adically complete $\cO$-algebra. We define 
$\fS_{L',R} \defeq (W(k')\otimes_{\Zp}R)\DB{u'}$. We have $\varphi: \fS_{L',R}  \ra \fS_{L',R}$ which acts as Frobenius on $W(k')$, trivially on $R$, and sends $u'$ to $(u')^p$. If $\fM$ is a $\fS_{L',R}$-module, then we have the standard $R\DB{u'}$-linear decomposition $\fM \simeq \oplus_{j'\in \cJ'} \fM\ix{j'}$ induced by the maps $W(k')\otimes_{\Zp}R \ra R$ sending $x\otimes r$ to $\sig_{j'} (x) r$ for $j'\in \cJ'$. 

We define a $\Delta$-action on $\fS_{L',R}$ as follows: if $g\in \Delta'$, $g(u') = \tfrac{g(\pi')}{\pi'}u'$ (which is independent of the choice of $\pi'$) and $g$ acts trivially on the coefficients; if $\sig^f \in \Delta$ is the lift of the $p^f$-Frobenius on $W(k')$ fixing $\pi'$, then $\sig^f$ generates $\Gal(K'/K)$, and it acts naturally on $W(k')$ and trivially on $u'$ and $R$. We define $v \defeq (u')^{p^{f'}-1}$ and $\fS_R \defeq (\fS_{L',R})^{\Delta=1} = (W(k)\otimes_{\Zp}R)\DB{v}$. We define $E(v) \defeq v+p \in W(k)[v]$. Note that $E(v)=E((u')^{{p^{f'}-1}})$ as a polynomial in $K[u']$ is the minimal polynomial of $\pi'$ over $K$.

Let $R$ be a $p$-adically complete $\cO$-algebra and $h\ge 0$ be an integer. We denote by $Y^{[0,h],\tau}_n(R)$ be the groupoid of Breuil--Kisin module over $\fS_{L',R}$ of rank $n$, height in $[0,h]$, and inertial type $\tau$ (\cite[\S5.1]{MLM}). Its objects consist of the following data:
\begin{itemize}
    \item a finitely generated projective $\fS_{L',R}$-module $\fM$ locally free of rank $n$;
    \item $\phi_\fM: \varphi^*(\fM)\ra\fM$ an injective $\fS_{L',R}$-linear map with cokernel killed by $E(v)^h$;
    \item a semilinear $\Delta$-action on $\fM$ commuting with $\phi_{\fM}$ such that Zariski locally on $R$, for each $j'\in \cJ'$, $\fM\ix{j'} \mod u' \simeq \tau^\vee \otimes R$ as $\Delta'$-representations.
\end{itemize}
We often write $\fM$ for an object in $Y^{[0,h],\tau}_n(R)$. 
If $\fM\in Y^{[0,h],\tau}_n(R)$, then we have its ($E(v)^h$-twisted) dual Breuil--Kisin module $\fM^\vee \in  Y^{[0,h],\tau^\vee}_n(R)$ (\cite[Definition 4.2.3]{lee_thesis}). For $i=1,2$, if $\fM_i \in Y_{n_i}^{[0,h_i],\tau_i}(R)$, then we have $\fM_1\otimes \fM_2 \in Y_{n_1n_2}^{[0,h_1+h_2],\tau_1\otimes \tau_2}(R)$ (Definition 4.2.4 in \loccit).

From now on, we let $\tau: I_{K} \ra \GSp_4(E)$ be a tame inertial type. Recall that using the morphism $\std: \GSp_4 \into \GL_4$, we can view $\tau$ as a tame inertial type for $G=\GL_4$. Similarly, $\simc(\tau)$ is a tame inertial type for $G=\GL_1$. Note that $\tau^\vee \otimes \simc(\tau) \simeq \tau$. We assume that $\tau$ is 1-generic and fix a lowest alcove presentation $(s,\mu)$ throughout the section. 

We denote by $\cI$ (resp.~$\cI_1$) be the standard Iwahori (resp.~pro-$p$ Iwahori) group scheme for $\GSp_4$ over $\cO$. If $R$ is an $\cO$-algebra, we have 
\begin{align*}
    \cI(R) &= \{A\in \GSp_4(R\DB{v}) \mid A \mod v \in B(R)\} \\
    \cI_1(R) &= \{A\in \GSp_4(R\DB{v}) \mid A \mod v \in U(R)\}.
\end{align*}

\begin{defn}
    Let $R$ be a $p$-adically complete $\cO$-algebra and $h\ge 0$ be an integer. A \textit{sympletic Breuil--Kisin module over $\fS_{L',R}$ of height in $[0,h]$ and inertial type $\tau$} is a triple $(\fM, \fN, \al)$ where $\fM\in Y_4^{[0,h],\tau}(R)$, $\fN\in Y_1^{0,\simc(\tau)}(R)$, and $\al: \fM\simeq \fM^\vee\otimes \fN$ is an isomorphism satisfying the alternating condition $(\al^\vee\otimes \fN)^\mo \circ \al = -1_\fM$.
\end{defn}

We denote by $Y^{[0,h],\tau}_{\Sym}(R)$ the groupoid of sympletic Breuil--Kisin module over $\fS_{L',R}$ of height in $[0,h]$ and inertial type $\tau$. It is known that $Y^{[0,h],\tau}_{\Sym}$ is a $p$-adic formal algebraic stack (\cite[Proposition 4.2.6]{lee_thesis}). 

Recall that an eigenbasis of $(\fM,\fN,\al)\in Y^{[0,h],\tau}_{\Sym}(R)$ is a pair $(\be,\gamma)$ of eigenbases $\be=(\be\ix{j'})_{j'\in\cJ'}$ of $\fM$ and $\gamma= (\gamma\ix{j'})_{j'\in \cJ'}$ of $\fN$ satisfying compatibility properties (see \cite[Definition 4.2.7]{lee_thesis} for details). For each $j'\in \cJ'$, we define the matrix $A\ix{j'}_{\fM,\be} \in \GSp_{4}(R\DB{v}[1/E(v)])$ as in \S4.2 in \loccit. Then, $A\ix{j'}_{\fM,\be}$ only depends on $j' \mod f$ and is upper triangular modulo $v$, and we often write $A\ix{j}_{\fM,\be}$ instead of $A\ix{j'}_{\fM,\be}$ where $j \equiv j' \mod f$.  

Let $(I\ix{j})_{j\in\cJ}\in \cI^{\cJ}(R)$ and $(A\ix{j})_{j\in \cJ} \in \GSp_4(R\DB{v}[1/E(v)])$. We define the \textit{$(s,\mu)$-twisted $\varphi$-conjugation action} by $\cI^{\cJ}(R)$ on $\GSp_4(R\DB{v}[1/E(v)])$ by the formula
\begin{align*}
    (I\ix{j})_{j\in\cJ} \tadstar (A\ix{j})_{j\in \cJ} \defeq (I\ix{j}A\ix{j} (\Ad(s_j^\mo v^{\mu_j+\eta_j})(\varphi(I\ix{j-1}))^\mo)_{j\in \cJ}.
\end{align*}
When $\varphi$ acts trivially on $I\ix{j}$ (e.g.~if $I\ix{j}\in T^\vee(R)$), then we drop the superscript $\varphi$.

\begin{prop}[{\cite[Proposition 4.2.10]{lee_thesis}}]\label{prop:change-of-basis}
    Let $(\fM,\fN,\al)\in Y^{[0,h],\tau}_{\Sym}(R)$. The set of eigenbases of $(\fM,\fN,\al)$ is a $\cI^{\cJ}(R)$-torsor. If $(\be_1,\gamma_1)$ and $(\be_2,\gamma_2)$ are two eigenbases and differ by $(I\ix{j})_{j\in\cJ}\in \cI^{\cJ}(R)$, then 
    \begin{align*}
        (A\ix{j}_{\fM,\be_1})_{j\in\cJ} = (I\ix{j})_{j\in\cJ} \tadstar (A\ix{j}_{\fM,\be_2})_{j\in \cJ}.
    \end{align*}
\end{prop}

As a consequence, the double coset $\cI(R) A\ix{j}_{\fM,\be} \cI(R)$ is independent of the choice of eigenbasis $(\be,\gamma)$. For $(\fM,\fN,\al)\in Y^{[0,h],\tau}_{\Sym}(\F)$, we define the \textit{shape} of $(\fM,\fN,\al)$ be the unique element $\tilz \in \utilW^\vee$ such that $\cI(\F) A\ix{j}_{\fM,\be} \cI(\F) =\cI(\F) \tilz_{j}\cI(\F)$ for each $j \in \cJ$.

Let $\lam \in X_*(\uT^\vee)$ be a dominant cocharacter such that $\std(\lam)\subset ([0,h]^{4})^\cJ$. There exists a closed substack $Y^{\le\lam,\tau}_{\Sym}\subset Y^{[0,h],\tau}_{\Sym}$ characterized as the unique closed $\cO$-flat substack such that for a finite extension $\cO'/\cO$, $Y^{\le\lam,\tau}_{\Sym}(\cO') \subset Y^{[0,h],\tau}_{\Sym}(\cO')$ is precisely the subgroupoid of $(\fM,\fN,\al)$ such that $A_{\fM,\be}\ix{j'}$ has elementary divisor $E(v)^{\std(\lam_{j'})}$ for each $j'\in \cJ'$ and some (equivalently, any) eigenbasis $(\be,\gamma)$ (see \cite[\S4.2]{lee_thesis}).

Finally, we explain the connection to potentially crystalline stacks. For simplicity, we only consider the cocharacter $\eta$. 
We define $\cX^{\le\eta,\tau}_{\Sym}\defeq \cup_{\lam}\cX^{\lam,\tau}_{\Sym}$ where the union takes over the set of dominant cocharacters $\lam\in X_*(\uT^\vee)$ such that $\lam \le \eta$. Note that any $\lam< \eta$ is irregular, and the substack $\cX^{\eta,\tau}_{\Sym}\subset \cX^{\le\eta,\tau}_{\Sym}$ can be identified with the maximal closed substack whose base change to $\Spec \F$ is of pure dimension $4f$. By \cite[Proposition 4.4.2]{lee_thesis}, if $\tau$ is 5-generic, there is a closed immersion $\cX^{\le\eta,\tau}_{\Sym} \into Y_{\Sym}^{\le\eta,\tau}$. Moreover, its image is the $\cO$-flat closed substack $Y^{\le\eta,\tau,\nbl_\infty}_{\Sym} \subset Y_{\Sym}^{\le\eta,\tau}$ uniquely characterized by the following property: for a $p$-adically complete topologically finite type flat $\cO$-algebra $R$, a morphism $\Spf R \ra Y^{\le\eta,\tau}_{\Sym}$ corresponding to $(\fM,\fN,\al)\in Y^{\le\eta,\tau}_{\Sym}(R)$ factors through $Y^{\le\eta,\tau,\nbl_\infty}_{\Sym}$ if and only if the ideal $I_{\fM,\nbl_\infty}\subset R$ constructed in \cite[Proposition 7.1.6]{MLM} is 0.

\subsubsection{Local models}\label{sub:localmodels}
Let $\cG$ be the Neron blowup of $\GSp_{4/\cO}$ in $B_{/\cO}$ along $\F$ defined in \cite[Definition 3.1]{MRR}. If $R$ is an $\cO$-algebra in which $\varpi$ is regular, then we have
\begin{align*}
    \cG(R) = \{A\in \GSp_4(R) \mid A \mod \varpi \in B(R/\varpi) \}.
\end{align*}

Let $a<b$ be integers. We consider the ind-group-schemes $L\cG$, $L^{[a,b]}\cG(R)$, and $L^+\cG$ over $\cO$ defined in \cite[\S3.1]{lee_thesis}. For any Noetherian $\cO$-algebra $R$, we have
\begin{align*}
    L\cG(R) &= \{ A\in \GSp_4(R\DP{v+p}) \mid A \mod v \in B(R\DP{v+p}/v) \} \\
    L^{[a,b]}\cG(R) &= \{ A\in L\cG(R) \mid (v+p)^{-a} A, (v+p)^b A^\mo \in \mathrm{Mat}_4 (R\DB{v+p}) \} \\
    L^+\cG(R) &= \{ A\in \GSp_4(R\DB{v+p}) \mid A \mod v \in B(R) \}. 
\end{align*}
Note that if $R$ is $p$-adically complete, $R\DB{v+p} = R\DB{v}$.

We define $\Gr_{\cG}$ (resp.~$\Gr^{[a,b]}_\cG$) to be the fpqc quotient sheaf $[L^+\cG \backslash L\cG]$ (resp.~$[L^+\cG \backslash L^{[a,b]}\cG]$). By \cite[Proposition 6.5]{PZ}, it is representable by an ind-projective ind-scheme (resp.~by a projective scheme). Moreover, $\Gr_{\cG,E} \defeq \Gr_{\cG}\times_\cO E$ is the usual affine Grassmannian for the group $\GSp_4$ over $E$, and $\Gr_{\cG,\F}\defeq \Gr_{\cG}\times_\cO \F$ is the affine flag variety for $\GSp_4$ over $\F$. For a dominant cocharacter $\lam\in X^*(T)$, let $S_E^\circ(\lam)$ be the Schubert cell in $\Gr_{\cG,E}$ given by $L^+\cG \backslash L^+\cG (v+p)^{\lam} L^+\cG$. We define $M(\le\lam)$ to be the closure of $S^\circ_E(\lam)$ in $\Gr_{\cG}$. This is the Pappas--Zhu local model associated with the group $\GSp_4$, the conjugacy class $\lam$, and the Iwahori subgroup $\cI$. It is known that $M(\le\lam)$ is a projective variety over $\cO$ (see \cite[\S7.1]{PZ}). If $\cT(\lam)\subset [a,b]^4$, then $M(\le\lam)\subset \Gr^{[a,b]}_\cG$.

Let $\tilz=wt_\nu \in \tilW^\vee$. We define $U(\tilz)$ a subfunctor of $L\cG$ given by
{\small\begin{align*}
    U(\tilz)(R) = \CB{g \in L\cG(R) \middle\vert \begin{array}{cl}
        \ov{g}(v+p)^{-\nu} \in \GSp_4(R[\frac{1}{v+p}]),    \\ 
         \ov{g}(v+p)^{-\nu}w^\mo \mod \frac{1}{v+p} \in \ov{U}(R), \\
         \ov{g}(v+p)^{-\nu} \mod \frac{v}{v+p} \in B(R[\frac{1}{v+p}]/(\frac{v}{v+p}))
    \end{array}}.
\end{align*}}
We also define $U^{[a,b]}(\tilz) \defeq U(\tilz)\cap L^{[a,b]}\cG$. It follows from \cite[Proposition 3.1.4 and Corollary 3.1.5]{lee_thesis} that $U^{[a,b]}(\tilz)$ is an affine scheme and the natural map $U^{[a,b]}(\tilz) \into \Gr_{\cG}^{[a,b]}$ is an open immersion. If $\lam \in X^*(T)$ such that $\cT(\lam)\subset [a,b]^4$,  we define $U(\tilz,\le\lam) \defeq U^{[a,b]}(\tilz)\cap M(\le\lam)$. Then, $\{U(\tilz,\le\lam)\}_{\tilz\in \Adm^\vee(\lam)}$ is an open cover of $M(\le\lam)$ (this essentially follows from \cite[Theorem 9.3]{PZ}; see also the proof of \cite[Corollary 4.2.18]{lee_thesis}).

We denote by $L^+\cM$ the Lie algebra of $L^+\cG$. For a Noetherian $\cO$-algebra $R$, we have
\begin{align*}
    L^+\cM(R) = \{A \in \Lie \GSp_4(R\DB{v+p}) \mid A \mod v \in B(R)\}.
\end{align*}

Let $\bfa=(\bfa_1,\bfa_2,\bfa_3) \in \cO^3$. We often view $\bfa$ as an element in $X_*(T^\vee)\otimes_\Z \cO$ and write $\std(\bfa) \defeq (\bfa_1,\bfa_2,\bfa_3-\bfa_2, \bfa_3-\bfa_1)$. We define a subfunctor $L\cG^{\nbl_\bfa}\subset L\cG$ by
\begin{align*}
   L\cG^{\nbl_\bfa}(R) \defeq \{ A \in L\cG(R) \mid   v\frac{dA}{dv}A^\mo + A \Diag (\std(\bfa)) A^\mo \in \frac{1}{v+p} L^+\cM(R)\}
\end{align*}
for any $\cO$-algebra $R$. Note that $L\cG^{\nbl_\bfa}$ is stable under left-multiplication by $L^+\cG$. We also define $\Gr^{\nbl_\bfa}_\cG \defeq [L^+\cG \backslash L\cG^{\nbl_\bfa}]$. 

We define $M^{\nv}(\le\lam,\nbl_\bfa) \defeq M(\le\lam)\cap \Gr_{\cG}^{\nbl_\bfa}$ and  $U^{\nv}(\tilz,\le\lam,\nbl_\bfa) \defeq U(\tilz,\le\lam)\cap \Gr_{\cG}^{\nbl_\bfa}$. 
More explicitly, let $R = \cO(U(\tilz,\le\lam))$ and $A\in L\cG(R)$ be the universal matrix corresponding to the morphism $U(\tilz,\le\lam) \ra L\cG$. Then $U^{\nv}(\tilz,\le\lam,\nbl_\bfa) $ is the locus given by the condition 
\begin{align}\label{eqn:alg-monodromy}
    v\frac{dA}{dv}A^\mo + A \Diag (\std(\bfa)) A^\mo \in \frac{1}{v+p} L^+\cM(R).
\end{align}

\begin{rmk}[{cf.~\cite[Remark 2.25]{LLLM}}]\label{rmk:Iwahori-noramlizer}
    Suppose that $\tilz \in \Adm^\vee(\eta)$ and $\delta\in \Omega$. If we write $\delta^* = w t_\nu$, then the matrix $w v^\nu$ normalizes $L^+\cG$. Moreover, for $\bfa\in \cO^3$, we have
    \begin{align*}
        w v^\nu U^{\nv}(\tilz,\le\eta,\nbl_\bfa) (w v^\nu)^\mo \simeq U^{\nv}(\delta\tilz\delta^\mo,\le\eta,\nbl_{w(\bfa)}). 
    \end{align*}
\end{rmk}

\begin{rmk}
    There is yet another closed subscheme $U(\tilz,\lam,\nbl_{\bfa})\subset U(\tilz,\le\lam)$ (see \cite[\S3.3]{lee_thesis}). In the case of our interest (Theorem \ref{thm:col-one-nbhd} below where we consider $\lam=\eta$ and particular $\tilz$), it turns out that $U(\tilz,\lam,\nbl_{\bfa})$ is equal to $U^{\nv}(\tilz,\le\lam,\nbl_\bfa)$. For general $\tilz$, $U(\tilz,\lam,\nbl_{\bfa})$ is contained in $U^{\nv}(\tilz,\le\lam,\nbl_\bfa)$ and can be identified with the maximal $\cO$-flat top dimensional subscheme. 
\end{rmk}

\subsubsection{Special fiber of local models}

Recall that $L^+\cG_\F = \cI_\F$ and $\Gr_{\cG,\F}$ is the affine flag variety for $\GSp_4$ over $\F$. We define a $T^\vee_\F$-torsor
\begin{align*}
    \tld{\Gr}_{\cG,\F} \defeq [\cI_{1,\F} \backslash L\cG_\F] \ra \Gr_{\cG,\F}.
\end{align*}
If $X \subset \Gr_{\cG,\F}$ is a subscheme, then we write $\tld{X} \ra X$ for the $T^\vee_\F$-torsor induced by the previous morphism and similarly for subschemes of a product of copies of $\Gr_{\cG,\F}$.

For each $\tilz\in \tilW^\vee$, let $S_\F^\circ(\tilz)$ be the affine Schubert cell given by $\cI_\F \backslash \cI_\F \tilz \cI_\F \subset \Gr_{\cG,\F}$ and let $S_\F(\tilz)$ be its closure. We also define $S_\F^\circ(\tilz)^{\nbl_\bfa}\defeq S_\F^\circ(\tilz)\cap \Gr^{\nbl_\bfa}_{\cG,\F}$ and $S_\F^{\nbl_\bfa}(\tilz)$ be its closure.

\begin{prop}[{\cite[Theorem 3.5.1]{lee_thesis}}]\label{prop:schubert-dim}
    Let $\tilz \in \Adm^\vee(\eta)$. Suppose that $\bfa\in \cO^3$ is 3-generic. Then
    \begin{align*}
      \dim_\F S_\F^{\nbl_{\bfa}}(\tilz) =  4 - \# \{\al \in \Phi^+ \mid \tilz^*(A_0) \subset H_{\alpha}^{(0,1)}  \}.
    \end{align*}
    In particular, $S_\F^{\nbl_{\bfa}}(\tilz)$ is $4$-dimensional if and only if $\tilz\in \Adm^{\vee,\reg}(\eta)$.
\end{prop}



Note that we have a $T_\F^{\vee}$-action on $M(\le\eta)_\F$ given by right-multiplication. Then $T^{\vee}_\F$-fixed point set in $M(\le\eta)_\F$ is given by the image of $\Adm^\vee(\eta)$.

\begin{prop}\label{prop:T-fixed-points}
    Suppose that $\bfa$ is 3-generic. Let $(\tilw_1,\tilw_2) \in \AP'(\eta)$. Then, $T^{\vee}_\F$-fixed points in $S^{\nbl_\bfa}((\tilw_2^\mo\tilw_h^\mo w_0 \tilw_1)^*)$ is given by
    \begin{align*}
        S^{\nbl_\bfa}((\tilw_2^\mo\tilw_h^\mo w_0 \tilw_1)^*)^{T^{\vee}_\F} = \{(\tilw_2^\mo \tilw_h^\mo \tilw)^* \mid  \tilw \in \utilW, \ \tilw \le w_0\tilw_1 \}.
    \end{align*}
\end{prop}

\begin{proof}
    Note that $S^{\nbl_\bfa}((\tilw_2^\mo\tilw_h^\mo w_0 \tilw_1)^*) = S^{\nbl_{(\tilw_2^\mo\tilw_h^\mo)^*(\bfa)}}((w_0 \tilw_1)^*)$ by  \cite[Remark 3.5.6]{lee_thesis}.  Then the claim follows from Theorem A.4 in \loccit.
\end{proof}

\subsubsection{Connections to potentially crystalline Emerton--Gee stacks}\label{sec:back-to-EGstack}
In this section, we let $\cJ= \Hom_{\Qp}(K,E)$ and assume that $\tau$ is 5-generic. For simplicity, we take $\lam=\eta$. We define $M_\cJ(\le \eta)\defeq \prod_{j\in \cJ} M(\le\eta)$, $U(\tilz)\defeq \prod_{j\in \cJ}U(\tilz_j)$ for $\tilz= (\tilz_j)_{j\in \cJ}\in \utilW^\vee$, and similarly for $U^{[a,b]}(\tilz)$ and $U(\tilz,\le\eta)$. We also define $\tld{U}(\tilz)\defeq T^{\vee,\cJ} \times U(\tilz)$ which we consider as a subfunctor of $L\cG^{\cJ}$ using the multiplication map. Similarly, we define $\tld{U}^{[a,b]}(\tilz)$ and $\tld{U}(\tilz,\le\eta)$.

By Proposition \ref{prop:change-of-basis}, we have an isomorphism $Y^{[0,3],\tau}_{\Sym} \simeq [L^{[0,3]}\cG^{\cJ} \tadslash L^+\cG]^{\pcp}$ where the quotient is taken with respect to $\tadstar$-action.    Let $Y^{[0,3],\tau}_{\Sym}(\tilz)\subset Y^{[0,3],\tau}_{\Sym}$ be the open substack obtained by taking the image of the composition
\begin{align*}
    \tld{U}^{[0,3]}(\tilz)^{\pcp} \ra [L^{[0,3]}\cG^{\cJ} \tadslash L^+\cG]^{\pcp} \simeq Y^{[0,3],\tau}_{\Sym}.
\end{align*}
We also define $Y^{\le\eta,\tau}_{\Sym}(\tilz) \defeq Y^{[0,3],\tau}_{\Sym}(\tilz) \cap Y^{\le\eta,\tau}_{\Sym}$. 
By \cite[Theorem 4.2.17]{lee_thesis}, the above morphism induces a $T^{\vee,\cJ}$-torsor $\tld{U}(\tilz,\le\eta)^{\pcp} \ra Y^{\le\eta,\tau}_{\Sym}(\tilz)$. 

The open substack $Y^{[0,3],\tau}_{\Sym}(\tilz)$ is uniquely characterized by the following property: for a $p$-adically complete $\cO$-algebra $R$ and $\fM \in Y^{[0,3],\tau}_{\Sym}(R)$, $\fM \in Y^{[0,3],\tau}_{\Sym}(\tilz)(R)$ if and only if $\fM$ admits an eigenbasis $(\be,\gamma)$ Zariski locally on $R$ such that $(A_{\fM,\be}\ix{j})_{j\in \cJ}\in \tld{U}^{[0,3]}(\tilz)(R)$. Such basis is called a $\tilz$-gauge basis; see \cite[Definition 4.2.14]{lee_thesis}.

Recall from \S\ref{sec:BK} the closed immersion $\cX_{\Sym}^{\le\eta,\tau}\into Y^{\le\eta,\tau}_{\Sym}$. We define $\cX^{\le\eta,\tau}_{\Sym}(\tilz) \defeq  \cX^{\le\eta,\tau}_{\Sym} \cap Y^{\le\eta,\tau}_{\Sym}(\tilz)$ and a closed $p$-adic formal subscheme $\tld{U}(\tilz,\le\eta,\nbl_\infty)\subset \tld{U}(\tilz,\le\eta)^{\pcp}$ by the following pullback diagram
\[
\begin{tikzcd}
    \tld{U}(\tilz,\le\eta,\nbl_\infty) \arrow[r, hook] \arrow[d] & \tld{U}(\tilz,\le\eta)^{\pcp} \arrow[d]  \\
    \cX^{\le\eta,\tau}_{\Sym}(\tilz) \arrow[r, hook] & Y^{\le\eta,\tau}_{\Sym}(\tilz).
\end{tikzcd}
\]
Let $\tld{U}_\reg(\tilz,\le\eta,\nbl_\infty) \subset \tld{U}(\tilz,\le\eta,\nbl_\infty)$ be the closed $p$-adic formal subscheme given by the union of $7f$-dimensional irreducible components. The left vertical arrow of the above diagram induces a $T^{\vee,\cJ}$-torsor
\begin{align*}
    \tld{U}_\reg(\tilz,\le\eta,\nbl_\infty) \onto \cX_{\Sym}^{\eta,\tau}(\tilz).
\end{align*}
More precisely, we have an isomorphism $\cX^{\eta,\tau}_{\Sym}(\tilz)  \simeq [\tld{U}_\reg(\tilz,\le\eta,\nbl_\infty)  /_{(s,\mu)} T^{\vee,\cJ}]$ (cf.~\cite[Theorem 4.4.3(1)]{lee_thesis}).

\subsubsection{Approximating monodromy condition}
We define a tuple $\bfa_\tau = (\bfa_{\tau,j})_{j\in \cJ} \in (\cO^3)^\cJ$ attached to $\tau$ by $\bfa_{\tau,j} \defeq (s'_{\mathrm{or},j})^\mo(\bfa')\ix{j}/(1-p^{rf})$ (see \cite[\S2.3]{lee_thesis} for the definition of  $(\bfa')\ix{j}$) . Note that if $\tau$ is $m$-generic, then $\bfa_\tau$ is also $m$-generic. Let $R=\cO(\tld{U}(\tilz,\le\eta))$ so that $\tld{U}(\tilz,\le\eta)^{\pcp} = \Spf R^{\pcp}$. As explained at the end of \S\ref{sec:BK}, the morphism $\tld{U}(\tilz,\le\eta)^{\pcp} \ra Y^{\le\eta,\tau}_{\Sym}(\tilz)$ corresponding to $(\fM,\fN,\al)\in Y^{\le\eta,\tau}_{\Sym}(R^{\pcp})$ induces an ideal $I_{\fM,\nbl_\infty}\subset R^{\pcp}$, and $\tld{U}(\tilz,\le\eta,\nbl_\infty) = \Spf R^{\pcp}/I_{\fM,\nbl_\infty}$. Let $A=(A\ix{j})_{j\in\cJ} \in L\cG^\cJ(R)$ be the tuple of matrices corresponding to $\tld{U}(\tilz,\le\eta) \ra L\cG^\cJ$. Then, we have the closed subscheme $\tld{U}^{\nv}(\tilz,\le\eta,\nbl_{\bfa_\tau})\subset \tld{U}(\tilz,\le\eta)$ given by the condition \eqref{eqn:alg-monodromy} on $A$.
\begin{prop}\label{rmk:monodromy-approx}
    Recall that $\tau$ is $5$-generic. In the above notations, we have $\tld{U}(\tilz,\le\eta,\nbl_\infty)\times_\cO \cO/p^2 \subset \tld{U}^{\nv}(\tilz,\le\eta,\nbl_{\bfa_\tau})$ as subschemes of $\tld{U}(\tilz,\le\eta)$.
\end{prop}
\begin{proof}
    This follows from \cite[Proposition 7.1.10]{MLM}.
\end{proof}

\subsubsection{Irreducible components in potentially crystalline Emerton--Gee stacks}
Recall that $U(\tilz,\le\eta)_\F$ for $\tilz \in \Adm^\vee(\eta)$ forms an open cover of $M_\cJ(\le\eta)_\F$. By taking pullback long the $T^{\vee,\cJ}$-torsor $\tld{\Gr}_{\cG,\F} \ra \Gr_{\cG,\F}$, we get an open cover $(\tld{U}(\tilz,\le\eta)_\F)_{\tilz\in \Adm^\vee(\eta)}$ of $\tld{M}_{\cJ}(\le\eta)_\F$. 
 
 By \cite[Proposition 4.2.13]{lee_thesis}, we can glue the morphisms $\tld{U}(\tilz,\le\eta)_\F \ra Y^{\le\eta,\tau}_{\Sym,\F}(\tilz)$ into a $T^{\vee,\cJ}$-torsor
\begin{align*}
    \tld{M}_{\cJ}(\le\eta)_\F \ra Y^{\le\eta,\tau}_{\Sym,\F}.
\end{align*}
We define a $T^{\vee,\cJ}$-torsor $\tld{\cX}_{\Sym,\F}^{\le\eta,\tau} \ra \cX_{\Sym,\F}^{\le\eta,\tau}$ by the following pullback diagram
\begin{equation}\label{eqn:modp-local-model}
    \begin{tikzcd}
    \tld{\cX}_{\Sym,\F}^{\le\eta,\tau}  \arrow[d] \arrow[r, hook] & \tld{M}_\cJ(\le\eta)_\F \arrow[d]
    \\
    \cX_{\Sym,\F}^{\le\eta,\tau} \arrow[r, hook] & Y^{\le\eta,\tau}_{\Sym,\F}. 
\end{tikzcd}
\end{equation}
We define $\tld{M}^{\nv}_\cJ(\le\eta,\nbl_{\bfa_\tau})_\F \defeq \tld{M}_\cJ(\le\eta)_\F \cap \tld{\Gr}_{\cG,\F}^{\nbl_{\bfa_\tau}}$ a closed subscheme of $\tld{M}_\cJ(\le\eta)_\F$. 
It follows from Proposition \ref{rmk:monodromy-approx} that the top horizontal arrow factors through $\tld{M}^{\nv}_\cJ(\le\eta,\nbl_{\bfa_\tau})_\F$. 

For $\tilz \in \utilW^\vee$, we define $S^{\nbl_{\bfa_\tau}}(\tilz) \defeq \prod_{j\in \cJ} S^{\nbl_{\bfa_{\tau,j}}}(\tilz_j) \subset \Gr_{\cG,\F}$ and a $T^{\vee,\cJ}$-torsor $\tld{S}^{\nbl_{\bfa_\tau}}(\tilz) 
 \ra S^{\nbl_{\bfa_\tau}}(\tilz)$. Then $\tld{S}^{\nbl_{\bfa_\tau}}(\tilz) $ is contained in $\tld{M}^{\nv}_\cJ(\le\eta, \nbl_{\bfa_\tau})_\F$ if and only if $\tilz \in \Adm^\vee(\eta)$. It is at most $7f$-dimensional and is $7f$-dimensional if and only if $\tilz\in \Adm^{\vee,\reg}(\eta)$ \cite[Corollary 3.5.3]{lee_thesis}.

Note that $\Irr_{4f}(\cX_{\Sym,\F}^{\eta,\tau}) = \Irr_{4f}(\cX_{\Sym,\F}^{\le\eta,\tau})$. We have the following inclusion obtained by the diagram \eqref{eqn:modp-local-model}
\begin{align*}
    \Irr_{4f}(\cX_{\Sym,\F}^{\eta,\tau}) \into \Irr_{7f}(\tld{M}^{\nv}_\cJ(\le\eta,\nbl_{\bfa_\tau})_\F).
\end{align*}

The following result is \cite[Theorem 4.5.4]{lee_thesis}. Note that \loccit~requires 6-genericity of $\tau$ only to apply the generic decomposition result for $\sig(\tau)$, which is now proven under 3-genericity by \cite{DLR} (see Lemma \ref{lem:JH-bij}).

\begin{prop}\label{prop:C_sig-T-torsor}
    Suppose that $\tau$ is 5-generic. Let $\cC_{\sig}\subset \cX^{\eta,\tau}_{\Sym,\F}$ be an irreducible component, $\tld{S}^{\nbl_{\bfa_\tau}}(\tilz)$ be the corresponding $7f$-dimensional irreducible component of $\tld{M}_\cJ^{\nv}(\le\eta,\nbl_{\bfa_\tau})_\F$ as in Proposition \ref{prop:schubert-dim}, and $(\tilw_1,\tilw_2)\in \AP(\eta)$ be the element such that $\tilz^*= \tilw_2^\mo w_0 \tilw_1$. Then $\sig = F_{\tau}(\tilw_1,\tilw_2)\in \JH(\osig(\tau))$ and the map $\tld{S}^{\nbl_{\bfa_\tau}}(\tilz) \ra \cC_{\sig}$ induced by \eqref{eqn:modp-local-model} is a $T^{\vee,\cJ}$-torsor.
\end{prop}

\begin{rmk}
    The previous Proposition shows that the set $\Irr_{4f}(\cX^{\eta,\tau}_{\Sym,\F})$ is contained in $\{\cC_\sig \mid \sig \in \JH(\osig(\tau))\}$. Following the argument in \cite[Theorem 4.5.2]{lee_thesis} together with the stronger weight elimination result (Proposition \ref{prop:weight-elim}), we can show that this inclusion is a bijection when $\tau$ is 12-generic. We will not need this result.
\end{rmk}

\subsubsection{Closed points and torus fixed points}\label{sec:closed-pt}

Recall from \cite[Theorem 6.6.3(3)]{EGstack} (which easily  generalizes to $\GSp_4$) that $\rhobar\in \cX^{\le\eta,\tau}_{\Sym}(\F)$ is a closed point if and only if the corresponding $\rhobar:G_K \ra \GSp_4(\F)$ is semisimple. Note that all semisimple $\rhobar$ are tame, and tame $\rhobar$ that are not semisimple are twists of unramified representations. Since we only consider $1$-generic $\rhobar$, they are never unramified. Let $\rhobar \in \cX^{\le\eta,\tau}_{\Sym,\F}(\F)$ be a closed point  and $(\fM,\fN,\al)\in Y^{\le\eta,\tau}_{\Sym,\F}(\F)$ be its image. Then the image in $M_\cJ(\eta)_\F$ of the preimage of $(\fM,\fN,\al)$ in $\tld{M}_\cJ(\le\eta)$ is given by the $T^{\vee,\cJ}$-fixed point $\tilw(\rhobar,\tau)^*$ (see \cite[Remark 5.5.6 and Proposition 5.5.7]{MLM}). By combining Propositions \ref{prop:T-fixed-points} and \ref{prop:C_sig-T-torsor}, we obtain the following (cf.~\cite[Corollary 3.7.2]{lee_thesis}).

\begin{cor}\label{cor:geom-SW}
    Let $\tau$ be a 5-generic tame inertial type and $\sig\in \JH(\osig(\tau))$ be a 3-deep Serre weight. Then, for a 3-generic tame $L$-parameter $\rhobar$, $\rhobar \in \cC_{\sig}$ if and only if $\sig \in W^?(\rhobar|_{I_{\Qp}})$. 
\end{cor}

\subsection{Colength one loci}

We set $\cJ = \Hom_{\Qp}(K, E)$. We define $X_p\defeq \Spec \cO[x,y]/(xy-p)$. The following is the main result of this subsection.

\begin{thm}\label{thm:col-one-nbhd}
    Let $\tau$ be a $6$-generic tame inertial type. For each $j\in \cJ$, we choose $\tilz_j$ that is either
    \begin{enumerate}
        \item in $ w^\mo \Adm_T^\vee(\eta) w$ for some $w\in W$ (the extremal case); or
        \item the unique colength one element in $(w^M)^\mo \Adm_M^\vee(\eta)w^M$ for a proper Levi subgroup $M\not\simeq T$ of $\GSp_4$ and $w^M\in W^M$ (the irregular colength one case); or
        \item $(\tilw_2^\mo \tilw_h^\mo w_0 \tilw_1)^*$ for $(\tilw_1,\tilw_2)\in \AP'(\eta)$ such that $\tilw_1 \in \Omega$ and $\tilw_2 C_0 = C_1$ (the regular colength one case).
    \end{enumerate}
    Let $\cJ_1$ (resp.~$\cJ_2$) be the subsets of $\cJ$ consisting of $j\in \cJ$ such that $\tilz_j$ is in case (1) (resp.~in case (2) or (3)). Then we have an isomorphism
    \begin{align*}
        \tld{U}(\tilz,\le\eta,\nbl_\infty) \simeq  (T^{\vee,\cJ} \times_\cO X_1^{\cJ_1} \times_\cO X_2^{\cJ_2})^{\wedge_p}
    \end{align*}
    where $X_1 = \A^4$ and $X_2 = X_p\times \A^3$. In particular, $\tld{U}(\tilz,\le\eta,\nbl_\infty) = \tld{U}_{\reg}(\tilz,\le\eta,\nbl_\infty)$.

\end{thm}

See Remark \ref{rmk:para-shape} for the discussion on the shapes in items (1) and (2) above. The shape in item (3) is discussed in \S\ref{sec:reg-col-one} and will be used only in the proof of the modularity lifting theorem (Theorem \ref{thm:MLT}). 
To prove Theorem \ref{thm:col-one-nbhd}, we first compute $U^{\nv}(\tilz_j,\le\eta_j, \nbl_{\bfa_{\tau,j}}$) for $\tilz_j$ for each $j\in \cJ$.  Then, the claim will follow almost immediately from Proposition \ref{rmk:monodromy-approx}. It is possible to compute $\tld{U}(\tilz,\le\eta,\nbl_\infty)$ directly, but we found it more convenient to argue using the one with $\nbl_{\bfa_\tau}$.



\subsubsection{Parabolic loci}
We temporarily take $\cJ = \{*\}$. Let $P\subset \GSp_4$ be a proper standard parabolic subgroup with its Levi subgroup $M$ and unipotent radical $N$. We denote by $\overline{N}$ the unipotent radical of the parabolic subgroup opposite to $P$. Let $M^\vee \subset \GSp_4^\vee$ be the dual Levi subgroup.  Let $\tilz \in \Adm_M^\vee(\eta)$ and $w^M \in W^M$. By Lemma \ref{lem:adm-levi}, we can view $(w^M)^\mo \tilz w^M$ as an element of $\Adm^\vee(\eta)$.

Note that $M$ is isomorphic to either $\GL_1^3$ or $\GL_2 \times \GL_1$. Let us write $M=\prod_{i\in I}\GL_i$ where $I$ is one of the multisets $\{1,1,1\}$ and $\{1,2\}$. 
This allows us to view $\eta$ as a tuple $(\chi_i)_{i\in I}$ where $\chi_i$ is a cocharacter for $\GL_i$, and similarly view $\tilz \in \Adm_M^\vee(\eta)$ as a tuple $(\tilz_i)_{i\in I}$ where $\tilz_i \in \Adm_{\GL_i}^\vee(\chi_i)$. We define
\begin{align*}
    U_M(\tilz,\le\eta) \defeq \prod_{i\in I} U_{\GL_i}(\tilz_i,\le\chi_i)
\end{align*}
where $U_{\GL_i}(\tilz_i,\le\chi_i)$ is the open chart of the Pappas--Zhu local model for $\GL_i$ defined in \cite[Equation (5.9)]{MLM}. The following proposition shows that $U((w^M)^\mo \tilz w^M,\le\eta)$ naturally attains a parabolic structure.

\begin{prop}[{cf.~\cite[Proposition 4.2.5]{LLLMextremal}}]\label{prop:parabolic}
    In addition to the above notation, we denote $\tilz =z t_{\nu}$. There is a morphism 
    \begin{align*}
        U((w^M)^\mo \tilz w^M,\le \eta) \ra U_M(\tilz ,\le \eta)
    \end{align*}
    which exhibits the source as an affine space over the target. Moreover, the universal matrix $A^{\univ} \in U((w^M)^\mo \tilz w^M,\le\eta)$ factors as
    \begin{align*}
        A^{\univ} =  (w^M)^\mo M^{\univ} N^{\univ} w^M
    \end{align*}
    where $M^{\univ}$ is the pullback of the universal matrix for $U_M(\tilz, \le\eta)$ to $U((w^M)^\mo \tilz w^M,\le \eta)$ and
    \begin{align*}
        N^{\univ} \in \overline{N}(\cO(U((w^M)^\mo \tilz w^M,\le\eta))[v])
    \end{align*}
    whose $\alpha$-entry for $\alpha \in \Phi^-$ has a degree bounded by $\RG{-\alpha,\nu} - \delta_{(w^M)^\mo z (\alpha)>0}$. The coefficients of the entries of $N^{\univ}$ are exactly the affine coordinates of $U((w^M)^\mo \tilz w^M,\le \eta)$ over $U_M( \tilz,\le \eta)$. 
\end{prop}

\begin{proof}
    The proofs of \cite[Proposition 4.2.4 and 4.2.5]{LLLMextremal} apply to our case by using \cite[Proposition 3.1.4]{lee_thesis}. 
\end{proof}


Let $\bfa\in \cO^3$. Now we explain how the algebraic monodromy condition $\nbl_\bfa$ affects the morphism in the previous Proposition. Note that the condition $\nbl_\bfa$ given by \eqref{eqn:alg-monodromy} is automatic for $U_M(\tilz,\le\eta)$ because $\eta$ is minuscule (i.e.~$0 \le \RG{\eta,\al^\vee}\le 1$ for $\al\in \Delta_M$) as a cocharacter of $M^\vee$. This follows from a direct computation. 

\begin{prop}\label{prop:parabolic+monodromy}
    Suppose that $\bfa$ is $6$-generic. The morphism in Proposition \ref{prop:parabolic} induces a morphism
    \begin{align*}
        U^{\nv}((w^M)^\mo \tilz w^M,\le \eta,\nbl_{\bfa}) \ra U_M(\tilz ,\le \eta)
    \end{align*}
    which exhibits the source as an affine space over the target of relative dimension 3 if $M \neq T$ and 4 if $M=T$. Moreover, if we write the universal matrix $A^{\univ,\nbl_\bfa}\in  U^{\nv}((w^M)^\mo \tilz w^M,\le \eta,{\nbl_{\bfa}})$ as a factorization
    \begin{align*}
        A^{\univ,\nbl_\bfa} = (w^M)^\mo N^{\univ,\nbl_\bfa}M^{\univ}  w^M,
    \end{align*}
    then $\cO({U^{\nv}((w^M)^\mo \tilz w^M,\le \eta,{\nbl_{\bfa}}}))$ is generated over $\cO(U_M(\tilz ,\le \eta))$ by the top degree coefficients of off-diagonal entries of $N^{\univ,\nbl_{\bfa}}$.
\end{prop}

\begin{proof}
    This follows from the proof of \cite[Lemma 4.4.3]{LLLMextremal}. Note that $e=1$ in our case.
\end{proof}

\begin{rmk}\label{rmk:para-loci}
    Suppose that $M=T$ in Proposition \ref{prop:parabolic+monodromy}. In this case, $\tilz = t_\eta$ and $U_M(\tilz,\le\eta)$  is nothing but $\Spec \cO$. 
    Suppose that $M \neq T$ and $\tilz$ is of colength one as discussed in Remark \ref{rmk:para-shape}. In particular, $(w^M)^\mo \tilz w^M$ is in the irregular colength one case in Theorem \ref{thm:col-one-nbhd}. In this case, it is known that $U_M(\tilz,\le\eta)$ is isomorphic to $X_p$ (e.g.~\cite[Theorem 1.1.1]{LLMPQ-CL}).
\end{rmk}

\subsubsection{A regular colength one locus}\label{sec:reg-col-one}


Consider $U(\tilz,\le\eta,\nbl_\bfa)$ when $\tilz^* = \tilw_2^\mo \tilw_h^\mo w_0 \tilw_1$ for $(\tilw_1,\tilw_2)\in \AP'(\eta)$ such that $\tilw_1 \in \Omega$ and $\tilw_2 C_0 = C_1$  We can write $\tilw_1 = \delta$ and $\tilw_2 = s_\alpha \delta$. Then $\tilz^*=  \delta^\mo t_{(0,-1;2)}s_2s_1s_2 \delta$. Using \cite[Lemma 4.1.9]{LLL}, we can check that $\tilw_2^\mo \tilw_h^\mo w_0 \tilw_1$ is regular colength one.

\begin{lemma}\label{prop:reg-col-one}
    Suppose that $\bfa$ is 5-generic. There is an isomorphism $U(\tilz,\le\eta,{\nbl_\bfa}) \simeq X_p \times \A^{3}$.
\end{lemma}
\begin{proof}
    By Remark \ref{rmk:Iwahori-noramlizer}, it suffices to consider the case $\delta=e$. This essentially follows from the proof of \cite[Theorem 2.3.10]{EL}. In order to match the notation in \loccit,~we let $a_0,\dots, a_3\in \Z$ such that $\std(\bfa)=(a_3-1,a_2-2,a_1-1,a_0-2)$. Then the $5$-genericity of $\bfa$ implies that $(a_3,a_2,a_1,a_0)$ is 4-generic. This condition was imposed by the strong $4$-genericity of $\rhobar$ in \loccit. The computation in \loccit~shows that the universal matrix for $U(\tilz,\le\eta)$ is given by the following
    \begin{equation*}
\resizebox{.98\hsize}{!}{\ensuremath{\begin{pmatrix}c_{00} & c_{00}(c_{31}'+c_{31}E(v)) & {c_{00}c_{13}} & (c_{00}c_{33}'+pc_{00}c_{33}-p^2)+(c_{00}c_{33}-p)E(v)- E(v)^2 \\
0 &  E(v)^2 & 0 & c_{13}E(v)^2 \\
0 & c_{21}vE(v) & E(v) & -(c'_{31} + {p c_{13}c_{21}})E(v) + ({c_{13}c_{21}}- c_{31})E(v)^2 \\
v & c_{31}'v+c_{31}vE(v) & {c_{13}}v & c_{33}''+c_{33}'E(v)+c_{33}E(v)^2\end{pmatrix}}}
\end{equation*}
Note that $c_{11}^*, c_{22}^*, c_{03}^*, c_{30}^*,$ and $\xi$ in \loccit~are taken to be 1 because what \loccit~computes is (a completion of) $\tld{U}(\tilz,\le\eta)$ rather than ${U}(\tilz,\le\eta)$. 
As explained in \loccit, the monodromy condition $\nbl_\bfa$ solves the variables $c_{31}', c'_{33}, c''_{33}$ in terms of other variables, leaving 5 variables satisfying the single equation
\begin{align*}
    c_{00}\left[ \frac{e+a_1-a_2+1}{e+a_0-a_3-1} c_{13}c_{31}+\frac{a_0-a_3 -1-e}{e+a_0-a_3-1} c_{33}\right]=p.
\end{align*}
Here, $e=p-1$. Note that we do not have the error term $O(p^2)$ in \loccit~because we use the approximated monodromy condition $\nbl_\bfa$ here. Then, we can get an isomorphism  $U(\tilz,\le\eta)^{\nbl_\bfa} \simeq X_p \times \A^{3}$ by sending $z_{1},z_2,z_3 \in \cO(\A^3)$, $x$, and $y\in \cO(X_p)$, to $c_{21},c_{13}$, $c_{31}$, $c_{00}$, and
\begin{align*}
    \frac{e+a_1-a_2+1}{e+a_0-a_3-1} {c_{13}c_{31}}+\frac{a_0-a_3 -1-e}{e+a_0-a_3-1} {c_{33}},
\end{align*}
respectively.
\end{proof}

\begin{rmk}
    Lemma \ref{prop:reg-col-one} should hold for different choices of $(\tilw_1,\tilw_2)\in \AP'(\eta)$ such that $l(\tilw_1)=l(\tilw_2)-1$. Since we only need 
the case in \loccit, we do not compute the other cases.
\end{rmk}

\subsubsection{Proof of Theorem \ref{thm:col-one-nbhd}}
    We define a morphism
    \begin{align*}
        f: \tld{U}(\tilz,\le\eta,\nbl_\infty) &\ra (T^{\vee,\cJ}\times_\cO X_1^{\cJ_1} \times_\cO X_2^{\cJ_2})^{\pcp}
    \end{align*}
    which is a closed immersion after reducing modulo $\varpi$.
    By $p$-completeness, $f$ is also a closed immersion. Then, it has to be an isomorphism because both the source and the target have the same dimension, and the target is reduced and irreducible.

    For $i\in \{1,2\}$ and each $j\in \cJ_i$, we write $X\ix{j} \defeq X_i$ and $R\ix{j}\defeq \cO(X\ix{j})$. Explicitly, we write
    \begin{align*}
        R\ix{j} = \left\{ \begin{array}{cc}
            \cO[z_1\ix{j},\dots, z_4\ix{j}] & \text{if $i=1$} \\
             \cO[z_1\ix{j}, z_2\ix{j}, z_3\ix{j}, x\ix{j}, y\ix{j}]/(x\ix{j}y\ix{j}-p) & \text{if $i=2$} 
             .
        \end{array} \right.
    \end{align*}
    Recall that $\tld{U}(\tilz,\le\eta)$ is the trivial $T^{\vee,\cJ}$-torsor over $U(\tilz,\le\eta)$. This gives a morphism $\tld{U}(\tilz,\le\eta,\nbl_\infty)^{\pcp} \ra T^{\vee,\cJ}$. To define the morphism $f$, it remains to define a morphism $R\ix{j} \ra \cO(\tld{U}(\tilz,\le\eta,\nbl_\infty))$ for each $j\in\cJ$.  If $i=1$, we send $z_1\ix{j},\dots,z_4\ix{j} \in R\ix{j}$ to the images of the top degree coefficients of the four off-diagonal entries of $N^{\univ,\nbl_{\bfa_\tau}}$ in Proposition \ref{prop:parabolic+monodromy}. If $i=2$ and $\tilz_j$ is in the irregular colength one case, we send  $z_1\ix{j},z_2\ix{j}, z_3\ix{j} \in R\ix{j}$ to the images of the top degree coefficients of the three off-diagonal entries of $N^{\univ,\nbl_{\bfa_\tau}}$ in \loccit. We send $x\ix{j},y\ix{j}$ to the images of the elements of $\cO(U_{M_j}(w_j^{M_j} \tilz_j (w_j^{M_j})^\mo, \le\eta_j))$ corresponding to $x,y \in \cO(X_p)$ under the isomorphism explained in Remark \ref{rmk:para-loci}. Finally, if $i=2$ and $\tilz_j$ is in the regular colength one case, we send $z_1\ix{j},z_2\ix{j},z_3\ix{j}$ to the images of $c_{21}\ix{j}, c_{13}\ix{j}, c_{31}\ix{j}$ in $\cO(U(\tilz_j,\le\eta_j))$ as in the proof of Lemma \ref{prop:reg-col-one}, $x\ix{j}$ and $y\ix{j}$ to $c_{00}\ix{j}$ and
    \begin{align*}
    \left[\frac{e+a_1-a_2+1}{e+a_0-a_3-1} {c_{13}c_{31}}+\frac{a_0-a_3 -1-e}{e+a_0-a_3-1} {c_{33}}\right] \frac{1}{1+pf_{\mathrm{error}}}
\end{align*}
where $f_{\mathrm{error}} \in \cO(\tld{U}(\tilz,\le\eta,\nbl_\infty))$  such that 
\begin{align}\label{xy=p}
    c_{00}\left[ \frac{e+a_1-a_2+1}{e+a_0-a_3-1} c_{13}c_{31}+\frac{a_0-a_3 -1-e}{e+a_0-a_3-1} c_{33}\right]=p(1+pf_{\mathrm{error}}).
\end{align}
Such $f_{\mathrm{error}}$ exists because the left hand side of \eqref{xy=p} equals to $p$ modulo $p^2$ by the computation of $\nbl_\bfa$ in Lemma \ref{prop:reg-col-one} and Proposition \ref{rmk:monodromy-approx}. 
Finally, $f \mod \varpi$ is a closed immersion by Proposition \ref{prop:parabolic+monodromy}, \ref{rmk:monodromy-approx}, and Lemma \ref{prop:reg-col-one}. \qed 

\subsection{Torus fixed points in colength one cases}

In this subsection, we prove the following partial generalization of Proposition \ref{prop:T-fixed-points}.

\begin{prop}\label{prop:col-one-pablo}
    Let $\tilw_1\in \tilW_1^+$ and $s\in W$ be a simple reflection. Suppose that $\bfa$ is $3$-generic. 
    \begin{enumerate}
        \item If $s=s_1$, then
        \begin{align*}
            \left\{( \tilw_1^\mo\tilw_h^\mo s w_0 \tilw)^* \mid \tilw \uparrow \tilw_1, \tilw \in \tilW_1^+\right\}  \subset S_\F^{\nabla_{\bfa}}((\tilw_1^\mo\tilw_h^\mo s w_0 \tilw_1)^*)^{\un{T}^\vee}.
        \end{align*}
        \item If $s=s_2$, then 
        \begin{align*}
            \left\{(\tilw_1^\mo\tilw_h^\mo s w_0 \tilw)^* \mid \tilw \uparrow \tilw_1,\ \tilw \in \tilW_1^+\setminus  \Omega \right\}  \subset S_\F^{\nabla_\bfa}((\tilw_1^\mo\tilw_h^\mo s w_0 \tilw_1)^*)^{\un{T}^\vee}.
        \end{align*}
    \end{enumerate}
\end{prop}

\begin{proof}
Let $w\in W$ be the unique element such that ${w}^\diamond= \tilw_1$. By \cite[Remark 3.5.6 and Proposition 3.5.9]{lee_thesis}, we have $S^{\nbl_\bfa}((\tilw_1^\mo\tilw_h^\mo s w_0 \tilw_1)^*) = S^{\nbl_{(\tilw_1^\mo\tilw_h^\mo)^*(\bfa)}}((sw_0\tilw_1)^*)(\tilw_1^\mo\tilw_h^\mo)^*$. Note that $(\tilw_1^\mo\tilw_h^\mo)^*(\bfa)$ is 3-generic. Thus, it suffices to show that
\begin{align*}
     \left\{(  s w_0 \tilw)^* \mid \tilw \uparrow \tilw_1\right\}  \subset S^{\nabla_{\bfa}}((sw_0\tilw_1)^*)^{\un{T}^\vee}
\end{align*}
for all $3$-generic $\bfa$. 
As explained in the proof of \cite[Theorem A.4]{lee_thesis}, a variant of \cite[Lemma 4.7.1]{MLM} implies that $(sw_0\tilw)^* \in  S^{\nabla_{\bfa}}((sw_0\tilw_1)^*)^{\un{T}^\vee}$ if and only if
\begin{align*}
    L^{--}\GSp_{4,\F}(sw_0\tilw)^* \cap S^\circ((sw_0\tilw_1)^*)^{\nbl_\bfa} \neq \emptyset
\end{align*}
where $L^{--}\GSp_{4,\F}$ is the ind-group-scheme over $\F$ sending an $\F$-algebra $R$ to
\begin{align*}
    L^{--}\GSp_{4,\F}(R) = \{g\in \GSp_4(R[\tfrac{1}{v}]) \mid g \mod \tfrac{1}{v} \in \overline{U}(R) \}.
\end{align*}

As in \cite[\S3.5]{lee_thesis}, we define $N_{\tilz} \defeq \tilz^\mo L^{--}\GSp_{4,\F} \tilz \cap \cI_\F$ for $\tilz\in \tilW^\vee$, where the intersection is taken inside $L\cG_\F$. Then we have $S^\circ((sw_0\tilw_1)^*) \simeq(sw_0\tilw_1)^* N_{(sw_0\tilw_1)^*}$ by Proposition 3.5.4 in \loccit. We define the closed subvariety $N^{\nbl_\bfa}_{(sw_0\tilw_1)^*} \subset N_{(sw_0\tilw_1)^*}$ by the condition that $g \in N^{\nbl_\bfa}_{(sw_0\tilw_1)^*}$ if and only if $(sw_0\tilw_1)^* g \in S^\circ((sw_0\tilw_1)^*)^{\nbl_\bfa}$.

Following the notation in the proof of \cite[Theorem A.4]{lee_thesis}, we need to show that 
\begin{align*}
    \RG{ve_1,\dots,ve_l,e_{l+1},\dots, e_4} (sw_0\tilw_1)^* N^{\nbl_\bfa}_{(sw_0\tilw_1)^*} \cap \RG{ e_1,\dots,e_l,v^\mo e_{l+1},\dots, v^\mo e_4}_{v^\mo}(sw_0\tilw)^* = \{0\}.
\end{align*}
for $l\in \{1,2,3,4\}$. By multiplying $(sw_0)^*$ on the right, we get
\begin{align*}
     \RG{ve_1,\dots,ve_l,e_{l+1},\dots, e_4} \tilw_1^* (sw_0)^* N^{\nbl_\bfa}_{(sw_0\tilw_1)^*} (sw_0)^* \cap \RG{ e_1,\dots,e_l,v^\mo e_{l+1},\dots, v^\mo e_4}_{v^\mo}\tilw^* = \{0\}.
\end{align*}
Moreover, as explained in \loccit, we can replace $N^{\nbl_\bfa}_{(sw_0\tilw_1)^*}$ by $N^{\nbl_\bfa}_{(sw_0\tilw_0)^*}$ where $\tilw_0$ is the element in $\tilW^+_1$ corresponding to the highest restricted alcove. Then, $(sw_0)^* N^{\nbl_\bfa}_{(sw_0\tilw_0)^*} (sw_0)^*$ is exactly the vanishing locus of the $\alpha^\vee$-coordinate of $M^{\tilw_0}$ in \loccit, where $\alpha^\vee$ is the simple coroot corresponding to $w_0sw_0$. Therefore, we need to check that the determinants of block matrices considered in \loccit~do not vanish under the vanishing of the $\alpha^\vee$-coordinate of $M^{\tilw_0}$. This follows from the computation in \loccit, except when $s=s_2$ and $\tilw \in \Omega$. In the exceptional case, the $\alpha^\vee$-coordinate is given by $c_{12}$ in \loccit, and $c_{12}=0$ implies $c_{14}'=0$. Thus, the determinant of the first block matrix considered in \loccit~vanishes under $c_{12}=0$.
\end{proof}

By Proposition \ref{prop:schubert-dim}, $S_\F^{\nbl_\bfa}((\tilw_1^\mo \tilw_h^\mo w_0 s \tilw_1)^*)$ has dimension 3, which is smaller than the dimension of $S_\F^{\nbl_\bfa}(\tilz)$ for $\tilz\in \Adm^{\vee,\reg}(\eta)$. In fact, $S_\F^{\nbl_\bfa}((\tilw_1^\mo \tilw_h^\mo w_0 s \tilw_1)^*)$ is contained in two of irreducible components in $M_\cJ(\le\eta,\nbl_\bfa)_\F$ of dimension 4 and serves as a bridge between the two components.

Let $w\in W$ be the element such that ${w}^\diamond=\tilw_1$ and let $\tilw_2 = {(sw)^\diamond}$. As in Remark \ref{rmk:two-components}, we have
\begin{align*}
    \tilw_1^\mo \tilw_h^\mo w_0 s \tilw_1 = \tilw_2^\mo \tilw_h^\mo w_0 s \tilw_2.
\end{align*}

\begin{lemma}\label{lem:two-components}
    Following the above notations, we have
    \begin{align*}
        S_\F^{\nbl_\bfa}((\tilw_1^\mo \tilw_h^\mo w_0 s \tilw_1)^*) \subset S_\F^{\nbl_\bfa}((\tilw_1^\mo \tilw_h^\mo w_0  \tilw_1)^*) \cap S_\F^{\nbl_\bfa}((\tilw_2^\mo \tilw_h^\mo w_0  \tilw_2)^*).
    \end{align*}
\end{lemma}
\begin{proof}
Let $i\in \{1,2\}$. The claim follows from
\begin{align*}
    S_\F^\circ((\tilw_i^\mo \tilw_h^\mo w_0 s \tilw_i)^*)^{\nbl_\bfa}t_\bfa &=  S_\F^\circ((\tilw_i^\mo \tilw_h^\mo w_0 s \tilw_i)^*) t_\bfa  \cap \Gr_{\cG,\F}^{\nbl_0} 
    \\
    &\subset \overline{S^\circ_\F((w_0s\tilw_i)^*)(\tilw_i^\mo \tilw_h^\mo)^* t_\bfa \cap \Gr_{\cG,\F}^{\nbl_0}}
    \\
    &\subset \overline{S^\circ_\F((w_0\tilw_i)^*)(\tilw_i^\mo \tilw_h^\mo)^* t_\bfa \cap \Gr_{\cG,\F}^{\nbl_0}}
    \\
    &= S_\F^{\nbl_\bfa}((\tilw_i^\mo \tilw_h^\mo w_0  \tilw_i)^*)t_\bfa
\end{align*}
where the closure is taken in $\Gr_{\cG,\F}^{\nbl_0}$. The first line follows from \cite[Remark 3.5.6]{lee_thesis}. The second line follows from a mild generalization of \cite[Proposition 4.3.5]{MLM}. Namely, we claim that
\begin{align*}
    S_\F^\circ((\tilw_i^\mo \tilw_h^\mo w_0 s \tilw_i)^*) t_\bfa  \cap \Gr_{\cG,\F}^{\nbl_0}  = S_\F^\circ(( w_0 s \tilw_i)^*) (\tilw_i^\mo \tilw_h^\mo)* t_\bfa  \cap \Gr_{\cG,\F}^{\nbl_0}. 
\end{align*}
When $w_0s$ is replaced by $w_0$, this is proven in \loccit. Its proof relies on 
\begin{align*}
    \ell(\tilw_i^\mo \tilw_h^\mo w_0 \tilw_i) = \ell(\tilw_i^\mo \tilw_h^\mo) + \ell(w_0 ) + \ell(\tilw_i).
\end{align*}
This remains true when $w_0$ is replaced by $w_0 s$ (see Remark \ref{rmk:two-components}), and thus the same proof applies to our setting. Note that Proposition 4.3.4 in \loccit~also generalizes to this setting because Theorem 4.2.4 in \loccit~does. 

For the third line, we mimic the proof of Proposition 4.3.6 in \loccit. 
It implies that $\overline{S^\circ_\F((w_0\tilw_i)^*)(\tilw_i^\mo \tilw_h^\mo)^* t_\bfa \cap \Gr_{\cG,\F}^{\nbl_0}}$ contains ${S^\circ_\F((w_0\tilw_i)^*)w_0sw_0(\tilw_i^\mo \tilw_h^\mo)^* t_\bfa \cap \Gr_{\cG,\F}^{\nbl_0}}$ as an open dense subvariety. The intersection
\begin{align*}
    \overline{S^\circ_\F((w_0s\tilw_i)^*)(\tilw_i^\mo \tilw_h^\mo)^* t_\bfa \cap \Gr_{\cG,\F}^{\nbl_0}} \cap S^\circ_\F((w_0\tilw_i)^*)w_0sw_0(\tilw_i^\mo \tilw_h^\mo)^* t_\bfa
\end{align*}
contains $(\tilw_i^\mo \tilw_h^\mo w_0s \tilw_i)^*t_\bfa$ and is non-empty. Thus, it is open dense in $\overline{S^\circ_\F((w_0s\tilw_i)^*)(\tilw_i^\mo \tilw_h^\mo)^* t_\bfa \cap \Gr_{\cG,\F}^{\nbl_0}}$, and this implies the claimed inclusion.
\end{proof}

%% file: deformation_rings.tex
\subsection{Applications to deformation rings}
In this section, we obtain some applications of local models to Galois deformation rings. Let $\CNL_\cO$ be the category of complete Noetherian local $\cO$-algebras with residue field $\F$. Let $\rhobar\in \cX_{\Sym}(\F)$. We denote by $R_{\rhobar}^{\square}$ the framed deformation ring representing the functor $D_{\rhobar}^{\square}$ taking $R \in \CNL_\cO$ to the set of $\rho:G_K \ra \GSp_4(R)$ lifting $\rhobar$. For a type $(\lam+\eta,\tau)$, we denote by $R_{\rhobar}^{\lam+\eta,\tau}$ the unique $\cO$-flat quotient of $R_{\rhobar}^{\square}$ whose $R$-points, for any finite flat $\cO$-algebra $R\in \CNL_\cO$, are lattices in potentially crystalline representations of type $(\lam+\eta,\tau)$. If $\psi:G_K \ra \cO^\times$ is a continuous character lifting $\simc(\rhobar)$, we denote by $R_{\rhobar}^{\square,\psi}$ and $R_{\rhobar}^{\lam+\eta,\tau,\psi}$ for the fixed similitude variants of  $R_{\rhobar}^{\square}$ and $R_{\rhobar}^{\lam+\eta,\tau}$, respectively.

Following \cite[Proposition 3.6.3 and 4.8.10]{EGstack}, there are morphisms $\Spf R_{\rhobar}^{\square} \ra \cX_{\Sym}$ and $\Spf R_{\rhobar}^{\eta,\tau} \ra \cX^{\eta,\tau}_{\Sym}$  which are versal at $\rhobar$. Suppose that $\tau$ is 5-generic and $\rhobar$ is a closed point. Let $\tilz\in \Adm^\vee(\eta)$ be the element corresponding to $\rhobar$ as in \S\ref{sec:closed-pt}. As in \S\ref{sec:back-to-EGstack}, we have a $T^{\vee,\cJ}$-torsor
\begin{align*}
    \tld{U}_\reg(\tilz,\le\eta,\nbl_\infty) \onto \cX^{\eta,\tau}_{\Sym}(\tilz).
\end{align*}
Therefore, $\Spf R_{\rhobar}^{\eta,\tau}$ and $\tld{U}_\reg(\tilz,\le\eta,\nbl_\infty)^{\wedge_{\tilz}}$ become isomorphic after taking formally smooth covers.

We define  $R^{\alg}_{\rhobar}$ the quotient ring of $R^{\square}_{\rhobar}$ given by the closed formal subscheme
\begin{align*}
    \Spf R^{\square}_{\rhobar} \times_{\cX_{\Sym}} \cX_{\Sym,\red} \subset \Spf R^{\square}_{\rhobar}.
\end{align*}
Since $\cX_{\Sym,\red}$ is algebraic, the versal map $\Spf R^{\alg}_{\rhobar} \ra \cX_{\Sym,\red}$ arises from a map
\begin{align*}
    \iota_{\rhobar}:  \Spec R^{\alg}_{\rhobar} \ra \cX_{\Sym,\red}.
\end{align*}

\subsubsection{Unibranch property}

\begin{defn}
    Let $Y$ be a scheme. A point $y\in Y$ is called \textit{unibranch} if the normalization of the local ring $(\cO_{Y,y})_\red$ is local.

    If $\cY$ is an algebraic stack with a smooth cover $Y\onto \cY$, then $y\in \cY$ is \textit{unibranch} if any (equivalently,~all) of its preimage in $Y$ is unibranch.
\end{defn}

\begin{rmk}\label{rmk:unibranch}
    Suppose that $Y$ is an excellent Noetherian scheme. Then $y\in Y$ is unibranch if and only if the completed local ring $\cO_{Y,y}^{\wedge}$ is a domain. Moreover, if $Y$ is normal, it is unibranch at any point $y \in Y$. 
\end{rmk}


\begin{prop}[{\cite[Proposition 3.7.3]{lee_thesis}}]\label{prop:modp-unibranch}
    Suppose that $\sig$ is a 3-deep Serre weight. Then, $\cC_{\sig}$ is unibranch at its closed points.
\end{prop}

By the previous remark and proposition, if $\rhobar \in \cX_{\Sym}(\F)$ is a closed point contained in $\cC_\sig$ for some 3-deep $\sig$, the pullback of $\cC_{\sig}$ to $R_{\rhobar}^{\alg}$ is an irreducible cycle in $\Spec R^{\alg}_{\rhobar}$, which we denote by $\cC_{\sig}(\rhobar)$. We denote the corresponding quotient of $R^{\alg}_{\rhobar}$ by $R_{\rhobar}^\sig$.


\begin{thm}\label{thm:col-one-def}
    Let $\tau$ be a 6-generic tame inertial type and  $\rhobar\in \cX_{\Sym}(\F)$ be a closed point such that $\tilw(\rhobar,\tau)^*_{j} = \tilz_j$ as in Theorem \ref{thm:col-one-nbhd} for each $j\in \cJ$. Then, $R_{\rhobar}^{\eta,\tau}$ is a domain. Moreover, $\Spec R_{\rhobar}^{\eta,\tau}/\varpi$ is reduced with $2^{\#\cJ_2}$-many irreducible components given by $\cC_{\sig}(\rhobar)$ for $\sig \in W^?(\rhobar|_{I_{K}})\cap \JH(\osig(\tau))$. 
\end{thm}
\begin{proof}
   A direct computation as in the proof of Proposition \ref{prop:parabolic+monodromy} and Lemma \ref{prop:reg-col-one} shows that the image of $\tilz$ under the projection from $\tilU(\tilz,\le\eta,\nbl_\infty)_\F$ onto $X_{p,\F}^{\cJ_2}$ is the nodal point given by $x=y=0$. Then the claim is a direct consequence of Theorem \ref{thm:col-one-nbhd},  Proposition \ref{prop:C_sig-T-torsor}, and  \cite[Proposition 2.4.8]{LLLMextremal}. 
\end{proof}

Before discussing the main application, we need some preliminaries.

\subsubsection{Cycles}
Let $d \ge 0$ be an integer and $X$ be an equi-dimensional Noetherian algebraic stack. We define $\Z[X]$ as the free abelian group on the set of irreducible components of $X$. We call elements of  $\Z[X]$ \textit{cycles} in $X$. If $C\subset X$ is an irreducible component, then we also call $C$ an \textit{irreducible cycle}. A cycle is \textit{effective} if its coefficients are nonnegative. 

Suppose that $Y$ is a $d$-dimensional Noetherian algebraic stack with an embedding $\iota :Y_\red \into X$. We define a cycle
\begin{align*}
    Z(Y) \defeq \sum_{C \subset Y} \mu_C(Y) \iota(C) \in \Z[X]
\end{align*}
attached to $Y$ where the sum runs over all irreducible components of $Y$ and $\mu_C(Y)$ denotes the multiplicity of $C$ as an irreducible component of $Y$ as in \cite[\href{https://stacks.math.columbia.edu/tag/0DR4}{Tag 0DR4}]{stacks-project}.
If $X=\Spec R$ for some Noetherian ring $R$ and $M$ is an $R$-module, we define the (top-dimensional) cycle attached to $M$ by
\begin{align*}
    Z(M) \defeq \sum_{C\subset X} \mu_C(M)C \in \Z[X] 
\end{align*}
where the sum runs over irreducible components of $X$, $\mu_C(M)$ denotes $\mathrm{length}_{R_{\mathfrak{p}_C}}(M_{\mathfrak{p}_C})$, and $\mathfrak{p}_C\subset R$ denotes the prime ideal corresponding to $C$.

By applying the previous discussion to $X=\cX_{\Sym,\red}$ (resp.~$\Spec R_{\rhobar}^{\alg}$) and $Y= \cX^{\lam+\eta,\tau}_{\Sym,\F}$ (resp.~$\Spec R_{\rhobar}^{\lam+\eta,\tau}/\varpi)$, we get the cycle $\cZ^{\lam,\tau}\defeq Z(\cX^{\lam+\eta,\tau}_{\Sym,\F}) \in \Z[\cX_{\Sym,\red}]$ (resp.~$\cZ^{\lam,\tau}(\rhobar) \defeq Z(\Spec R^{\lam+\eta,\tau}_{\rhobar}/\varpi) \in \Z[\Spec R^{\alg}_{\rhobar}]$).


The morphism $\iota_{\rhobar}: \Spec R_{\rhobar}^{\alg} \ra \cX_{\Sym,\red}$ induces a map from the set of irreducible components in $\Spec R^{\alg}_{\rhobar}$ to the set of irreducble components in $\cX_{\Sym,\red}$ containing $\rhobar$. By regarding cycles as $\Z$-valued functions on sets of irreducible components, we obtain a map
\begin{align*}
    \iota_{\rhobar}^* :   \Z[\cX_{\Sym,\red}] \ra \Z[\Spec R^{\alg}_{\rhobar}].
\end{align*}
Then, it follows from \cite[\href{https://stacks.math.columbia.edu/tag/0DRD}{
Tag 0DRD}]{stacks-project} that $\iota_{\rhobar}^*(\cZ^{\lam,\tau}) = \cZ^{\lam,\tau}(\rhobar)$.

The following result will be useful for the next application of local models to deformation rings.

\begin{prop}[{\cite[Proposition 2.2.13]{EG}}]\label{prop:cycle-specialization}
    Let $X$ be a Noetherian scheme flat over $\cO$ of relative dimension $d$. If $Z(X) = \sum_{C} n_C C$ for some $n_C \in \Z$, then
    \begin{align*}
        Z(X\times_{\cO}\F)  = \sum_{C} n_C Z(C \times_{\cO} \F)
    \end{align*}
    where the sums run over all irreducible components of $X$.
\end{prop}



\begin{setup}\label{setup}
    We choose the following objects:
    \begin{enumerate}
        \item  $\rhobar\in \cX_{\Sym}(\F)$ a $9$-generic tamely ramified representation;
        \item $(\tilw_1,\tilw)\in \AP'(\eta)$;
        \item $s\in \uW$ a simple reflection where we assume that for all $j\in \cJ$, $s\neq s_{2,j}$ whenever $\tilw_{1,j}\in \Omega$ and $\tilw_{1,j}\neq \tilw_j$.
    \end{enumerate}
    The above objects determine the following:
    \begin{itemize}
        \item $\tau$ a tame inertial type over $E$ such that $\tilw(\rhobar,\tau) = \tilw^\mo \tilw_h^\mo w_0 s \tilw_1$ (such $\tau$ is necessarily 6-generic);
        \item $\rhobar_0$ a tame inertial type over $\F$ such that $\tilw(\rhobar_0,\tau) = \tilw^\mo \tilw_h^\mo w_0 s \tilw$ (such $\rhobar_0$ is necessarily 6-generic); 
        \item We have $W^?(\rhobar_0)\cap \JH(\osig(\tau)) = \{ \sig_1,\sig_2\}$ where $\sig_1 = F_{\tau}(w)$ and $\sig_2=F_{\tau}(sw)$ for the unique $w\in \uW$ such that ${w^\diamond}= \tilw$. See \cite[Proposition 2.4.8]{LLLMextremal}. We also remark that $\sig_1 = F_{\rhobar}(\tilw_1,\tilw)$. 
    \end{itemize}
\end{setup}

\begin{lemma}\label{lem:W-cap-JH-inclusion}
    Following the notation in Setup \ref{setup}, we have $W^?(\rhobar_0)\cap \JH(\osig(\tau)) \subset W^?(\rhobar)\cap \JH(\osig(\tau))$.
\end{lemma}

\begin{proof}
    By Corollary \ref{cor:geom-SW}, it suffices to show that $\rhobar \in \cC_{\sig_1}\cap \cC_{\sig_2}$.
    Let $\tilz_1 \defeq (\tilw^\mo \tilw_h^\mo w_0 \tilw)^*$ and $\tilz_2 \defeq ({((sw)^\diamond)}^\mo \tilw_h^\mo w_0 {(sw)^\diamond})^*$ in $\Adm^{\reg}(\eta)$. For $i\in \{1,2\}$, we get a $T^{\vee,\cJ}_\F$-torsor $\tld{S}^{\nbl_{\bfa_\tau}}(\tilz_i)\onto \cC_{\sig_i}$ induced by the diagram \eqref{eqn:modp-local-model}. As discussed in \S\ref{sec:closed-pt}, it suffices to show that  $\tilw(\rhobar,\tau) \in S^{\nbl_\bfa}(\tilz_1)\cap S^{\nbl_\bfa}(\tilz_2)$, and this follows from Proposition \ref{prop:col-one-pablo}.
\end{proof}

The following is our key result for global applications.

\begin{thm}\label{thm:irred-comp-def-ring}
    Following the notation in Setup \ref{setup}, the irreducible components $\cC_{\sigma_1}(\rhobar)$ and $\cC_{\sig_2}(\rhobar)$ uniquely generize to a common irreducible component $\cC$ inside $\Spec R_{\rhobar}^{\eta, \tau}$. 
\end{thm}

\begin{proof}
    
Let $\nu: (\cX^{\eta,\tau}_{\Sym})^{\nm} \ra \cX^{\eta,\tau}_{\Sym}$ be the normalization map.  By applying Theorem \ref{thm:col-one-nbhd} for $\tilz = (w^\mo t_{\eta} w)^*$ and   $\tilz = ((sw)^\mo t_{\eta} sw)^*$, we can check that $\cX^{\eta,\tau}_{\Sym,\F}$ is generically reduced at $\cC_{\sig_1}$ and $\cC_{\sig_2}$. In particular, for $i=1,2$, there exists a unique irreducible component $\tld{\cC}_{\sig_i} \subset (\cX^{\eta,\tau}_{\Sym})^{\nm}_\F$ such that $\nu(\tld{\cC}_{\sig_i})=\cC_{\sig_i}$. This also shows that $R^{\eta,\tau}_{\rhobar}/\varpi$ is generically reduced at $\cC_{\sig_i}(\rhobar)$. Then the uniqueness of a generization of $\cC_i(\rhobar)$ follows from Proposition \ref{prop:cycle-specialization}.  

Note that there is a natural bijection between $\Irr_{4f+12} (\Spec R^{\eta,\tau}_{\rhobar})$ and $\nu^\mo(\rhobar)$. Thus, to find a common generization of $\cC_{\sigma_1}(\rhobar)$ and $\cC_{\sig_2}(\rhobar)$, we need to show that $\nu(\tld{\cC}_{\sig_1} \cap \tld{\cC}_{\sig_2})$ contains $\rhobar$.

By applying Theorem \ref{thm:col-one-nbhd} for $\tilz_0 \defeq \tilw(\rhobar_0,\tau)^*$, $\cX^{\eta,\tau}_{\Sym}(\tilz_0)$ is normal and
\begin{align*}
    \cU\defeq\cC_{\sigma_1}\cap \cC_{\sig_2} \cap \cX^{\eta,\tau}_{\Sym}(\tilz_0)
\end{align*}
is a non-empty open substack of $\cC_{\sigma_1}\cap \cC_{\sig_2}$. By the normality, we have a copy of $\cU$ in $\tld{\cC}_{\sig_1} \cap \tld{\cC}_{\sig_2}$. By the closedness of a normalization map, it suffices to show that the closure of $\cU$ in $\cX^{\eta,\tau}_{\Sym,\F}$ contains $\rhobar$. Then this follows from Proposition \ref{prop:col-one-pablo} and the local model diagram.

\end{proof}

\subsection{Explicit Breuil--M\'ezard cycles}\label{sec:BMcycle}
We collect results on Breuil--M\'ezard cycles for $\GSp_4$, which will be used for our modularity lifting result. The following result proves the Breuil--M\'ezard conjectures for small Hodge--Tate weights and generic tame types. 

For a type $(\lam+\eta,\tau)$ and a Serre weight $\sig$, we denote by $n_{\sig}(\lam,\tau)$ the multiplicity of $\sig$ in any Jordan--H\"older series of $\osig(\lam,\tau)$.

\begin{thm}[{\cite[Theorem 1.3.1]{FLH}}]
    There exists a collection of cycles $\cZ_{\sig}^{\BM}$ attached to Serre weights $\sig$ of $\GSp_4(k)$ such that for any $(2h_{\lam}+6)$-generic type $(\lam+\eta,\tau)$,
    \begin{align*}
        \cZ^{\lam,\tau} = \sum_{\sig \in \JH(\osig(\lam,\tau))} n_{\sig}(\lam,\tau)\cZ_{\sig}^{\BM}.
    \end{align*}
\end{thm}

The construction of $\cZ_{\sig}^{\BM}$ in \cite{FLH} provides an algorithm to compute them explicitly. This is implemented in \cite{LHL-BM} for low rank groups. For $\GSp_4$, we have the following result.

\begin{thm}[{\cite{LHL-BM}}]\label{thm:BM-cycle}
    For $3$-deep Serre weight $\sig=F(\lam)$,
    \begin{align*}
        \cZ_{\sig}^{\BM} = \sum_{\lam'} \cC_{F(\lam')}
    \end{align*}
    where the sum runs over the set of $\lam'\in X^*(\uT)$ such that for each $j\in \cJ$, either $(\lam')\ix{j} = \lam\ix{j}$ or $\lam\ix{j} \in C_2$, $(\lam')\ix{j}\in C_0$, and $(\lam')\ix{j}\uparrow \lam\ix{j}$.
\end{thm}


By taking pullback, we obtain explicit Breuil--M\'ezard cycles in deformation rings. For a 6-generic tame inertial $L$-parameter $\rhobar$, we denote by $\cZ_{\sig}^{\BM}(\rhobar)$ the pullback of $\cZ_{\sig}^{\BM}$ to $R_{\rhobar}^{\alg}$.

\begin{cor}\label{cor:BM-cycle-def}
    Let $\rhobar$ be a $6$-generic tame inertial $L$-parameter. If $(\lam+\eta,\tau)$ is a $(2h_{\lam}+6)$-generic type, then
    \begin{align*}
        \cZ^{\lam,\tau}(\rhobar) =  \sum_{\sig \in \JH(\osig(\lam,\tau))} n_{\sig}(\lam,\tau)\cZ_{\sig}^{\BM}(\rhobar).
    \end{align*}
\end{cor}

\subsection{Functions on potentially crystalline stacks}\label{sub:functions}
In this subsection, we construct certain functions on potentially crystalline Emerton--Gee stacks matching the integral Hecke operators in \S\ref{subsub:hecke-alg}.

Let $\tau$ be a tame inertial type. Recall the standard representation $\std: \GSp_4\ra \GL_4$. It induces a natural morphism 
\begin{align*}
    \std: \cX_{\Sym}^{\eta,\tau} &\ra \cX_{4}^{\eta',\tau'}
    \\
    \rho &\mapsto \std(\rho)
\end{align*}
where $\cX_4^{\eta',\tau'}$ is the stack of 4-dimensional potentially crystalline representations of Hodge type $\eta'=(3,2,1,0)$ and $\tau'=\std(\tau)$. When $\tau$ is 5-generic, \cite[Corollary 5.1.2]{lee_satake} constructs certain global functions on $\cX_4^{\eta',\tau'}$ by pulling back functions on a stack of Breuil--Kisin modules. As explained in \loccit,~these functions evaluated at a potentially crystalline representation $\rho\in \cX_4^{\eta',\tau'}(\cO)$ compute Frobenius eigenvalues of the Weil--Deligne representation associated to $\rho$. We explain the pullback of these functions to $\cX_{\Sym}^{\eta,\tau}$ along $\std$. For simplicity, we only consider principal series $\tau$.

Let $\sig$, $\tau$, $\sig(\tau)$, and $\sig^\circ(\tau)$ be as in \S\ref{subsub:hecke-alg}. We emphasize that we identify $\tau= [\phi(w_{f-1}^\mo(\tld{\mu}))]$ as $T(\cO)$-valued representations.


\begin{thm}\label{thm:functions}
    Suppose that $\tau$ is 5-generic. There is an injective $\cO$-algebra homomorphism
    \begin{align*}
        \Psi^\tau: \cH(\sig^\circ(\tau)) \ra \cO(\cX^{\eta,\tau}_{\Sym})
    \end{align*}
    which interpolates the local Langlands correspondence in the following sense: for a finite extension $\cO'/\cO$, if $\rho \in \cX^{\eta,\tau}_{\Sym}(\cO')$, then the natural $\cH(\sig^\circ(\tau))$-action on $\Hom_{\GSp_4(\cO_K)}(\sig^\circ(\tau), r_p^\mo(\mathrm{WD}(\rho)))$ is given by the character
    \begin{align*}
        \cH(\sig^\circ(\tau)) \xrightarrow{\Psi^{\tau}} \cO(\cX^{\eta,\tau}_{\Sym}) \xrightarrow{\mathrm{ev}_{\rho}} \cO'.
    \end{align*}
\end{thm}
We remark that $r_p^\mo (\mathrm{WD}(\rho))$ above is well-defined because the $L$-packet for a principal series is a singleton.
\begin{proof}
    Recall that $\std$ induces an embedding of $W^\vee$ into $ W^\vee_4$ the (dual) Weyl group of $\GL_4$. We consider the action of $W$ on $\{1,2,3,4\}$ induced by $W\simeq W^\vee \subset W^\vee_4$ and the usual action of $W^\vee_4$. 
    
    Let $\rho \in \cX^{\eta',\tau'}_4(\cO')$. Then $\mathrm{WD}(\rho) = \oplus_{i=1}^4 \chi_i$ and $\mathrm{WD}(\rho)|_{I_K} = \std(\tau)$ as $T_4$-valued representations. By \cite[Corollary 5.1.2]{lee_satake}, for each $I\subset \{1,2,3,4\}$, there are  functions $F_I'\in \cO(\cX^{\eta',\tau'}_4)[1/p]$ and $\tld{F}'_I = q^{\#I(7-\#I)/2}F_I'\in \cO(\cX^{\eta',\tau'}_4)$ characterized by the following property: for $\rho$ as above, $F_I'(\rho) = \prod_{i\in I}\chi_i(\Frob_K) \in E$. We define $F_I:= \std^*(F_I')$. 
    For $\mu\in X_*(T)$, let $F_\mu\in \cO(\cX^{\eta,\tau}_{\Sym})[1/p]$ be the function taking $\rho\in \cX_{\Sym}^{\eta,\tau}(\cO')$ to $\phi^\mo(w_{f-1}^\mo(\mu))(\rho)(\Frob_K)$. 

    We first define
    \begin{align*}
        \Psi^{\tau}[\frac{1}{p}]: \cH(\sig(\tau)) \ra \cO(\cX^{\eta,\tau}_{\Sym})[\frac{1}{p}]
    \end{align*}
    by sending $T'_\mu$ to the function $F_\mu$. This is a well-defined $E$-algebra morphism by Proposition \ref{prop:hecke-rational}. To see that it maps $\cH(\sig^\circ(\tau))$ into $\cO(\cX^{\eta,\tau}_{\Sym})$, we can check that it maps
    $T_{w(-\be_1)}$, $T_{w(-\be_2)}$, and $T_{-\be_0}$ to 
    \begin{align*}
        \tld{F}_{w_{f-1}^\mo w(I_1)}\defeq \std^*(\tld{F}_{w_{f-1}^\mo w(I_1)}'), \ \ \tld{F}_{w_{f-1}^\mo w(I_2)} \defeq \std^*(\tld{F}_{w_{f-1}^\mo w(I_2)}'), \ \  \text{and} \ \  \tld{F}_{I_0} \defeq q^3 F_{I_0}
    \end{align*}
    respectively, where $I_1=\{1,2\}$, $I_2=\{1\}$, and $I_0=\{1,4\}$. It is clear that the first two functions are in $\cO(\cX^{\eta,\tau}_{\Sym})$. For the last, note that $F_{I_0}$ takes $\rho$ to $\simc(\rho)(\Frob_K)$. Since $\simc(\rho)$ has Hodge type $\{3\}$, $q^3 F_{I_0}$ can be realized as the pullback of a global function on $\cX_1^{3, \simc(\tau)}$. The proof of injectivity is identical to \cite[Theorem 4.3.2]{lee_satake}, and the proof of compatibility with the local Langlands correspondence also follows by repeating the argument in \loccit~using \cite[Proposition 2.4.3, 2.4.4, and 2.4.6]{BCGP}. 
\end{proof}

Let $\tld{F}_i \defeq \Psi^\tau(T_{\be_i})$ for $i=1,2$. The following proposition shows that any Galois representation in the non-vanishing locus of $\tld{F}_1$ or $\tld{F}_2$ is automatically valued in a proper parabolic subgroup of $\GSp_4$ with precise $p$-adic Hodge-theoretic properties on its Levi quotient. We write
\begin{align*}
    \std(\tld{\mu}) = \tld{\mu}_1\oplus \tld{\mu}_2 \oplus \tld{\mu}_3 \oplus \tld{\mu}_4.
\end{align*}

\begin{prop}\label{prop:para-locus}
    Suppose that $\tau$ is 5-generic. For $\rho\in \cX^{\eta,\tau}_{\Sym}(\cO)$, if $\tld{F}_i(\rho)\in \cO^\times$ for $i=1$ (resp.~$i=2$), then after $\GSp_4(\cO)$-conjugation, $\rho$ is valued in $S(\cO)$ (resp.~$Q(\cO)$) and its image in the Levi quotient is given by $\psi_1\oplus \psi_2$ (resp.~$\chi_1 \oplus \psi \oplus \chi_2$) where $\psi_1\in \cX_2^{(3,2),\tld{\mu}_1\oplus \tld{\mu}_2}(\cO)$ and $\psi_2\in \cX_2^{(1,0),\tld{\mu}_3\oplus \tld{\mu}_4}(\cO)$ (resp.~$\chi_1 \in \cX_1^{3,\tld{\mu}_1}(\cO)$, $\psi \in \cX_2^{(2,1),\tld{\mu}_2\oplus \tld{\mu}_3}(\cO)$, and $\chi_2 \in \cX_1^{0,\tld{\mu}_4}(\cO)$).
\end{prop}
\begin{proof}
    This follows from \cite[Proposition 5.3.8]{lee_satake} and the construction of $\tld{F}_i$. 
\end{proof}

\begin{cor}\label{cor:vanishing}
    Suppose that $\tau$ is $6$-generic. Let $\kappa \in \JH(\osig(\tau))$ be a weight distinct from $\sig$ satisfying $\kappa \uparrow \sig$. Then the restriction of $\tld{F}_i$ to $\cC_{\kappa}$ is identically zero for $i=1,2$.
\end{cor}
\begin{proof}
    If $\tld{F}_i$ does not vanish at some point of $\cC_{\kappa}$, then the point admits a potentially crystalline lift of the above form. To deduce a contradiction, we show that the mod $p$ reduction of such a potentially crystalline representation is not contained in $\cC_{\kappa}$.

    Define $\cX_S$ and $\cX_Q$ as the following fiber products:
    \[
    \begin{tikzcd}
        \cX_S \arrow[r] \arrow[d] & \cX_{2,\F}^{(3,2),\tld{\mu}_1\oplus \tld{\mu}_2} \arrow[d, "\det"]  & \cX_Q \arrow[r] \arrow[d] &  \cX_{1,\F}^{3,\tld{\mu}_1} \times \cX_{1,\F}^{0,\tld{\mu}_4} \arrow[d, "\otimes"]
        \\
        \cX_{2,\F}^{(1,0),\tld{\mu}_3\oplus \tld{\mu}_4} \arrow[r, "\det"] & \cX_{1,\F} & \cX_{2,\F}^{(2,1),\tld{\mu}_2\oplus \tld{\mu}_3} \arrow[r, "\det"] & \cX_{1,\F}.
    \end{tikzcd}
    \]
    Then there are morphisms from $\cX_S$ and $\cX_Q$ into $\cX_{\Sym,\F}^{\eta,\tau}$ given by the embedding of $S, Q$ into $\GSp_4$. 
    Since $\cC_{\kappa}$ is closed under partial semisimplification, it suffices to show that the images of these morphisms do not intersect with $\cC_\kappa$ (cf.~\cite[Proposition 5.3.10]{lee_satake}). 
    In turn, all loci under consideration are closed under semisimplification. Thus, we can check the intersection at the level of closed points. By the local model theory, the claim follows from the fact that the subset
    \begin{align*}
        \{ \tilw(\rhobar,\tau) \mid \rhobar \in \cC_{\kappa}(\F) \text{ semisimple}\} \subset \Adm(\eta)
    \end{align*}
    is disjoint from the images of $(w^M)^\mo \Adm_M(\eta)w^M \into \Adm(\eta)$ where $M$ is the Levi factor of $S$ or $Q$ and $w$ is the Weyl group part of the lowest alcove presentation of $\tau$. This can be easily checked by hand.
\end{proof}

%% file: main_result.tex
\section{Applications}

\subsection{Global setup}\label{sec:global}
Let $F$ be a totally real field in which $p$ is unramified. Let $\chi: \A_F^\times/F^\times \ra \C^\times$ be a Hecke character. We write $\chi_{p,\iota} = \iota^\mo \circ \chi$. Let $U = U_p U^{p} \le \GSp_4(\cO_p) \times \GSp_4(\A^{\infty,p}_F)$ be a compact open subgroup. We assume that $U^{p}$ is neat. 

\subsubsection{\'Etale cohomology of Hilbert--Siegel modular varieties}

Let $\Sh_{U}$ be the Hilbert--Siegel modular variety associated to $\Res_{F/\Q}\GSp_4$ of level $U$. It is a quasi-projective scheme over $F$ (\cite{Lan-book}). For a $U_p$-representation on a finite $\cO$-module $V$, we denote by $\cF_V$ the locally constant \'etale sheaf of $\cO$-modules on $\Sh_U$ (cf.~\cite[Proposition A1.8]{freitag-kiehl}). We are interested in the \'etale cohomology group $H^i_{\et}(\Sh_{U,\overline{F}},\cF_V)$. We denote by $H^i_{\et}(\Sh_{U,\overline{F}},\cF_V)_{\chi} \subset H^i_{\et}(\Sh_{U,\overline{F}},\cF_V)$ the subspace on which $\A_F^\times$ acts by $\chi_{p,\iota}$.

From now on, we assume that $U$ is \textit{sufficiently small}, i.e.~it contains no element of order $p$. 
Let $S_p$ be the set of places of $F$ dividing $p$ and $P_U$ be the finite set of finite places of $F$ at which $U$ is ramified. For a finite set $P$ containing $S_p\cup P_U$, we define the   \emph{universal Hecke algebra} $\T^{S,\univ}$ to be the polynomial ring over $\cO$ generated by $S_v, T_{v,1}, T_{v,2}$ for each $v \notin P$. 
Then $\T^{S,\univ}$ naturally acts on $H^i_{\et}(\Sh_{U,\overline{F}},\cF_V)_{\chi}$ by Hecke correspondences where $S_v,T_{v,1}, T_{v,2}$ act through the double coset operators
\begin{align*}
\left[ U_v\begin{pmatrix}
    \varpi_v & & & \\
    & \varpi_v & & \\
    & & \varpi_v & \\
    & & & \varpi_v
\end{pmatrix} U_v\right] , \left[ U_v\begin{pmatrix}
    \varpi_v & & & \\
    & \varpi_v & & \\
    & & 1 & \\
    & & & 1
\end{pmatrix} U_v\right] , \left[ U_v\begin{pmatrix}
    \varpi_v^2 & & & \\
    & \varpi_v & & \\
    & & \varpi_v & \\
    & & & 1
\end{pmatrix} U_v\right].
\end{align*}
We denote by $\T_\chi^{P,i}(U,V)$ the quotient of $\T^{S,\univ}$ that acts faithfully on $H^i_{\et}(\Sh_{U,\overline{F}},\cF_V)_{\chi}$. Later we will only consider $i=3$ and just write $\T_\chi^{P}(U,V)$ when $i=3$.

Let $\lam\in X^*(T)^{\cJ}$ be a dominant character. Let $\fp \lhd \T_\chi^{P,i}(U,W(\lam))[1/p]$ be a maximal ideal. By Matsushima's formula and \cite{sorensen-gsp4,mok-cm}, there is a Galois representation $r: G_F \ra \GSp_4(\cO)$ attached to $\fp$ satisfying local-global compatibilities (see \cite[Proposition 5.3.3]{lee_thesis} for the detail). 
Suppose that $\fm\lhd \T_\chi^{P,i}(U,W(\lam))$ is the maximal ideal containing $\fp\cap \T_\chi^{P,i}(U,W(\lam))$. We say that $\fm$ (or its preimage in $\T^{S,\univ}$) is \textit{non-Eisenstein} if $\rbar\defeq r\otimes_\cO \F$ is absolutely irreducible.

We denote by $P_{\rbar}$ the set of finite places at which $\rbar$ is ramified. If $P$ also contains $P_{\rbar}$, we denote by $\fm_{\rbar,\chi}^P\le \T^{S,\univ}$ the maximal ideal defined in \cite[\S5.3]{lee_thesis}. Note that if $\fp$ and $\fm$ are as above and $\rbar= r\otimes_\cO \F$  with $r$ attached to $\fp$, then $\fm_{\rbar,\chi}^P$ is the preimage of $\fm$ in $\T^{S,\univ}$.

Let $\rbar: G_F \ra \GSp_4(\F)$ be a continuous representation unramified at all but finitely many places. As in \cite[Definition 5.4.1]{lee_thesis}, we say that $\rbar$ satisfies \textit{Taylor--Wiles conditions} if it is absolutely irreducible, odd, vast, and tidy. 
By a standard argument using the Cebotarev density theorem, these conditions imply that there are infinitely many places $v$ of $F$ such that $\rbar|_{G_{F_v}}$ is unramified and \textit{generic} in the sense of \cite[Definition 1.1]{Hamann-Lee}.

\begin{prop}\label{prop:torsion-vanishing}
   Let $d=[F:\Q]$. If $\rbar$ satisfies Taylor--Wiles conditions, then $H^i_{\et}(\Sh_{U,\overline{F}},\cF_V)_{\chi,\fm_{\rbar,\chi}^P} = 0$ unless $i=3d$. 
\end{prop}
\begin{proof}
    It suffices to consider when $V$ is an $\F$-module. We first consider the case when $V=\F$. Since $\rbar|_{G_{F_v}}$ is generic for infinitely many place $v$ of $F$, we can apply the vanishing result of Hamann--Lee \cite[Theorem 1.8]{Hamann-Lee}. Then the claim follows from the fact that $H^i_{\et}(\Sh_{U,\overline{F}},\F)_{\fm_{\rbar,\chi}^P}= H^i_{\et,c}(\Sh_{U,\overline{F}},\F)_{\fm_{\rbar,\chi}^P}$, which follows from our assumption that $\rbar$ is absolutely irreducible (see \cite[Proposition 5.1]{Ortiz}; the assumption that $F=\Q$ or $i=3$ therein is not necessary). 

    Now consider the general case. 
    Let $U_p'\lhd U_p$ be a finite index normal subgroup that acts trivially on $V$ and write $\Delta\defeq U_p/U_p'$. Let $U'\lhd U$ be the subgroup such that $(U')_p=U'_p$ and $(U')^p = U^p$.  By the Hochschild--Serre spectral sequence and the previous paragraph, we get
    \begin{align*}
        H^i_{\et}(\Sh_{U,\overline{F}},\cF_V)_{\chi,\fm_{\rbar,\chi}^P} = H^{i-3d}(\Delta, H^3_{\et}(\Sh_{U',\overline{F}},\F)_{\chi,\fm_{\rbar,\chi}^P} \otimes_\F V)
    \end{align*}
    for $i\ge 3d$ and $H^i_{\et}(\Sh_{U,\overline{F}},\cF_V)_{\chi,\fm_{\rbar,\chi}^P}=0$ for $i<3d$. Note that $\Delta$ acts on $H^{3d}_{\et}(\Sh_{U',\overline{F}},\F)_{\chi,\fm_{\rbar,\chi}^P} \otimes_\F V$ by the diagonal action. By \cite[Proposition 1.2]{crew-et_p_cover}, $H^{3d}_{\et}(\Sh_{U',\overline{F}},\F)_{\chi,\fm_{\rbar,\chi}^P}$ is projective over $\F[\Delta]$, and this implies that $H^{3d}_{\et}(\Sh_{U',\overline{F}},\F)_{\chi,\fm_{\rbar,\chi}^P} \otimes_\F V$ is projective over $\F[\Delta]$ by \cite[Chapter III.~Corollary 5.7]{Brown}.  Thus, $H^{i-3d}(\Delta, H^3_{\et}(\Sh_{U',\overline{F}},\F)_{\chi,\fm_{\rbar,\chi}^P} \otimes_\F V)=0$ if $i>3d$.  
\end{proof}


\begin{defn}
Let $\rbar: G_F \ra \GSp_4(\F)$ be an absolutely irreducible continuous representation unramified at all but finitely many places. Let $U\le \cG(\A_F^{\infty})$ be a compact open subgroup and $\chi$ be a Hecke character. 
\begin{enumerate}
\item We say that $\rbar$ is \textit{automorphic} if it is isomorphic to $r_{\pi,p,\iota} \otimes_\cO \F$ for some regular algebraic cuspidal automorphic representation $\pi$ of $\GSp_4(\A_F)$. 

\item Let $\sigma$ be a Serre weight of $\rG$. Suppose that $U_p= \GSp_4(\cO_F\otimes\Z_p)$. We say that $\rbar$ is \emph{modular of weight $\sigma$} (and level $U$), or equivalently, $\sig$ is a \emph{modular weight} of $\rbar$ (at level $U$) if $H^i_{\et}(\Sh_{U,\overline{F}},\cF_{\sig^\vee})_{\chi,\fm_{\rbar,\chi}^P}\neq 0$ for some $\chi$, $P$, and $i\in \Z_{\ge 0}$. 
\item We let $W(\rbar)$ be the set of modular Serre weights of $\rbar$. 
\end{enumerate}
\end{defn}

We remark that $\rbar$ satisfying the Taylor--Wiles conditions is automorphic if and only if $H^i_{\et}(\Sh_{U,\overline{F}},\F)_{\chi,\fm_{\rbar,\chi}^P}\neq 0$ for some $i\in \Z_{\ge 0}$ for some level $U$, Hecke character $\chi$, and a finite set of finite places $P$ containing $P_U \cup P_{\rbar} \cup S_p$. Moreover, it follows from a standard argument and the torsion vanishing result (Proposition \ref{prop:torsion-vanishing}) that this is equivalent to $W(\rbar)\neq \emptyset$ (cf.~\cite[Lemma 2.1.11]{BLGG}). 


\begin{conj}[{\cite{GHS, HT}}]\label{conj:SWC}
If $\rbar$ is automorphic and $\rbar|_{G_{F_v}}$ is tame and sufficiently generic for $v|p$, then $W(\rbar)=W^?(\rbar_p|_{I_{\Qp}})$.
\end{conj}


\subsubsection{Coherent cohomology of Siegel threefolds}\label{sub:coh-coh}
Throughout this subsection, we take $F=\Q$ and $U_p = \GSp_4(\cO_p)$. Then $\Sh_U$ admits a smooth integral model defined over $\cO$ (cf.~\cite{Lan-book}). For any $\cO$-algebra $R$, we write its base change to $R$ by $\Sh_{U,R}$.

We recall some facts about toroidal compactification of $\Sh_U$ and automorphic vector bundles on it. We refer \cite[\S2]{Ortiz} for the details. We denote by $\Sh_{U}^{\tor}$ for a toroidal compactification of $\Sh_{U}$. A priori, this depends on a choice of cone decomposition, but its coherent cohomology groups are independent of such a choice. Thus, we suppress the choice in our notation.

For a weight $\lam \in X^*(T)$, there is an automorphic vector bundle $\omega(\lam)$ on $\Sh_U$ with extensions $\omega^{\mathrm{sub}}(\lam)$ and $\omega^{\mathrm{can}}(\lam)$ on $\Sh_U^{\mathrm{tor}}$. For $?\in \{\can,\sub\}$, we have the subspace $H^i(\Sh^{\tor},\omega^?(\lam))_{\chi}\subset H^i(\Sh^{\tor},\omega^?(\lam))$ defined analogously to the \'etale case. If $P$ is a set of places containing $P_U$ and $p$, then $\T^{S,\univ}$ acts on $H^i(\Sh^{\tor},\omega^?(\lam))$ and $H^i(\Sh^{\tor},\omega^?(\lam))_\chi$. We write $\T_\chi^{P,\coh,i}(U,\omega^?(\lam))$ for the quotient of $\T^{S,\univ}$ that acts faithfully on $H^i(\Sh^{\tor},\omega^?(\lam))_\chi$. Later, we will fix $i=0$ and just write $\T_\chi^{P,\coh}(U,\omega^?(\lam))$ if $i=0$. 


We summarize the vanishing result for mod $p$ coherent cohomology necessary for our applications.
\begin{prop}\label{prop:coh-vanishing}
    Suppose that $\lambda \in X_1(T)$ is 2-generic. Then $H^i(\Sh_{U,\F}^{\tor}, \omega^{\sub}(\lambda+(3,3;0)) = 0 $ for $i \ge 1$. If $\fm \lhd \T^{S,\univ}$ is a non-Eisenstein maximal ideal, then $H^i(\Sh_{U,\F}^{\tor}, \omega^{\sub}(\lambda+(3,3;0))_{\fm} = H^i(\Sh_{U,\F}^{\tor}, \omega^{\mathrm{can}}(\lambda+(3,3;0))_{\fm}$ for all $i$.
\end{prop}
\begin{proof}
    The first claim follows from \cite[Theorem 2.12(1)]{Ortiz} and the second one follows from the proof of \cite[Theorem 24(iii)]{AH}.
\end{proof}

From now on, we just write $\omega(\lam)$ for $\omega^{\can}(\lambda)$ unless stated otherwise. 
We also remark that $H^0(\Sh_{U,R},\omega(\lam)) = H^0(\Sh_{U,R}^{\tor},\omega(\lam))$ for $R\in \{\cO,\F\}$ (\cite[Proposition 2.11(6)]{Ortiz}).

We record the following result on a mod $p$ differential operator from \cite{Ortiz}. Note that our parameterization of $X^*(T)$ is slightly different from \loccit. 

\begin{thm}[{\cite[Theorem 4.12]{Ortiz}}]\label{thm:ortiz}
Let $\lam=(a+3,b+3;c)\in X^*(T)$ such that $p-1 \ge a+3 \ge b+3 \ge 0$ and $b+3\neq p-1$. Then there is a $\T^{S,\univ}$-equivariant injective morphism
\begin{align*}
    \theta^1_{\lam} : H^0(\Sh_{U,\F}, \omega(\lam)) \ra H^0(\Sh_{U,\F}, \omega(2p-a-1,b+3;c+a-p+2)).
\end{align*}
\end{thm}

\begin{rmk}
    Let $\lam_1=(a,b;c)\in C_1$ and $\lam_2=(2p-a-4,b;c+a-p+2)\in C_2$ so that $\lam_1\uparrow \lam_2$. Then by putting $\lam=\lam_1+(3,3;0)$ in Theorem \ref{thm:ortiz}, we get
    \begin{align*}
    \theta^1_{\lam} : H^0(\Sh_{U,\F}, \omega(\lam_1+(3,3;0))) \ra H^0(\Sh_{U,\F}, \omega(\lam_2+(3,3;0))).
\end{align*}
\end{rmk}

\subsubsection{Comparison between two cohomologies}\label{sub:two-coh-comp}
We explain how to relate coherent cohomology to \'etale cohomology. We keep the setup of the previous subsection. Let $\rbar: G_F \ra \GSp_4(\F)$ be a continuous homomorphism satisfying Taylor--Wiles conditions and let $\fm = \fm_{\rbar, \chi}^{P}$ where $P$ contains $P_U\cup P_{\rbar} \cup \{p\}$. By \cite[Equation (5.1)]{Ortiz}, we have the following rational comparison result
\begin{align*}
    H^3_{\et}(\Sh_{U,\Qpbar}, \cF_{W(\lam)} \otimes_{\cO}\Qpbar )_{\fm}\simeq \oplus_{w\in W^M} H^{l(w)}(\Sh^{\mathrm{tor}}_{U,\Qpbar},\omega(w\cdot \lam + (3,3;0)))_{\fm}.
\end{align*}
Here, $M$ is the Levi factor of the Siegel parabolic subgroup (generated by $B$ and $s_1$). 

Note that the quotient of $\T^{S, \univ}$ faithfully acting on 
$H^{l(w)}(\Sh^{\mathrm{tor}}_{U_n,\Qpbar},\omega(w\cdot \lam + (3,3;0)))_{\fm}$ for $w\in W^M$ are all isomorphic. This follows from the fact that once we localize at $\fm$, automorphic representations corresponding to these Hecke eigensystems are non-CAP, in which case the multiplicity of the (anti)holomorphic discrete series (contributing to $H^0$ and $H^3$) is equal to the one of the nonholomorphic discrete series (contributing to $H^1$ and $H^2$) in the same $L$-packet (see the proof of \cite[Proposition 5.1]{Ortiz}). The dimensions of $H^0$ and $H^3$ (resp.~of $H^1$ and $H^2$) are the same by the Poincare duality for $(\mathfrak{g},K)$-cohomology. 
As a result, we obtain the following comparison for Hecke algebras.
\begin{cor}\label{cor:hecke-comp}
    Let $\rbar: G_\Q \ra \GSp_4(\F)$ be a continuous homomorphism satisfying Taylor--Wiles conditions. Then, there is a natural isomorphism
    \begin{align*}
        \T^{S}_\chi(U,W(\lam))_{\fm_{\rbar,\chi}^P} \simeq \T^{S,\coh}_\chi(U,\omega(\lam+(3,3;0)))_{\fm_{\rbar,\chi}^P}.
    \end{align*}
\end{cor}
\begin{proof}
    This follows from the previous paragraph as both Hecke algebras are the faithful quotient of the localization $\T^{S,\univ}_{\fm_{\rbar,\chi}^P}$ acting on either $H^3_{\et}(\Sh_{U_n,\Qpbar}, \cF_{W(\lam)}\otimes_{\cO}\Qpbar )_{\fm_{\rbar,\chi}^P}$ or $H^{0}(\Sh^{\mathrm{tor}}_{U_n,\Qpbar},\omega(\lam + (3,3;0)))_{\fm_{\rbar,\chi}^P}$.
\end{proof}

\subsection{Patching functor and patched modules}\label{sec:patching}
We temporarily return to the local setup. Recall the finite \'etale $\Zp$-algebra $\cO_p$.  We can write it as a product $\cO_p = \prod_{v\in S_p} \cO_v$ where $S_p$ is a finite set and $\cO_v$ is a finite \'etale local $\Zp$-algebra. We write $F_p \defeq \cO_p[1/p]$ and $F_v = \cO_v[1/p]$ for $v\in S_p$. We take $\cJ$ to be $\Hom_{\Zp}(\cO_p,\cO)$.

Let $\rhobar:G_{\Q_p} \ra \LuG(\F)$ be a tame $L$-parameter. We fix a continuous character $\psip:G_{\Q_p} \ra \prescript{L}{}{\underline{\G_m}(\cO)}$ lifting $\simc(\rhobar)$. Note that $\psip$ is equivalent to a collection $\{\psi_v : G_{F_v} \ra \cO^\times\}_{v\in S_p}$. We assume that $\psi_v\eps^{-d}$ is tamely ramified with a finite image for $v \in S_p$ for some $d\in \Z$.

Fix $R^p$ a (nonzero) complete Noetherian equidimensional flat $\cO$-algebra with residue field $\F$ such that $R^p[1/p]$ is generically formally smooth and each irreducible components of $\Spec R^p$ and $\Spec R^p/\varpi$ is geometrically irreducible. In the next subsection, we will make the choiece of $R^p$ explicit. 
We define 
\begin{align*}
    R_{\rhobar}^{\psi_p}\defeq \ctimes_{v\in S_p, \cO} R_{\rhobar_v}^{\square,\psi_v}, \ \ R_\infty^{\psi_p}\defeq R_{\rhobar}^{\psi_p}\ctimes_\cO R^p.
\end{align*}
For dominant weight $\lam\in X^*(T)^{\cJ}$ and a tame inertial type $\tau$, we define
\begin{align*}
    R_{\rhobar}^{\lam+\eta,\tau,\psi_p} \defeq \ctimes_{v\in S_p,\cO} R_{\rhobar_v}^{\lam_v+\eta_v,\tau_v,\psi_v}, \ \ R_{\infty}^{\lam+\eta,\tau,\psi_p}\defeq R_\infty^{\psi_p} \otimes_{R_{\rhobar}^{\psi_p}} R_{\rhobar}^{\lam+\eta,\tau,\psi_p}.
\end{align*}

Let $\Mod(R_{\infty}^{\psi_p})$ be the category of finitely generated modules over $R_{\infty}^{\psi_p}$ and $\Rep_\cO^{\psi_p}(\GSp_4(\cO_p))$ be the category of topological $\cO[\GSp_4(\cO_p)]$-modules, which are finitely generated over $\cO$, with fixed central character given by $\otimes_{v\in S_p}(\psi_v\eps^{-3}|_{I_{F_v}})\circ (\Art_{F_v}|_{\cO^\times_{v}})$. We define the \textit{Hecke category} $\Rep_\cO^{\mathrm{Hk}, \psi_p}(\GSp_4(\cO_p))$ whose objects are objects of $\Rep_\cO^{\psi_p}(\GSp_4(\cO_p))$ and for two objects $A,B$, their morphisms are given by
\begin{align*}
    \Hom(A,B) \defeq \Hom_{\GSp_4(F_p)} (\cind_{\GSp_4(\cO_p)}^{\GSp_4(F_p)} A , \cind_{\GSp_4(\cO_p)}^{\GSp_4(F_p)} B). 
\end{align*}
Note that there is a natural faithful functor $\Rep_\cO^{\psi_p}(\GSp_4(\cO_p)) \ra \Rep_\cO^{\mathrm{Hk}, \psi_p}(\GSp_4(\cO_p))$ which is given by the identity on objects and compact induction on morphisms.

\begin{defn}
    A \textit{(fixed similitude) patching functor} for the pair $(\rhobar,\psip)$ (and $R_{\infty}^{\psi_p}$) is a covariant functor
    \begin{align*}
    M_\infty: \Rep^{\Hk,\psi_p}_\cO(\GSp_4(\cO_p)) \ra \Mod(R_\infty^{\psi_p})
\end{align*}
which is exact when restricted to $\Rep_\cO^{\psi_p}(\GSp_4(\cO_p))$ and 
satisfies the following conditions: let $(\lam+\eta,\tau)$ be a type with $\sigma^\circ(\lam,\tau) \subset \sig(\lam,\tau)$ a $\GSp_4(\cO_p)$-stable $\cO$-lattice, then
\begin{enumerate}[(1)]
    \item $M_\infty(\sigma^\circ(\lam,\tau))$ is a maximal Cohen--Macaulay module over  $R_\infty^{\lam+\eta,\tau,\psip}$ (if it is nonzero); 
    \item for all $\sigma \in \JH(\ov{\sigma}(\lam,\tau))$, $M_\infty(\sigma)$ is a maximal Cohen--Macaulay module over $R_\infty^{\lam+\eta,\tau,\psip}/\varpi$ (if it is nonzero); and
    \item for $\tau$ as in \S\ref{subsub:hecke-alg}, the natural action of $\cH(\sig^\circ(\tau))$ on $M_\infty(\sig^\circ(\tau))$ equals the action induced by $\Psi^\tau$ in Theorem \ref{thm:functions} and the $R^{\eta,\tau,\psip}_{\infty}$-action.
\end{enumerate}
We also say that $M\in \Mod(R_{\infty}^{\psip})$ is a \textit{patched module for $(\lam+\eta,\tau)$} if it is a maximal Cohen--Macaulay module over $R_{\infty}^{\lam+\eta,\tau,\psip}$. In particular, given a patching functor $M_\infty$, $M_\infty(\sigma^\circ(\lam,\tau))$ is an example of patched modules.
\end{defn}

\begin{rmk}\label{rmk:patching-functor-cong}
    For our applications, we do need to choose $\psip$ compatible with Hodge type (and tame inertial type) we are working with (cf.~\cite[Definition 5.2.1(2)]{lee_thesis}). In \S\ref{sub:main}, we choose $\psip$ lifting $\simc(\rhobar)$ such that $\psi_v\eps^{-3}$ is tamely ramified with finite image for $v\in S_p$. This is compatible with the Hodge type $\eta$. In \S\ref{sub:mlt}, we choose $\psip$ lifting $\simc(\rhobar)$ such that $\psi_v\eps^{-d}$ is tamely ramified with finite image for $v\in S_p$, where $d$ is the integer giving the Hodge type $\simc(\lam+\eta)$ at all embeddings of $\cO_p$ into $\cO$. Even for different choices of $\psip$, the different patching functors can be identified with restricted to $\Rep^{\Hk,\psi_p}_\F(\GSp_4(\cO_p))$ (see \cite[Theorem 5.4.2]{lee_thesis}).  
\end{rmk}

\begin{rmk}\label{rmk:patching-functor-vs-module}
    As mentioned above, the definition of patching functor subsumes patched modules. In the next subsection, we construct a patching functor using \'etale cohomology of Hilbert--Siegel modular varieties. However, it is not clear how to construct a patching functor using coherent cohomology. Still, we can and do construct patched modules in the coherent setting, which is sufficient for our application.
\end{rmk}

\begin{rmk}\label{rmk:irred-comp}
    By \cite[Proposition 2.2]{KW_Serre_2}, the set of irreducible components of $\Spec R_{\infty}^{\lam+\eta,\tau,\psip}$ is given by the product of the sets of irreducible components of $\Spec R_{\rhobar_v}^{\lam_v+\eta_v,\tau_v,\psi_v}$ for $v\in S_p$ and $\Spec R^p$. In particular, for any tame inertial type $\tau$, the maximal Cohen--Macaulayness implies that $\supp_{R_{\rhobar}^{\lam+\eta,\tau,\psip}}M_\infty(\sig^\circ(\lam,\tau))$ is a union of irreducible components in $\Spec R_{\rhobar}^{\lam+\eta,\tau,\psip}$.
\end{rmk}

\begin{defn}
    Let $M_{\infty}$ be a fixed similitude patching functor for $(\rhobar,\psi_p)$. 
    \begin{enumerate}
        \item  We define $W_{M_\infty}(\rhobar)$ be the set of Serre weights $\sig$ for $\rG$ such that $M_\infty(\sig) \neq 0$. 
        \item We define $W_{M_\infty}^{\supp}(\rhobar)$ be the set of 3-deep Serre weights $\sig$ for $\rG$ such that $M_\infty(\sig) \neq 0$ and $\cC_{\sig}(\rhobar) \subset \supp_{R_{\rhobar}^{\square}}M_\infty(\sig)$.
    \end{enumerate}
   
\end{defn}

The following is our main result on Serre weight conjectures in an axiomatic setup.

\begin{thm}\label{thm:abstract-SWC}
    Suppose that $\rhobar$ is 9-generic. If $M_\infty$ is a fixed similitude patching functor for $(\rhobar,\psip)$, then the following are equivalent:
    \begin{enumerate}
        \item $W_{M_\infty}(\rhobar) \neq \emptyset$;
        \item $W_{M_\infty}^{\supp}(\rhobar)\neq \emptyset$;
        \item $W^?(\rhobar|_{I_{\Q_p}}) = W_{M_\infty}(\rhobar) = W_{M_\infty}^{\supp}(\rhobar)$.
    \end{enumerate}
\end{thm}

We prove the Theorem after several preliminary results. We first prove the following weight elimination result, relaxing the genericity condition in \cite[Theorem 5.3.5]{lee_thesis}.

\begin{prop}\label{prop:weight-elim}
     Suppose that $\rhobar$ is $9$-generic. If $M_\infty$ is a fixed similitude patching functor for $(\rhobar,\psip)$, then $W_{M_\infty}(\rhobar)\subset W^?(\rhobar|_{I_{\Q_p}})$.
\end{prop}

\begin{proof}
    We explain how to adapt the proof of \cite[Theorem 6.1]{DLR} to our case.

    Suppose that $F(\lam)\in W_{M_\infty}(\rhobar)$. Then $\rhobar$ admits a potentially semistable lift of type $(\eta,\tau(1,\lam))$. Then the proof of \cite[Theorem 8]{Enns} shows that $\lam$ is $0$-deep. 
    Let $\lam_0\in \uC_0$ such that $\lam=\tilw_{\lam}\cdot \lam_0$ for some $\tilw_{\lam}\in \tilW^+_1$. 
    If $\lam_0$ is $2h_\eta$-deep, then there exists a Deligne--Lusztig representation $R$ such that $F(\lam) \in \JH({\ov{R}})$, and moreover if $\lam_0$ is \textit{not} $m$-deep, then $R$ is not $(m+1)$-generic \cite{general-we}. 
    Then the remainder is the same as the proof of \cite[Theorem 6.1]{DLR}. In particular, it shows that $\lam_0$ is $(h_{\tilw_h\tilw_{\lam}(0)} + 1)$-deep using the fact that $\tilw(R)^\mo \tilw(\rhobar) \in \Adm(\eta)$ (with appropriate lowest alcove presentations of $\rhobar$ and $R$). 
     By Theorem 5.4(1) in \loccit, $F(\lam)\in \JH(\ov{R}_w(\lam_0+\eta-w\pi(\tilw_h\tilw_{\lam})^\mo(0)))$ for all $w\in \uW$. Then the argument in the last paragraph of the proof of Theorem 6.1 in \loccit~shows that $F(\lam)\in W^?(\rhobar|_{I_{\Qp}})$.
\end{proof}

\begin{defn}[cf.~{\cite[Definition 2.6.3]{LL-CM}}]
    Let $\rhobar\in \cX_{\Sym}(\F)$ be a 9-generic tamely ramified representation. For $\sig_1, \sig_2 \in W^?(\rhobar|_{I_{\Qp}})$, we say that $\sig_1$ and $\sig_2$ are \textit{adjacent} if there exists $(\tilw_1,\tilw)\in \AP'(\eta)$ and a simple reflection $s\in \uW$ as in Setup \ref{setup} such that $\{\sig_1,\sig_2\} =W^?(\rhobar_0)\cap \JH(\osig(\tau))$.
\end{defn}

\begin{prop}\label{prop:connected}
    Let $\rhobar\in \cX_{\Sym}(\F)$ be a 9-generic tamely ramified representation. Consider a graph whose set of vertices is $W^?(\rhobar|_{I_{\Qp}})$ and any two vertices are connected by an edge if and only if they are adjacent. Then the graph is connected.
\end{prop}
\begin{proof}
    By choosing $({w^\diamond},{w^\diamond})\in \AP'(\eta)$ for $w\in \uW$ and a simple reflection $s\in \uW$, we can show that $F_{\rhobar}(w)$ and $F_{\rhobar}(sw)$ are adjacent. By Remark \ref{rmk:obv-wt-W-torsor}, the subgraph whose vertices given by $W_{\obv}(\rhobar|_{I_{\Qp}})$ with all edges connecting them is connected. 

    To finish the proof, we need to show that for each $\sig_1 = F_{\rhobar}(\tilw_1,\tilw) \in W^?(\rhobar|_{I_{\Qp}})\setminus W_{\obv}(\rhobar|_{I_{\Qp}})$, there is a sequence $\sig_2,\dots, \sig_r$ in $W^?(\rhobar|_{I_{\Qp}})$ such that $\sig_i$ is adjacent to $\sig_{i+1}$ for each $i=1,\dots,r-1$ and $\sig_r\in W_{\obv}(\rhobar|_{I_{\Qp}})$. The last claim will be justified by taking $\sig_r$ of the form $F_{\rhobar}(\tilw_1',\tilw')$ where $\tilw'\in \uOm$. Since $\tilw_1' \uparrow \tilw'$, this means $\tilw_1' = \tilw'$ and thus $\sig_r\in W_{\obv}(\rhobar|_{I_{\Qp}})$.

    Given $\sig_1=F_{\rhobar}(\tilw_1,\tilw)$, a choice of a simple reflection $s\in \uW$ provides $\sig_2$ as in Setup \ref{setup}. We can write $s= s_{1,j}$ or $s_{2,j}$ for some $j\in \cJ$. It follows from the definition of $\sig_2$ that $\sig_2 = F_{\rhobar}(\tilw_1',\tilw')$ for some $(\tilw_1',\tilw')\in \AP'(\eta)$ such that $\tilw'_{j'}(A_0) = \tilw_{j'}(A_0)$ for all $j' \neq j$ and $\tilw'_j(A_0)$ is determined by the choice of $s$ and $\tilw_j(A_0)$. Thus, we can treat each $j\in \cJ$ separately. 
    
    At $j'=j$, we have $\tilw'_{j}\in \Omega$ if either $\tilw_j(A_0)= A_2$ and $s=s_{1,j}$ or $\tilw_j(A_0)=A_3$ and $s=s_{2,j}$. Thus, we are done if $\tilw_j(A_0)\in \{A_2,A_3\}$ except when $\tilw_{1,j}\in \Omega$, in which case we cannot take $s=s_{2,j}$.

    If $\tilw_{1,j}\in \Omega$, $\tilw_j(A_0)=A_3$ and $s=s_{1,j}$, a direct computation shows that $\sig_2 = F_{\rhobar}(\tilw_1',\tilw')$ for some $(\tilw_1',\tilw')\in \AP'(\eta)$ such that $\tilw_{1,j}' \in \Omega$ and $\tilw_j'(A_0)=A_1$. If $\tilw_{1,j}\in \Omega$, $\tilw_j(A_0)=A_1$, and $s=s_{1,j}$, then $\sig_2 = F_{\rhobar}(\tilw_1',\tilw')$ for some $(\tilw_1',\tilw')\in \AP'(\eta)$ such that $\tilw_{1,j}'(A_0)=A_2$ and $\tilw_j'(A_0)=A_3$. 
    Since such $\sig_2$ is adjacent to $\sig_3$ corresponding to $(\tilw_1'',\tilw'')$ with $\tilw''_j\in \Omega$ by the previous paragraph, this finishes the proof.
\end{proof}

\begin{lemma}\label{lem:BMcycle-bound}
  Suppose that $\rhobar$ is 9-generic. Let $M_\infty$ be a fixed similitude patching functor for $(\rhobar,\psip)$. For any $\sig\in W^?(\rhobar)$, $\supp_{R_{\rhobar}^{\square}} M_\infty(\sig) \subset \cup_{\kappa \uparrow \sig} \Spec \cC_{\kappa}(\rhobar)$.
\end{lemma}
\begin{proof}
    By the exactness of $M_\infty$, we have
    \begin{align*}
        \supp_{R_{\rhobar}^{\square}} M_\infty(\sig) \subset \cap_{\tau}  \supp_{R_{\rhobar}^{\square}} M_\infty(\osig(\tau)) \subset  \cap_{\tau}  \Spec R_{\rhobar}^{\eta,\tau}/\varpi. 
    \end{align*}
    where the intersections run over all $3$-generic tame inertial type $\tau$ such that $\sig \in \JH(\osig(\tau))$. By \cite[Proposition 2.5.2]{LL-CM}, $\sig$ is $(5+d_{\sig})$-deep (see \S2.3 for the definition of $d_{\sig}$). By Lemma 2.4.1 in \loccit, $\tau$ is necessarily 5-generic. Then 
    $ \cap_{\tau}  \Spec R_{\rhobar}^{\eta,\tau}/\varpi$ is contained in $ \cup_{\kappa \uparrow \sig} \Spec \cC_{\kappa}(\rhobar)$ by Proposition \ref{prop:C_sig-T-torsor} and Lemma \ref{lem:covering-uparrow}.
\end{proof}

\begin{lemma}\label{lem:modularity-argument}
    Suppose that $\rhobar$ is 9-generic. If $\sig \in W^?(\rhobar|_{I_{\Qp}})\cap \JH_\out(\osig(\tau))$ and $\cC_{\sig}(\rhobar)\subset \supp_{R_{\rhobar}^{\square}}  M_\infty(\osig(\tau))$, then $M_\infty(\sig)\neq 0$.
\end{lemma}
\begin{proof}
    By Proposition \ref{prop:weight-elim}, $\cC_{\sig}(\rhobar)\subset \supp_{R_{\rhobar}^{\square}}  M_\infty(\osig(\tau))$ implies that there is $\kappa \in  W^?(\rhobar|_{I_{\Qp}})\cap \JH(\osig(\tau))$ such that $\supp_{R_{\rhobar}^{\square}} M_\infty(\kappa)$ contains $\cC_{\sig}(\rhobar)$. Note that $\sig$ and $\kappa$ are necessarily 6-deep (Lemma \ref{lem:JH-bij}). By Lemma \ref{lem:BMcycle-bound}, $\sig \uparrow \kappa$. Since $\sig$ is an outer weight, this implies $\sig=\kappa$ by \cite[Proposition 2.4.9]{LL-CM}.
\end{proof}

\begin{prop}\label{prop:modularity-constancy}
    Suppose that $\rhobar$ is 9-generic and $\sig_1, \sig_2\in W^?(\rhobar|_{I_{\Qp}})$ are adjacent. Then $\sig_1 \in W_{M_\infty}^{\supp}(\rhobar)$ if and only if $\sig_2 \in W_{M_\infty}^{\supp}(\rhobar)$.
\end{prop}

\begin{proof}
    Since $\sig_1$ and $\sig_2$ are adjacent, we can choose $(\tilw_1,\tilw)\in \AP'(\eta)$ and $s\in \uW$ with the resulting $\tau$ as in Setup \ref{setup}. Suppose that $\sig_1 \in W^{\supp}_{M_\infty}(\rhobar)$. Since $\sig_1 \in \JH(\osig(\tau))$, we have $M_\infty(\osig(\tau))\neq 0$. Moreover, by Theorem \ref{thm:irred-comp-def-ring}, there is a unique irreducible component $\cC\subset \Spec R_{\rhobar}^{\eta,\tau}$ containing $\cC_{\sig_i}(\rhobar)$ for $i\in \{1,2\}$. Thus, we have
    \begin{align*}
       \cC_{\sig_i}(\rhobar)\subset  \cC \subset \supp_{R_{\rhobar}^{\square}}M_\infty(\sig^\circ(\tau)).
    \end{align*}
    Then $\sig_2 \in W^{\supp}_{M_\infty}(\rhobar)$ by Lemma \ref{lem:modularity-argument}.
\end{proof}

\begin{proof}[Proof of Theorem \ref{thm:abstract-SWC}]
We show that (3) implies (2). For all $w\in \uW$, we have $M_\infty(F_{\rhobar}(w) ) \neq 0$. Let $\tau$ be a tame inertial type such that $\tilw(\rhobar,\tau) = ({w^{\diamond}})^\mo \tilw_h^\mo w_0 w^\diamond$. By \cite[Corollary 2.6.5]{MLM}, we have $W^?(\rhobar|_{I_{\Qp}})\cap \JH(\osig(\tau)) = \{F_{\rhobar}(w)\}$ (note that $F_{\rhobar}(w) = F_{\tau}(w)$) and thus
\begin{align*}
    M_\infty(F_{\rhobar}(w) )  = M_\infty(\osig(\tau))
\end{align*}
whose support over $R_{\rhobar}^{\square}$ is precisely $\Spec R_{\rhobar}^{\eta,\tau}/\varpi = \cC_{F_{\rhobar}(w)}(\rhobar)$ by Theorem \ref{thm:col-one-def}. Obviously, (2) implies (1).

We show that (1) implies (2). For the sake of contradiction, assume that $W_{M_\infty}^{\supp}(\rhobar)=\emptyset$. We can and do choose $\sig\in W_{M_\infty}(\rhobar)$ such that there is no $\sig\neq \kappa \in W_{M_\infty}(\rhobar)$ satisfying $\kappa \uparrow \sig$. By the assumption, $\cC_{\sig}(\rhobar)\not\subset\supp_{R_{\rhobar}^{\square}}M_\infty(\sig)$. By Lemma \ref{lem:BMcycle-bound}, there is a Serre weight $\kappa\neq \sig$ such that $\kappa \uparrow \sig$ and $\cC_{\kappa}(\rhobar) \subset \supp_{R_{\rhobar}^{\square}}M_\infty(\sig)$. Let $\sig^\circ(\tau)$ be the principal series lifting $\sig$ as in \S\ref{subsub:hecke-alg}. Then $M_\infty(\sig)$ is a module over $R_{\rhobar}^{\eta,\tau}$. For $i=1,2$, since $\tld{F}_i|_{\cC_{\kappa}(\rhobar)}=0$ by Corollary \ref{cor:vanishing}, we have $\cC_{\kappa}(\rhobar) \subset \supp_{R_{\rhobar}^{\square}}M_\infty(\sig)/\tld{F}_{i}$. By Theorem \ref{thm:functions} and Proposition \ref{prop:hecke-mod-p-red}, we have $M_\infty(\sig)/\tld{F}_{i} = M_\infty(\sig)/\ov{T}_{-\be_i}$. By applying $M_\infty$ to the objects in the diagram in Proposition \ref{prop:weight-cyc} without compact induction, we get the induced map
\begin{align*}
    \cyc_i: M_\infty(C_i(\sig)) \onto M_\infty(\sig)/\ov{T}_{-\be_i}.
\end{align*}
By Lemma \ref{lem:JH-bound}, there exists $\kappa'\in \JH(\osig(\tau))$ such that $\cC_{\kappa}(\rhobar)\subset \supp_{R_{\rhobar}^{\square}} M_\infty(\kappa')$, and by Lemma \ref{lem:BMcycle-bound}, $\kappa \uparrow \kappa'$. This also implies $\kappa' \uparrow \sig$ and contradicts our assumption. 
Finally, (2) implies (3) by Propositions \ref{prop:connected} and \ref{prop:modularity-constancy}.
\end{proof}

\subsection{Patching construction}\label{sub:patching-construction}
We provide a construction of a patching functor using \'etale cohomology and patched modules using coherent cohomology. The former construction essentially follows \cite[\S5.4]{lee_thesis} with minor modifications. For the latter, we follow the strategy of \cite{AH} where certain unitary Shimura varieties were considered. 

Let $F$, $\rbar$ and $\chi$ be as in \S\ref{sec:global}. We also denote by $\rbar_p:G_{\Qp} \ra \prescript{L}{}{\underline{G}}(\F)$ the tame $L$-parameter given by $(\rbar|_{G_{F_v}})_{v\in S_p}$. Similarly, we have $\psi_p$ given by $\psi_v\defeq\chi \eps_p^{3}|_{G_{F_v}}$ for $v\mid p$. We take $\cO_p$ to be $\cO_F\otimes \Z_p$ and $U\le \GSp_4(\A^{\infty}_F)$ to be a compact open subgroup such that $U_p=\GSp_4(\cO_p)$.

Let $S$ be a set of places containing $S_p\cup P_U \cup P_{\rbar}$. Let $R\subset S$ be a subset of places $v$ such that $U_v=\Iw_1(v)$ is the pro-$v$ Iwahori subgroup, $q_v \equiv 1 \mod p$, and $\rbar|_{G_{F_v}}$ is trivial. For each $v\in R$, we choose a pair of characters of $\cO_v^\times$ $\zeta_v=(\zeta_{v,1},\zeta_{v,2})$ and define $\zeta_R$ as in \cite[\S5.3]{lee_thesis}. We denote the image of $\T^{S,\univ}$ in $\End_{\cO}(H^3_{\et}(\Sh_{U,\overline{F}},\cF_V)_{\chi,\zeta_R})$ by $\T^S_{\chi}(U,V)_{\zeta_R}$, where the subscript $\zeta_R$ denotes taking the $\zeta_R$-coinvariant. We adapt a similar notation for coherent cohomology and its Hecke algebra.

We define $q\defeq h^1(F_S/F,\ad^0(\rbar)(1))$ (cf.~\cite[Proposition 4.4.9]{EL}). For a subset $T\subset S$, define $\cT_T:=\cO\DB{y_1,\dots, y_{11\abs{T}-1}}$. Let $S_\infty\defeq\cT_S\DB{x_1,\dots, x_{2q}}$ viewed as an augmented $\cO$-algebra and $\mathfrak{a}_\infty$ be its augmentation ideal. We choose $R^p:= (\ctimes_{v\in S\backslash S_p}R_{\rbar|_{G_{F_v}}}^{\square,\psi|_{G_{F_v}}})\DB{X_1,\dots, X_g}$ where $g:=2q-4[F:\Q]+\abs{S}-1$.

The following theorem provides a patching functor in the \'etale setting.

\begin{thm}\label{thm:patching-exist}
    Suppose that $\rbar:G_F \ra \GSp_4(\F)$ is automorphic of level $U$ and satisfies the Taylor--Wiles conditions. Then, there exists a fixed similitude patching functor $M_\infty$ for $(\rbar_p,\psi_p)$ together with an $\cO$-algebra morphism $S_\infty \ra R_{\infty}^{\psip}$ such that for $V\in \Rep_{\cO}^{\Hk,\psip}(\GSp_4(\cO_p))$, 
    \begin{align*}
        M_\infty(V)/\mathfrak{a}_\infty = H^3_{\et}(\Sh_{U,\overline{F}},\cF_{V^\vee})_{\chi,\zeta_R,\fm_{\rbar,\chi}^S}^\vee.
    \end{align*}
    In particular, for a Serre weight $\sig$ of $\GSp_4(\cO_p/p)$, $\sig \in W(\rbar)$ if and only if $M_\infty(\sig)\neq 0$.
\end{thm}

\begin{proof}
    We construct the $R_\infty^{\psi_p}$-module $M_\infty$ (from which the patching functor can be defined) by repeating the argument in \cite[\S5.4.4]{lee_thesis}. We refer \loccit~for undefined notation. 
    Once we have $M_\infty$, we define the functor
    \begin{align*}
        M_\infty(V) \defeq \Hom_{\cO\DB{\GSp_4(\cO_p)}}^{\cont}(M_{\infty},V^\vee)^\vee
    \end{align*}
    for $V\in \Rep^{\psip}_\cO(\GSp_4(\cO_p))$. Note that the $\GSp_4(\cO_p)$-action on the module $M_\infty$ extends to $\GSp_4(F_p)$ as explained in \cite[Lemma 4.3.9]{EL}. Let $V=V_{f}\oplus V_t$ the decomposition into $\cO$-free part and $\varpi$-torsion part. We have
    \begin{align*}
        M_\infty(V_f) \simeq \Hom_{\cO\DB{\GSp_4(\cO_p)}}^{\cont}(M_{\infty},V_f^d)^d
    \end{align*}
    where $d$ denotes the Schikhof dual, and
    \begin{align*}
        \Hom_{\cO\DB{\GSp_4(\cO_p)}}^{\cont}(M_{\infty},V_?^\star) &= \Hom_{\cO[\GSp_4(\cO_p)]}(V_?, M_\infty^\star) 
        \\
        &= \Hom_{\GSp_4(F_p)}(\cind_{\GSp_4(\cO_p)}^{\GSp_4(F_p)}V_?, M_\infty^\star) 
    \end{align*}
    for $(?,\star)\in \{(f,d),(t,\vee)\}$ (cf.~\cite[pg.~61--62]{CEGGPS}). 
    This allows us to define $M_\infty$ on the Hecke category $\Rep_{\cO}^{\Hk,\psip}(\GSp_4(\cO_p))$. 
    The property (3) follows from Theorem \ref{thm:functions} similar to the proof of \cite[Lemma 4.17]{CEGGPS}.
    
    For the patching datum, we make the following changes. For $n\ge 1$, recall that we have a set of Taylor--Wiles primes $Q_n$. For each $v\in Q_n$, define $U_{\Delta,v}$ to be the kernel of the composition $\cI(v) \onto \cI(v)/Z_{\GSp_4(\cO_K)}\cI_1(v) \onto (T(k_v)/Z_{\GSp_4(k_v)})(p)$ where $Z_{*}$ denotes the center of $*$ and $(*)(p)$ denotes the maximal pro-$p$ quotient of $*$. We also denote by $\Delta_v \defeq \cI(v)/U_{\Delta,_v}$ and $\Delta_{Q_n}\defeq \prod_{v\in Q_n} \Delta_v$. We define compact open subgroups $U^p_\Delta(Q_n)\le U_0^p(Q_n) \le \GSp_4(\A^{\infty,p}_F)$ by setting
    \begin{enumerate}
        \item if $v\notin Q_n\cup R$ and $v\nmid p,\infty$, $U_{\Delta}^p(Q_n)_v = U_0^p(Q_n)_v = U^p_v$;
        \item if $v\in Q_n$,  $U_{\Delta}^p(Q_n)_v = U_{\Delta,v}$ and $U_0^p(Q_n)_v = \cI(v)$;
        \item if $v\in R$, $U_{\Delta}^p(Q_n)_v = \cI_1(v)$ and $U_0^p(Q_n)_v = \cI(v)$.  
    \end{enumerate}
    
    For a compact open subgroup $H\le \GSp_4(\cO_p)$, $n,r\ge 1$, we choose
    \begin{itemize}
        \item $M_r^{\psi_p}(H)_0\defeq H^{3d}_{\et}(\Sh_{H\cdot U^p,\overline{F}},\cO/\varpi^r)_{\chi,\zeta_R,\fm_{\rbar,\chi}^S}^\vee$
        \item for $n\ge 1$, we take 
        \begin{align*}
            M_r^{\psi_p}(H)_n&\defeq H^{3d}_{\et}(\Sh_{H\cdot U_{\Delta}^p(Q_n),\overline{F}},\cO/\varpi^r)_{\chi,\zeta_R,\fm_{\rbar,\chi}^{S\cup Q_n},\fm_{Q_n}}^\vee\otimes_{ R_{\mathcal{S}_{Q_n}}} R_{n}^{\psi_p} \\
            \al_n^{\psi_p}&: M_r^{\psi_p}(H)_n/\mathfrak{a}_\infty \simeq M_r^{\psi_p}(H)_0.
        \end{align*}
    \end{itemize}
    By Proposition 5.3.2 in \loccit\footnote{Note that we use \'etale cohomology of Shimura variety as the source of automorphic forms while \loccit~uses algebraic automorphic forms on an inner form of $\GSp_4$. Still, the same statement holds in our setting. Unlike in \loccit, we do not need to use the Jacquet--Langlands transfer, and we apply Matsushima's formula to relate \'etale cohomology groups to spaces of automorphic forms on $\GSp_4(\A_F)$ instead of \cite[(4.2.2)]{EL}.}, there is a surjective morphism $R_{\mathcal{S}_{Q_n}} \onto \T^S_{\chi}(H\cdot U^p_\Delta(Q_n),\cO)_{\zeta_R,\fm_{\rbar,\chi}^P}$. This induces  a $R_{\mathcal{S}_{Q_n}}$-module structure on $H^{3d}_{\et}(\Sh_{H\cdot U_1^p(Q_n),\overline{F}},\cO/\varpi^r)_{\chi,\zeta_R, \fm_{\rbar,\chi}^{S\cup Q_n},\fm_{Q_n}}$.
    
    We briefly explain why the above forms a $G_p\defeq\prod_{v\in S_p}\GSp_4(F_v)$-patching datum in the sense of \cite[Definition 4.3.5]{EL}. Namely, we show that $H^{3d}_{\et}(\Sh_{H'\cdot U_{\Delta}^p(Q_n),\overline{F}},\cO/\varpi^r)_{\chi,\zeta_R,\fm_{\rbar,\chi}^{S\cup Q_n}}$ is free over $\cO/\varpi^r[\Delta_{Q_n}]$ and is also projective over $\cO/\varpi^r[H/H']$ for a normal subgroup $H'\lhd H$. Note that $\Sh_{H'\cdot U_{\Delta}^p(Q_n)} \ra \Sh_{H\cdot U_0^p(Q_n)}$ is an \'etale Galois cover with Galois group $\Gamma\defeq \Delta_{Q_n}\times H/H' \times \prod_{v\in R}T(k_v)$. By \cite[Proposition 1.2]{crew-et_p_cover}, the complex computing compactly supported \'etale cohomology of $\Sh_{H'\cdot U_\Delta^p(Q_n)}$ with coefficients in $\cO/\varpi^r$ is a perfect complex of $\cO/\varpi^r[\Gamma]$-modules. Once we localize the complex at $\fm^{S\cup Q_n}_{\rbar,\chi}$, it becomes isomorphic to $H^{3d}_{\et}(\Sh_{H'\cdot U_{\Delta}^p(Q_n),\overline{F}},\cO/\varpi^r)_{\chi,\fm_{\rbar,\chi}^{S\cup Q_n}}$ by Proposition \ref{prop:torsion-vanishing}, which is projective over $\cO/\varpi^r[\Gamma]$. Then $H^{3d}_{\et}(\Sh_{H'\cdot U_{\Delta}^p(Q_n),\overline{F}},\cO/\varpi^r)_{\chi,\zeta_R,\fm_{\rbar,\chi}^{S\cup Q_n}}$ is also free over $\cO/\varpi^r[\Delta_{Q_n}]$ since $\cO/\varpi^r[\Delta_{Q_n}]$ is a local ring. Finally, localizing at $\fm_{Q_n}$ is equivalent to taking its direct summand (cf.~\cite[Corollary 2.4.35]{BCGP}).  
\end{proof}

Now we turn to the coherent setting as in \S\ref{sub:coh-coh}. In particular, we take $F=\Q$. 

\begin{thm}\label{thm:coh-pat-mod}
    Let $\lam\in X_1(T)$ be a 2-generic character. Then there exists a patched module $M_\infty^{\coh}(W(\lam))$ for $(\lam,\mathrm{triv})$ equipped with free $S_\infty$-module structure such that $M_\infty^{\coh}(W(\lam))/\mathfrak{a}_\infty = H^0(\Sh^{\tor}_{U_pU^p},\omega(\lam+(3,3;0)))_{\chi,\zeta_R,\fm^{S}_{\rbar,\chi}}^\vee$ as $R_{\infty}^{\lam+\eta,\psip}$-modules. (Here, $\mathrm{triv}$ is the trivial inertial type.) 
\end{thm}

\begin{rmk}\label{rmk:coh-patch-functorial}
    Although it is not clear how to construct a patching functor in the coherent setting, the coherent patched modules are (contravariant) functorial in the following sense. If there is a family of morphisms
    \begin{align*}
        H^0(\Sh^{\tor}_{U_pU^p,\F},\omega(\lam+(3,3;0)))_{\chi,\zeta_R,\fm^{S}_{\rbar,\chi}} \ra H^0(\Sh^{\tor}_{U_pU^p,\F},\omega(\lam'+(3,3;0)))_{\chi,\zeta_R,\fm^{S}_{\rbar,\chi}}
    \end{align*}
    for all $U$ that are compatible with level-raising inclusions, then this induces a map between patched modules
    \begin{align*}
        M_\infty^{\coh}(W(\lam'))/\varpi \ra M_\infty^{\coh}(W(\lam))/\varpi.
    \end{align*}
    In particular, 
    the injective map $\theta^1_{\lam}$ in Theorem \ref{thm:ortiz} induces a surjection between corresponding patched modules.
\end{rmk}

\begin{proof}
    For a given $\lam$ and integers $n,r\ge 1$, we choose
    \begin{itemize}
        \item $M_{r,0}^{\psi_p}\defeq H^0(\Sh_{U_p\cdot U^p,\cO/\varpi^r},\omega(\lam+(3,3;0)))_{\chi,\zeta_R,\fm_{\rbar,\chi}^S}^\vee$
        \item for $n\ge 1$, we take 
        \begin{align*}
            M_{r,n}^{\psi_p}&\defeq H^0(\Sh_{U_p U_{\Delta}^p(Q_n),\cO/\varpi^r},\omega(\lam+(3,3;0)))_{\chi,\zeta_R,\fm_{\rbar,\chi}^{S\cup Q_n},\fm_{Q_n}}^\vee\otimes_{ R_{\mathcal{S}_{Q_n}}} R_{n}^{\psi_p} \\
            \al_n^{\psi_p}&: M_{r,n}^{\psi_p}/\mathfrak{a}_\infty \simeq M_{r,0}^{\psi_p}.
        \end{align*}
    \end{itemize}
    Then $M_{r,n}^{\psi_p}$ is a free $\cO/\varpi^r[\Delta_{Q_n}]$-module following the proof of \cite[Corollary 33]{AH} and the mod $p$ vanishing result in our setting (Proposition \ref{prop:coh-vanishing}). Then the ultrapatching construction as in \cite[\S5.4.4]{lee_thesis} provides $M_\infty^{\coh}(W(\lam))$. Note that we do not need the second axiom in \cite[Definition 4.3.5]{EL} because the level at $p$ is fixed here. 
\end{proof}

As in \S\ref{sub:two-coh-comp}, \'etale and coherent patched modules can be compared. For our application, we only need the following result on their support at $p$.  

\begin{cor}\label{cor:patching-comp}
    Let $\lam\in X_1(T)$ be a 2-generic character. For the patching functor $M_\infty$ as in Theorem \ref{thm:patching-exist} and the patched module $M_\infty^{\coh}(W(\lam))$ as in Theorem \ref{thm:coh-pat-mod}, we have
    \begin{align*}
        \supp_{R_{\rbar_p}^{\lam+\eta,\psip}} M_\infty(W(\lam)^\vee) =  \supp_{R_{\rbar_p}^{\lam+\eta,\psip}} M_\infty^{\coh}(W(\lam)). 
    \end{align*}
\end{cor}
\begin{proof}
    Note that both patched modules are maximal Cohen--Macaulay over $R_{\infty}^{\lam+\eta,\psip}$. Thus, their supports are unions of irreducible components. As explained in Remark \ref{rmk:irred-comp}, an irreducible component $\cC_p \subset \Spec R_{\rbar_p}^{\lam+\eta,\psip}$ is in the support of above patched modules if and only if $\cC_p\times \cC^p$ is in the support over $R_{\infty}^{\lam+\eta,\psip}$ for some irreducible component $\cC^p \subset \Spec R^p$. Since $R_{\rbar_p}^{\lam+\eta,\psip}$ is $\cO$-flat with regular generic fiber (cf.~\cite[Theorem 3.3.8]{BellovinGee}), its irreducible components are disjoint in the generic fiber. Thus the latter condition can be checked after inverting $p$ and  reducing modulo $\mathfrak{a}_\infty$ by applying \cite[Lemma 3.9]{paskunas-2adic} where we take $A$ and $x_1,\dots, x_d$ in \loccit~to be $R_{\infty}^{\lam+\eta,\psip}$ and $p,y_1,\dots, y_{11\abs{T}-1}, x_1,\dots,x_{2q}\in S_\infty$ (which form a system of parameters in $R_\infty^{\lam+\eta,\psip}$). 
    Then Corollary \ref{cor:hecke-comp} implies the claim.%
\end{proof}

\subsection{The main result}\label{sub:main}

We prove Conjecture \ref{conj:SWC} for $\rbar$ under standard Taylor--Wiles conditions. 


\begin{thm}\label{thm:SWC}
    Suppose that $\rbar$ is automorphic, satisfies Taylor--Wiles conditions, and for each 
    $v|p$, $\rbar|_{G_{F_v}}$ is tame and $9$-generic. Then $W^?(\rbar_p|_{I_{\Q_p}}) = W(\rbar) = W_{M_\infty}^{\supp}(\rbar_p|_{I_{\Q_p}})$. 
\end{thm}

\begin{proof}
Since $\rbar$ is automorphic, $W(\rbar)\neq \emptyset$. Then this follows from Theorem \ref{thm:abstract-SWC}.
\end{proof}

\subsection{Modularity lifting}\label{sub:mlt}
In this subsection, we set $F=\Q$. We prove the following modularity lifting result for small Hodge--Tate weights and generic tame inertial type.

\begin{thm}\label{thm:MLT}
    Suppose that $\rbar$ is automorphic, satisfies Taylor--Wiles conditions, and $\rbar|_{G_{\Qp}}$ is tame and 9-generic. Let $\lam\in X^*(T)$ be a dominant character. If $r: G_\Q \ra \GSp_4(E)$ is a lift of $\rbar$ that is unramified at all but finitely many places and potentially crystalline of type $(\lam+\eta,\tau)$ at $p$ with $(2h_{\lam}+6)$-generic tame inertial type $\tau$, then $r$ is automorphic. 
\end{thm}

Let $\rhobar\defeq\rbar|_{G_{\Qp}}$. By the standard argument, we need to show that the patched module $M_\infty( \sig^\circ(\lam,\tau))$ is fully supported on $R_\infty^{\lam + \eta,\tau,\psip}$. As discussed in Remark \ref{rmk:irred-comp}, any irreducible component in $\Spec R_\infty^{\lam+\eta,\tau,\psip}$ is of the form $\cC^p \times \cC_p$ for some irreducible components $\cC^p \subset \Spec R^p$ and $\cC_p \subset \Spec R^{\lam+\eta,\tau,\psip}_{\rhobar}$. We first prove the full support on $\Spec R^p$.

\begin{lemma}\label{lem:full-supp-R^p}
   For a 1-generic tame inertial type $\tau$, there exists a reduced closed subscheme $X^{\lam,\tau} \subset \Spec R_{\rhobar}^{\lam+\eta,\tau}$ that is a union of irreducible components and satisfying $\supp_{R_\infty^{\lam+\eta,\tau}} M_\infty(\sig^\circ(\lam,\tau)) = \Spec R^p \times X^{\lam,\tau}$.
\end{lemma}
\begin{proof}
    Suppose that $\cC_0^p \times \cC_p \subset \supp_{R_\infty^{\lam+\eta,\tau}} M_\infty(\sig^\circ(\lam,\tau))$ for some irreducible components $\cC_0^p \subset \Spec R^p$ and $\cC_p \subset \Spec R_{\rhobar}^{\lam+\eta,\tau}$. Then, we claim that for any irreducible component $\cC^p \subset \Spec R^p$, $\cC^p\times \cC_p \subset \supp_{R_\infty^{\lam+\eta,\tau}} M_\infty(\sig^\circ(\lam,\tau))$. By the lifting result of Patrikis--Tang \cite[Theorem 3.4]{PT}, there is a geometric lift $r$ of $\rbar$ that is unramified outside $S$ and $r|_{G_{\Q_v}}$ for all $v\in S$ defines a closed point $x$ in $\cC^p \times \cC_p$.  By the base change and Ihara avoidance argument (cf.~\cite[Theorem 6.3.1]{lee_thesis}) and the assumption on $\cC_p$, $r$ is automorphic. (Note that \loccit~uses Theorem 4.4.3 in \loccit~to ensure that a certain local deformation ring at $p$ is domain, which was used to show the full support of a patched module. We do not have an analogous result for local deformation rings at $p$, but this is fine because we only consider a fixed component $\cC_p$ at this moment.)   By \cite[Theorem 2.7.1(4) and Lemma 7.1.3]{BCGP}, the point $x$ corresponding to $r|_{G_{\Q_v}}$ for all $v\in S\backslash \{p\}$ is a smooth point in $\cC^p$. Moreover, $\Spec R^{\eta,\tau}_{\rhobar}[1/p]$ is the disjoint union of its irreducible component by \cite[Theorem 3.3.8]{BellovinGee}. In particular, $\cC^p \times \cC_p$ is the unique irreducible component containing $x$. This proves our claim.
\end{proof}

Using the previous Lemma, we can factorize the cycles attached to the patched modules.

\begin{lemma}\label{lem:factorizable-Minfty}
    There exists an effective cycle $Z^p \in \Z[\Spec R^p]$, with specialization $\ov{Z}^p \in \Z[\Spec R^p/\varpi]$, whose support is $\Spec R^p$ satisfying the following properties:
    \begin{enumerate}
        \item For a 1-generic tame inertial type $\tau$, there exists a cycle $Z^{\lam,\tau} \in \Z[\Spec R_{\rhobar}^{\lam+\eta,\tau}]$ such that $Z(M_\infty(\sig^\circ(\lam,\tau))) = 4Z^p \times Z^{\lam,\tau}$. Moreover, we can and do choose $Z^p$ and $Z^{\lam,\tau}$ such that the multiplicity of irreducible components in $Z^{\lam,\tau}$ is either zero or one.
        \item For a Serre weight $\sig$, there exists a cycle $Z_{\sig} \in \Z[\Spec R_{\rhobar}^{\alg}]$ such that $Z(M_\infty(\sig)) = 4 \ov{Z}^p \times Z_\sig$.
    \end{enumerate}
\end{lemma}

\begin{proof}
    We prove the existence of $Z^p$ satisfying the first property. The second one follows from the first and by expressing $\sig$ as a linear combination of $\osig(\tau)$ in the Grothendieck group (\cite[Theorem 33]{serre-book}). Note that $M_\infty(\sig)=0$ unless $\sig$ is 6-deep by Proposition \ref{prop:weight-elim} and Lemma \ref{lem:JH-bij}.

    The claim is equivalent to the following: for $Z(M_\infty(\sig^\circ(\lam,\tau))) = \sum_{\cC}\mu(\cC)[\cC]$ where the sum runs over all irreducible component $\cC= \cC^p \times \cC_p$ in $\Spec R_\infty^{\lam+\eta,\tau}$, we can write $\mu(\cC)=4\mu(\cC^p)\mu(\cC_p)$ where $\mu(\cC^p)$ and $\mu(\cC_p)$ depend only on $\cC^p$ and $\cC_p$ and the latter is either zero or one. 
    Therefore, we can take $Z^p= \sum_{\cC^p}\mu(\cC^p)[\cC^p]$ and $Z^{\lam,\tau} = \sum_{\cC_p} \mu(\cC_p)[\cC_p]$. Note that the cycle $Z^p$ is fully supported on $\Spec R^p$ by Lemma \ref{lem:full-supp-R^p}.    
    
    The multiplicity $\mu(\cC)$ is given by the length of the module $M_\infty(\sig^\circ(\lam,\tau))$ localized at the prime ideal of $R_\infty^{\lam+\eta,\tau}$ corresponding to $\cC$. Since $M_\infty(\sig^\circ(\lam,\tau))[1/p]$ is a maximal Cohen--Macaulay module over $R_{\infty}^{\lam+\eta,\tau}[1/p]$, it is a projective module once restricted to the smooth locus. In turn, $\mu(\cC)$ can be computed by the dimension of the fiber of $M_\infty(\sig^\circ(\lam,\tau))[1/p]$ at a smooth point in $\cC[1/p]$. As in the proof of Lemma \ref{lem:full-supp-R^p}, we can choose a smooth closed point in $\cC[1/p]$ given by a geometric lift $r$ of $\rbar$ unramified outside $S$ (note that $\cC_p[1/p]$ is smooth by \cite[Theorem 3.3.8]{BellovinGee}).  In this case, the fiber of $M_\infty(\sig^\circ(\lam,\tau))[1/p]$ can be identified with the Hecke eigenspace inside the \'etale cohomology $H^3_{\et}(\Sh_{U,\overline{\Q}}, \cF_{\sig^\circ(\lam,\tau)^\vee})[1/p][\fm_{r}]$ where the Hecke eigensystem $\fm_r \subset \T^{S,\univ}[1/p]$ is determined by $r$. Following \cite[\S1]{weissauer}, we have
    \begin{align*}
        H^3_{\et}(\Sh_{U,\overline{\Q}}, \cF_{\sig^\circ(\lam,\tau)^\vee})[1/p][\fm_{r}] \otimes_{E}\Qpbar 
        = \oplus_{\pi^\infty} W_{\pi^\infty} \otimes \Hom_{\GSp_4(\Zp)}(\sig(\tau),\pi_p) \otimes (\pi^{\infty,p})^{U^p}
    \end{align*}
    where the sum runs over the finite part $\pi^\infty$ of cuspidal automorphic representations $\pi$ of $\GSp_4(\A_\Q)$ whose attached Galois representation is isomorphic to $r$ (viewed as $\Qpbar$-vector spaces via $\Qpbar\simeq \C$) and $W_{\pi^\infty}$ is a certain multiplicity space. If such a $\pi$ exists, then all such $\pi$ are exactly of the form $\pi=\otimes'\pi_v$ where $\pi_v$ is in the $L$-packet for $r|_{G_{\Q_v}}$ if $v\neq p$ and for $\mathrm{WD}(r|_{G_{\Q_v}})$ if $v=p$ by Arthur's multiplicity formula \cite[Theorem 2.9.3]{BCGP}. In this case, we define $\mu(\cC_p) = \sum_{\pi_p}\dim_{\Qpbar} \Hom_{\GSp_4(\Zp)}(\sig(\tau),\pi_p)$ where the sum runs over the $L$-packet for $\mathrm{WD}(r|_{G_{\Q_v}})$. It is equal to one because $\Hom_{\GSp_4(\Zp)}(\sig(\tau),\pi_p)$ is non-zero for a unique $\pi_p$ in which case is one-dimensional by the multiplicity one result for tame types (cf.~\cite[Theorem 2.4.1]{lee_thesis}). Otherwise, we define $\mu(\cC_p)=0$. We also define $\mu(\cC^p) = \sum_{\pi^{\infty,p}} \dim_{\Qpbar} (\pi^{\infty,p})^{U_p}$ where the sum runs over the product of local $L$-packets away from $\infty$ and $p$. From this, we can easily see that $\mu(\cC)=4\mu(\cC^p)\mu(\cC_p)$, 
    where the constant $4$ comes from $\dim_{\Qpbar}W_{\pi^\infty}=4$ (see  \cite[pg.~82]{weissauer}; $m^+(\Pi_f)$ and $m^-(\Pi_f)$ in \loccit~equal one by Arthur's multiplicity formula). 
\end{proof}

It remains to prove that $\ov{Z}^{\lam,\tau}$ is equal to the cycle $\cZ^{\lam,\tau}(\rhobar)$. By Lemma \ref{lem:factorizable-Minfty}, we have $\ov{Z}^{\lam,\tau} \le \cZ^{\lam,\tau}(\rhobar)$. Moreover, by the exactness of $M_\infty$, we have $\ov{Z}^{\lam,\tau} = \sum_{\sig \in \JH(\osig(\lam,\tau))} n_{\sig}(\lam,\tau) Z_{\sig}$. By Corollary \ref{cor:BM-cycle-def} it suffices to show that $Z_{\sig} \ge \cZ_{\sig}^{\BM}(\rhobar)$. 

Let $\sig = F(\lam) \in W(\rbar)$. We have $Z_{\sig} \ge \cC_{\sig}(\rhobar)$ by Theorem \ref{thm:SWC}. If $\lam \notin C_2$, $\cZ_{\sig}^{\BM}(\rhobar) = \cC_{\sig}(\rhobar)$, so we are done. Suppose that $\lam \in C_2$. For $i\in \{0,1,2\}$, let $\lam_i\in C_i$ be a character such that $\lam_0\uparrow \lam_1 \uparrow \lam_2=\lam$. Again by \loccit, we have $\cZ_{\sig}^{\BM} = \cC_{\sig}+ \cC_{F(\lam_0)}$. Thus, we only need to prove the following.

\begin{prop}\label{prop:strong-SWC}
    Following the above notation, $Z_{\sig} \ge \cC_{F(\lam_0)}(\rhobar)$.     
\end{prop}

Before proving the above proposition, we provide a preliminary result.

\begin{lemma}\label{lem:supp-C0-C1}
    Following the above notation, $Z_{F(\lam_i)} = \cC_{F(\lam_i)}(\rhobar)$ for $i=0,1$.
\end{lemma}

\begin{proof}
    Let $\tilw_1\in \Omega$ be the element such that $(\tilw_1,\tilw_1)\in \AP'(\eta)$ corresponds to $F(\lam_0)$ under the bijection in Lemma \ref{lem:JH-bij}(2) and $\tilw_2$ be the element such that $\tilw_1 \uparrow \tilw_2$ and $\tilw_2\cdot C_0 = C_1$. Then $F_{\rhobar}(\tilw_1,\tilw_2)=F(\lam_1)$. We have a tame inertial type $\tau$ such that $\tilw({\rhobar},\tau)^*=\tilw_2^\mo\tilw_h^\mo w_0 \tilw_1$ as in \S\ref{sec:reg-col-one}. Then $F(\lam_i)$ for $i=0,1$ is contained in $\JH(\osig(\tau))$ by Lemma \ref{lem:JH-bij}. By Theorem \ref{thm:col-one-def}, $R_{\rhobar}^{\eta,\tau}$ is a domain with a reduced special fiber and $\Spec R_{\rhobar}^{\eta,\tau}/\varpi = \cC_{F(\lam_0)}(\rhobar) \cup \cC_{F(\lam_1)}(\rhobar)$. Therefore, $\ov{Z}^{\tau}$ is equal to $\cC_{F(\lam_0)}(\rhobar) + \cC_{F(\lam_1)}(\rhobar)$. By the exactness of $M_\infty$, this is equal to the sum of $Z_{F(\lam_i)}$ for $i=0,1$. Since $Z_{F(\lam_i)}\ge \cC_{F(\lam_i)}(\rhobar)$ for $i=0,1$, these inequalities are indeed equalities.
\end{proof}

\begin{proof}[Proof of Proposition \ref{prop:strong-SWC}]
    If $F(\lam_0)\not\in W^?(\rhobar|_{I_{\Qp}})$, then $\cC_{F(\lam_0)}(\rhobar)=\emptyset$ and there is nothing to prove. Suppose that $F(\lam_0) \in W^?(\rhobar|_{I_{\Qp}})$. Note that this implies $F(\lam_i) \in W^?(\rhobar|_{I_{\Qp}})$ for $i\in\{1,2\}$. 

    Let $\lam_i'\defeq -w_0(\lam_i)$. A direct computation shows that $\lam_i'\in C_i$ and $\lam_0'\uparrow \lam_1' \uparrow \lam_2'$. Note that $F(\lam_i')^\vee = F(\lam_i)$. By Lemma \ref{lem:weyl-decomp}, we have
    \begin{align*}
         \supp_{R_{\rhobar}^{\square}} M_\infty(W(\lam_2')^\vee)/\varpi &= \supp_{R_{\rhobar}^{\square}} M_\infty(F(\lam_1))\cup \supp_{R_{\rhobar}^{\square}} M_\infty(F(\lam_2))
         \\
          \supp_{R_{\rhobar}^{\square}} M_\infty(W(\lam_1')^\vee)/\varpi &= \supp_{R_{\rhobar}^{\square}} M_\infty(F(\lam_0))\cup \supp_{R_{\rhobar}^{\square}} M_\infty(F(\lam_1)).
    \end{align*}
    In particular, $\cC_{F(\lam_0)}(\rhobar)$ is in the support of $M_\infty(W(\lam_1')^\vee)/\varpi$. We claim that it is also contained in the support of $M_\infty(W(\lam_2')^\vee)/\varpi$. Since the support of $M_\infty(F(\lam_1))$ does not contain $\cC_{F(\lam_0)}$ by Lemma \ref{lem:supp-C0-C1}, the claim implies that $M_\infty(\sig)$ is supported on $\cC_{F(\lam_0)}(\rhobar)$. 

    By Corollary \ref{cor:patching-comp}, the claim is equivalent to the support of $M_\infty^{\coh}(W(\lam_2'))$ containing $\cC_{F(\lam_0)}$.  Again by Corollary \ref{cor:patching-comp}, the support of $M_\infty^{\coh}(W(\lam_1'))/\varpi$ contains $\cC_{F(\lam_0)}(\rhobar)$. As in Remark \ref{rmk:coh-patch-functorial}, the injective morphism $\theta_{\lam_1}^1$ in Theorem \ref{thm:ortiz} induces a surjective map
    \begin{align*}
        \theta_{\lam_1,\infty}^1: M_\infty^{\coh}(W(\lam_2'))/\varpi \onto M_\infty^{\coh}(W(\lam_1'))/\varpi.
    \end{align*}
    This implies the claim.
\end{proof}

\begin{proof}[Proof of Theorem \ref{thm:MLT}]
By Proposition \ref{prop:strong-SWC}, we have $Z_{\sig} = \cZ_{\sig}^{\BM}(\rhobar)$. This implies that $Z^{\lam,\tau} = \cZ^{\lam,\tau}(\rhobar)$. Therefore, $\supp_{R_\infty^{\eta,\tau}}(M_\infty(\sig^\circ(\lam,\tau))) = \Spec R_\infty^{\lam+\eta,\tau}$. This proves the theorem.
\end{proof}